\documentclass[10pt,a4paper]{article}
\usepackage[utf8]{inputenc}
\usepackage{amsmath}
\usepackage{amsfonts}
\usepackage{amssymb}
\usepackage{amsthm}
\usepackage{enumitem}
\usepackage{color}
\usepackage{mathtools}
\usepackage{tikz-cd}

\newcommand{\set}[1]{\left\{ #1 \right\}}
\newcommand{\n}[1]{\left\| #1 \right\|}
\newcommand{\p}[1]{\left( #1 \right)}
\newcommand{\va}[1]{\left| #1  \right|}
\newcommand{\brac}[1]{\left\langle #1 \right\rangle}
\newcommand{\FBI}{\textup{FBI}}
\newcommand{\scat}{\textup{scat}}

\DeclareMathOperator{\re}{Re}
\DeclareMathOperator{\im}{Im}
\DeclareMathOperator{\tr}{tr}

\newtheorem{lemma}{Lemma}
\newtheorem{proposition}[lemma]{Proposition}
\newtheorem{thm}[lemma]{Theorem}
\theoremstyle{definition}
\newtheorem{definition}{Definition}
\theoremstyle{remark}
\newtheorem{remark}{Remark}

\title{Upper bound on the number of resonances for even asymptotically hyperbolic manifolds with real-analytic ends}
\author{Malo Jézéquel\footnote{Department of Mathematics, Massachusetts Institute of Technology, Cambridge, MA 02139. Email: mpjez@mit.edu}}
\date{}

\begin{document}

\maketitle

\begin{abstract}
We prove a polynomial upper bound on the number of resonances in a disk whose radius tends to $+ \infty$ for even asymptotically hyperbolic manifolds with real-analytic ends. Our analysis also gives a similar upper bound on the number of quasinormal frequencies for Schwarzschild--de Sitter spacetimes.
\end{abstract}

\section{Introduction}

The purpose of this work is to prove an upper bound for the number of resonances for even asymptotically hyperbolic manifolds with real-analytic (but \emph{a priori} not exactly hyperbolic) ends. Let us recall that an asymptotically hyperbolic manifold is a Riemannian manifold $(M,g)$ such that $M$ is the interior of a compact manifold with boundary $\overline{M}$ and there is an identification of a neighbourhood of $\partial \overline{M}$ with $[0,\epsilon[_{y_1} \times \partial \overline{M}_{y'}$ that puts the metric $g$ into the form
\begin{equation}\label{eq:normal_intro}
g = \frac{\mathrm{d} y_1^2 + g_1(y_1,y',\mathrm{d}y')}{y_1^2},
\end{equation}
where $g_1(y_1,y',\mathrm{d}y')$ is a family of metrics on $\partial \overline{M}$ depending on $y_1$. We say that $(M,g)$ is even if $g_1$ is a smooth function of $y_1^2$. We refer to \cite[\S 5.1]{dyatlov_zworski_book} for a detailed discussion of this notion. 

Letting $\Delta$ denote the (non-positive) Laplace operator on an even asymptotically hyperbolic manifold $(M,g)$ of dimension $n$, one commmonly introduces the family of operators, depending on the complex parameter $\lambda$,
\begin{equation}\label{eq:resolvent_L2}
\p{-\Delta - \frac{(n-1)^2}{4} - \lambda^2 }^{-1} : L^2(M) \to L^2(M), \quad \im \lambda > 0.
\end{equation}
Since the essential spectrum of $- \Delta$ is $[(n-1)^2/4, + \infty[$, this family of operators is well-defined and meromorphic for $\im \lambda > 0$, with maybe a finite number of poles between $0$ and $i (n-1)/2$ on the imaginary axis, corresponding to the eigenvalues of $- \Delta$ in $]0, (n-1)^2/4[$. Notice that the residues of these poles have finite ranks.

The \emph{scattering resolvent} of $(M,g)$ is then defined as the meromorphic continuation of \eqref{eq:resolvent_L2}, as provided by the following result.

\begin{thm}[\cite{mazzeo_melrose,guillarmou_resolvent}]\label{thm:extension_resolvent}
Let $(M,g)$ be an even asymptotically hyperbolic manifold of dimension $n$. Then the resolvent \eqref{eq:resolvent_L2} admits a meromorphic extension $R_{\scat}(\lambda)$ to $\mathbb{C}$ as an operator from $C_c^\infty(M)$ to $\mathcal{D}'(M)$, with residues of finite rank.
\end{thm}

In the case of manifolds that are exactly hyperbolic near infinity, one may also refer to \cite{guizwor_resolvent}. Notice that we do not use here the same spectral parameter as in \cite{mazzeo_melrose,guillarmou_resolvent,guizwor_resolvent}. The spectral parameter from these references is given in terms of our $\lambda$ as $\zeta = \frac{n-1}{2} - i \lambda$. Another proof of Theorem \ref{thm:extension_resolvent} has been given by Vasy in \cite{vasy_method} (see also \cite{vasy_method_2}, \cite{vasy_revisited} and \cite[Chapter 5]{dyatlov_zworski_book}). 

The poles of the scattering resolvent (the meromorphic continuation of \eqref{eq:resolvent_L2}) are called the resonances of $(M,g)$. If $\mu \neq 0$ is a scattering resonance for $(M,g)$ then we define the multiplicity of $\mu$ as the rank of the operator 
\begin{equation}\label{eq:operateur_rang_fini}
\frac{i}{\pi} \int_\gamma \lambda R_{\scat}(\lambda) \mathrm{d}\lambda,
\end{equation}
where $\gamma$ is a small positively oriented circle around $\mu$ (so that the index of $\mu$ with respect to $\gamma$ is $1$, and the index of any other resonance is zero). That this operator has finite rank follows from the fact that the residues of $R_{\scat}(\lambda)$ have finite ranks. Another definition for the multiplicity of resonances may be found for instance in \cite[Definition 1.2]{guillopezworski}, but it coincides with the one we gave when $\mu$ is non-zero (see \cite[Proposition 2.11]{guillopezworski}). The definition of the multiplicity of $0$ as a resonance is more subtle (and will not matter in our case), see the discussion after Theorem 1 in \cite{zworski_multiplicity_zero}. Notice that in \cite{mazzeo_melrose,vasy_method}, the scattering resolvent $R_{\scat}(\lambda)$ is constructed as an operator from the space $\dot{C}^\infty(M)$ of smooth functions on $\overline{M}$ that vanish at infinite orders on $\partial \overline{M}$ to its dual. Since $C_c^\infty(M)$ is contained in $\dot{C}^\infty(M)$, we stated in Theorem \ref{thm:extension_resolvent} a weaker result. Notice however that, since $C_c^\infty(M)$ is dense in $\dot{C}^\infty(M)$, this simplification does not modify the notion of multiplicity of a resonance.

Our main result is an upper bound on the number of resonances for even asymptotically hyperbolic manifolds with real-analytic ends (as defined in \S \ref{section:geometric_assumption}).

\begin{thm}\label{theorem:main}
Let $(M,g)$ be an even asymptotically hyperbolic manifold real-analytic near infinity (as defined in \S \ref{section:geometric_assumption}) of dimension $n$. For $r > 0$, let $N(r)$ denote the number of resonances of $(M,g)$ of modulus less than $r$, counted with multiplicities. Then:
\begin{equation}\label{eq:bound_main}
N(r) \underset{r \to + \infty}{=} \mathcal{O}(r^n).
\end{equation}
\end{thm}

This upper bound is natural, since it is coherent with the asymptotic for the number of eigenvalues for the Laplacian on a closed Riemannian manifold given by Weyl law. There are also non-compact examples for which the bound \eqref{eq:bound_main} is optimal, see the lower bounds from \cite{guillopezworski,borthwickhyperbolic} discussed below.

There is a long tradition of studies of such counting problems in scattering theory, going back to the work of Tullio Regge in the fifties \cite{regge1958analytic}. Results similar to Theorem \ref{theorem:main} have been established in the context of scattering (e.g. by a compactly supported potential or by certain black boxes) on odd-dimensional \emph{Euclidean} spaces \cite{melrose_obstacle,boundscattering,sz,Vodev}. In the context of asymptotically hyperbolic manifolds, the bound \eqref{eq:bound_main} is known for manifolds with \emph{exactly} hyperbolic ends \cite{guillope_zworski_upper_bound,CuevasVodev,borthwickhyperbolic}. Still in the case of manifolds with exactly hyperbolic ends, we also have some lower bounds available: in the case of surfaces Guillopé and Zworski \cite{guillopezworski} proved that $r^2 = \mathcal{O}(N(r))$, which implies that \eqref{eq:bound_main} is optimal in that case. In higher dimension $n$, Borthwick \cite{borthwickhyperbolic} proved a similar lower bound $r^{n} = \mathcal{O}(N^{\textup{sc}}(r))$ for compact perturbations of conformally compact hyperbolic manifolds (a stronger assumption than just having exactly hyperbolic ends). This lower bound is obtained for the counting function $N^{\textup{sc}}(r)$ associated to a larger set of resonances than $N(r)$, and that also satisfies \eqref{eq:bound_main}. However, a few cases in which the same lower bound for $N(r)$ follows are given in \cite{borthwickhyperbolic}. Finally, a lower bound for $N(r)$ of the form
\begin{equation*}
\limsup_{r \to + \infty} \frac{\log N(r)}{\log r} = n
\end{equation*}
is proven for generic compact perturbations of a manifold with exactly hyperbolic ends in \cite{loweroptimal}.

Leaving the context of manifolds with exactly hyperbolic ends, much less is known on the asymptotic of the counting function $N(r)$. The bound \eqref{eq:bound_main} is established by Borthwick and Philipp \cite{warped} in the case of asymptotically hyperbolic manifolds with warped-product ends, that is for which the coordinates $(y_1,y')$ in \eqref{eq:normal_intro} may be chosen so that $g_1(y_1,y',\mathrm{d}y') = g_1(y',\mathrm{d}y')$ does not depend on $y_1$. The proof of a similar bound is sketched in \cite{sketch} for a class of asymptotically hyperbolic manifolds with ends that are asymptotically warped. In \cite{WangAnalytic}, Wang establishes, for certain real-analytic asymptotically hyperbolic metrics on $\mathbb{R}^3$, a polynomial bound $\mathcal{O}(r^6)$ for the number of resonances in a sector of the form 
\begin{equation}\label{eq:sector_Wang}
\set{ z \in \mathbb{C} : \epsilon < \va{z} < r, - \frac{\pi}{2} + \epsilon < \arg z < \frac{3\pi}{2} - \epsilon}
\end{equation}
when $r$ tends to $+ \infty$ while $\epsilon > 0$ is fixed. The evenness assumption is not made in \cite{WangAnalytic}, hence the necessity to count resonances in sectors of the form \eqref{eq:sector_Wang} rather than in disks (one has to avoid the essential singularities that can appear in the noneven case according to \cite{guillarmou_resolvent}). In the even case, our result, Theorem \ref{theorem:main}, improves the bound from \cite{WangAnalytic}, not only because we can count resonances in a disk, but also because our result, valid in any dimension, gives a better exponent in the $3$-dimensional case.

Let us point out that the upper bound \eqref{eq:bound_main} is also satisfied by the counting functions for the \emph{Ruelle resonances} of a real-analytic Anosov flow, as follows from a result of Fried \cite{friedzeta} based on techniques introduced by Rugh \cite{R1,R2}. We gave a new proof of this result with Guedes Bonthonneau in \cite{BJ20}, adapting techniques originally developed by Helffer and Sjöstrand \cite{helfferResonancesLimiteSemiclassique1986, sjostrandDensityResonancesStrictly1996}. The tools of real-analytic microlocal analysis that we use in the present paper rely heavily on \cite{BJ20}.

The main idea behind the proof of Theorem \ref{theorem:main} is to adapt the method of Vasy \cite{vasy_method} to construct the scattering resolvent, by introducing tools of real-analytic microlocal analysis. The method of Vasy does not only apply to even asymptotically hyperbolic manifolds, it may also be used to study resonances associated to the wave equation on \emph{Schwarzschild--de Sitter spacetimes} (in this context, resonances are also called \emph{quasinormal frequencies}). The interested reader may for instance refer to \cite[\S 6]{dafermos_rodnianski} for a description of the geometry of Schwarzschild--de Sitter spacetimes. Consequently, our method also gives an upper bound on the number of resonances (or quasinormal frequencies) for Schwarzschild--de Sitter spacetimes.

\begin{thm}\label{theorem:schwarzschild}
The number of quasinormal frequencies of modulus less than $r$ for a Schwarzschild--de Sitter spacetime is $\mathcal{O}(r^3)$ when $r$ tends to $+ \infty$.
\end{thm}

It is proven in \cite{sabarreto_zworski} that the quasinormal frequencies for a Schwarzschild--de Sitter spacetime are well approximated by the pseudopoles
\begin{equation*}
c\p{\pm \ell \pm \frac{1}{2} - i \left(k + \frac{1}{2}\right)}, 
\end{equation*}
for $k \in \mathbb{N}$ and $\ell \in \mathbb{N}^*$, the corresponding pole having multiplicity $2 \ell +1$. Here, $c$ is a constant depending on the mass of the black hole and the cosmological constant. However, the approximation given in \cite{sabarreto_zworski} is only effective for a pseudopole $\mu$ such that $\va{\mu}$ tends to $+ \infty$ while the imaginary part of $\mu$ remains bounded from below. Consequently, while Theorem \ref{theorem:schwarzschild} seems reasonable in view of the approximation result from \cite{sabarreto_zworski}, these two results discuss two different asymptotics. The result from \cite{sabarreto_zworski} cannot be used to prove Theorem \ref{theorem:schwarzschild}, nor to prove that Theorem \ref{theorem:schwarzschild} is sharp (even though it suggest that it should be the case).

It may be possible that the method of the proof of Theorems \ref{theorem:main} and \ref{theorem:schwarzschild} generalize to the case of slowly rotating Kerr--de Sitter black holes (as the method of Vasy also applies in this context \cite[\S 6]{vasy_method}). However, there are some additional technical difficulties that would probably arise in that case, due to the microlocal geometry being more complicated than in the Scharzschild--de Sitter case. In particular, there are bicharacteristics that originate at the source above the event horizon, then enter the domain of outer communication and eventually leave it. Our strategy of proof would require the propagation of singularities along these bicharacteristics using real-analytic microlocal analysis. Consequently, in order to deal with Kerr--de Sitter spacetimes, one cannot use real-analytic tools only near the event and cosmological horizon, as it is the case in the proof of Theorem \ref{theorem:schwarzschild}, see Remark \ref{remark:explain_assumption}. Since the coefficients of Kerr--de Sitter spacetimes are real-analytic on the whole domain of outer communication, it is not unlikely that this problem may be solved. In any case, we expect that one would need to use an escape function more carefully designed than in our analysis below.

\subsection*{Idea of the proof} As mentioned above, the proof of Theorems \ref{theorem:main} and \ref{theorem:schwarzschild} is based on an adaptation of the method of Vasy \cite{vasy_method} to construct the scattering resolvent, by introducing tools of real-analytic microlocal analysis. Our approach of the method of Vasy is mostly based on the exposition from \cite[Chapter 5]{dyatlov_zworski_book}.

The starting point of the proof of Theorems \ref{theorem:main} and \ref{theorem:schwarzschild} is the following observation. When using the method of Vasy to construct the scattering resolvent, one will construct a meromorphic extension to \eqref{eq:resolvent_L2} on a half plane of the form
\begin{equation}\label{eq:half_plane}
\set{ \lambda \in \mathbb{C} : \im \lambda > - C},
\end{equation}
for a given $C > 0$, by studying the action of a modified Laplacian on a functional space $H_C$ \emph{that depends on $C$}. The constant $C$ may be chosen arbitrarily large, so that we get indeed a meromorphic continuation to $\mathbb{C}$, but this requires a change in the space on which the modified Laplacian is acting. 

In the context of even asymptotically hyperbolic manifolds, the space $H_C$ is constructed in the following manner: one embeds $M$ as a relatively compact subset of a manifold $X$, and replaces the operator $- \Delta - (n-1)^2/4 - \lambda^2$ by a family of modified Laplacians. These modified Laplacians are elliptic on $M$ but have a source/sink structure above the boundary of $M$ in $X$. One can then set up a Fredholm theory for the modified Laplacians by using microlocal radial estimates (see for instance \cite[\S E.4]{dyatlov_zworski_book}). However, radial estimates in the $C^\infty$ category are limited by a threshold condition. In our setting, it imposes choosing space $H_C$ as a space of functions with a number of derivatives proportional to $C$ in order to get a meromorphic continuation of \eqref{eq:resolvent_L2} on the half-plane \eqref{eq:half_plane}. 

Consequently, working only with $C^\infty$ tools, one will \emph{a priori} only have access to bound on the number of resonances when restricting to a half-plane of the form \eqref{eq:half_plane}. A natural idea to tackle this difficulty is to work with a space ``$H_\infty$'' of functions that are smooth near the boundary of $M$ in $X$ (in our case, this would be real-analytic objects). If one is able to prove a real-analytic version of the radial estimates, it should be possible to bypass the threshold condition and construct directly the meromorphic continuation of \eqref{eq:resolvent_L2} to the whole $\mathbb{C}$, working on a single space $H_\infty$. One can then hope that this functional analytic setting can be used to prove a global bound on the number of resonances, without the need to restrict to a half-plane of the form \eqref{eq:half_plane}. We will use the tools from \cite{BJ20}, based on the work of Helffer and Sjöstrand \cite{helfferResonancesLimiteSemiclassique1986,sjostrandDensityResonancesStrictly1996}, to prove an estimate that is in some sense a real-analytic version of a radial estimate (see also \cite{galkowskiViscosityLimits0th2019}). Notice that similar estimates are proved in \cite{galkowskiAnalyticHypoellipticityKeldysh2020,BGJ} in different geometric contexts, and with a focus more on the hypoellipticity statement that may be deduced from the radial estimates rather than on the functional analytic consequences. In some sense, the results on resonances for $0$th order pseudodifferential operators in \cite{galkowskiViscosityLimits0th2019} and the results on real-analytic and Gevrey Anosov flows from \cite{BJ20} are already implicitly based on real-analytic radial estimates.

There is an important technical difference between the idea of the proof of Theorems \ref{theorem:main} and \ref{theorem:schwarzschild} as depicted above and the way the proof is actually written. Indeed, we cannot work with a space $H_\infty$ of functions that are analytic everywhere on $X$ (in particular because we do not want to assume that $g$ is analytic everywhere in $M$). Due to the lack of real-analytic bump functions, it is not easy to construct a space of functions that are real-analytic somewhere but have (at most) finite differentiability somewhere else, and that can be used to construct the scattering resolvent. We solve this issue using a strategy that was already present in \cite{BGJ}: we introduce a semiclassical parameter $h > 0$ and work with a space of distributions $\mathcal{H}$ on $X$ that depends on $h$. Let us point out that the \emph{space} $\mathcal{H}$ really depends on $h$, not only its norm. As $h$ tends to $0$, the elements of $\mathcal{H}$ are more and more regular near the boundary of $M$ in $X$. We can then invert a rescaled modified Laplacian acting on $\mathcal{H}$ after the addition of a trace class operator whose trace class norm is controlled as $h$ tends to $0$, and the upper bound from Theorems \ref{theorem:main} and \ref{theorem:schwarzschild} will follow.

\subsection*{Structure of the paper} In \S \ref{section:general_statement}, we introduce a set of general assumptions that will allow us to deal simultaneously with the analysis in the context of Theorems \ref{theorem:main} and \ref{theorem:schwarzschild}. The point of these assumptions is not to cover a wide generality, but to avoid to write the same proof twice with only notational changes. We state in \S \ref{section:general_statement} a general result, Proposition \ref{proposition:general_statement}, from which Theorems \ref{theorem:main} and \ref{theorem:schwarzschild} will be deduced.

In \S \ref{section:SdS} and \S \ref{section:asymptotically_hyperbolic}, we prove respectively Theorems \ref{theorem:schwarzschild} and \ref{theorem:main}.

In \S \ref{section:FBI_transform}, we recall and extend some results from \cite{BJ20} that will be needed for the proof of Proposition \ref{proposition:general_statement}.

Finally, \S \ref{section:general_construction} is the main technical part of the paper, as it contains the proof of Proposition \ref{proposition:general_statement}.

\subsection*{Acknowledgement} I would like to thank Maciej Zworski and Semyon Dyatlov for useful discussions about resonances for asymptotically hyperbolic manifolds.

\newpage

\tableofcontents

\newpage

\section{A general statement}\label{section:general_statement}

In order to deal simultaneously with the cases of asymptotically hyperbolic manifolds and of Schwarzschild--de Sitter spacetimes, we introduce here an abstract set of assumptions that are enough to make our analysis work.

\subsection{General assumption}\label{subsection:general_assumption}
We will use the notion of semiclassical differential operator, so let us recall very briefly what it means (see \cite{zworski_book} or \cite[Appendix E]{dyatlov_zworski_book} for more details on semiclassical analysis). A semiclassical differential operator $Q$ of order $m \in \mathbb{N}$ on a smooth manifold $X$ is a differential operator on $X$, depending on a small, so-called semiclassical, implicit parameter $h > 0$, of the form
\begin{equation*}
Q = \sum_{k = 0}^m h^k Q_k,
\end{equation*}
where $Q_k$ is a differential operator of order $k$ on $X$ that does not depend on $h$, for $k = 0,\dots,m$. With $Q$ one may associate its (semiclassical) principal symbol $q : T^* X \to \mathbb{C}$, which is a polynomial of degree $m$ in each fiber of $T^* X$. We may define $q$ as the unique $h$-independent function on $T^* X$ such that, for every smooth function $\varphi : M \to \mathbb{C}$ and $x \in X$, we have
\begin{equation*}
e^{- i\frac{\varphi(x)}{h}} Q\p{e^{i \frac{\varphi}{h}}}(x) \underset{h \to 0}{=} q(x, \mathrm{d}_x \varphi) + \mathcal{O}(h).
\end{equation*}
Notice that $q = \sum_{k = 0}^m q_m$ where $q_k$ denotes the (classical) homogeneous principal symbol of the differential operator $Q_k$ for $k = 0,\dots,m$. In the applications from \S \ref{section:SdS} and \S \ref{section:asymptotically_hyperbolic}, the introduction of the semiclassical parameter $h$ will be somehow artificial, this is just a technical trick.

Now that this reminder is done, we are ready to state our set of general assumptions.

Let $X$ be a closed real-analytic manifold of dimension $n$. We endow $X$ with a real-analytic Riemannian metric (this is always possible, see \cite{Morrey_embedding}). Let $Y$ be an open subset of $X$ with real-analytic boundary $\partial Y$. Consider a family of differential operators
\begin{equation}\label{eq:pinceau}
\mathcal{P}_h(\omega) = P_2 + \omega P_1 + \omega^2 P_0,
\end{equation}
where $\omega \in \mathbb{C}$ and the operator $P_j$ for $j \in \set{0,1,2}$ is a semiclassical differential operator (that does not depend on $\omega$) on $X$ of order $j$ with principal symbol $p_j$. We assume that there is $\epsilon > 0$ and a neighbourhood $U$ of $\partial Y$ with real-analytic coordinates $(x_1,x') : U \to ]- \epsilon,\epsilon[ \times \partial Y$ such that $\set{x_1 = 0} = \partial Y$ and $\set{x_1 > 0} = Y \cap U$. We require in addition that $P_0,P_1$ and $P_2$ have real-analytic coefficients in $U$ and that the following properties hold:
\begin{enumerate}
\item for $(x_1,x',\xi_1,\xi') \in T^* U \simeq T^*(]-\epsilon,\epsilon[ \times \partial Y)$, we have $p_2(x_1,x',\xi_1,\xi') = w(x_1) \xi_1^2 + q_1(x_1,x',\xi')$ where $q_1$ is a homogeneous real-valued symbol of order $2$ on $]-\epsilon,\epsilon[ \times T^* \partial Y$ and $w : ]- \epsilon,\epsilon[ \to \mathbb{R}$ is a real-analytic function such that $w(0) = 0$ and $w'(0) > 0$; \label{item:structure_principal}
\item there is a constant $C>0$ such that for $(x_1,x',\xi_1,\xi') \in T^* U$ we have $q_1(x_1,x',\xi') \geq C^{-1} \va{\xi'}^2$; \label{item:elliptique_bord}
\item the symbol $p_1$ is real-valued, $p_1(x_1,x',\xi_1,\xi') = p_1(x_1,\xi_1)$ does not depend on $(x',\xi')$ for $(x_1,x',\xi_1,\xi') \in T^* U$, and there is $C > 0$ such that
\begin{equation*}
\frac{p_1(x_1,\xi_1)}{\xi_1} \leq - C^{-1},
\end{equation*}
in particular the sign of $p_1(x_1,\xi_1)$ is the same as the sign of $-\xi_1$; \label{item:structure_sous_principal}
\item the symbol $p_2$ is real-valued and positive on $T^* Y \setminus \set{0}$; \label{item:ellipticite_interieur}
\item the symbol $p_0$ is real-valued and negative on a neighbourhood of $\overline{Y}$. \label{item:ellipticite_bas}
\end{enumerate}

\begin{remark}\label{remark:explain_assumption}
Let us explain the significance of these assumptions. In the context of the proof of Theorem \ref{theorem:main}, the manifold $X$ will be an even extension for $M$, and $Y$ will be $M$ seen as a subset of the even extension $X$. In the context of Theorem \ref{theorem:schwarzschild}, $Y$ will be the domain of outer communication and $\partial Y$ corresponds to the event and cosmological horizons. In both cases, $\mathcal{P}_h(\omega)$ will be a (semiclassically rescaled) family of modified operators. For instance, in the context of Theorem \ref{theorem:main}, we replace the operator $- h^2 \Delta - h^2 (n-1)^2/4 - \omega^2$ by a modified Laplacian $\mathcal{P}_h(\omega)$ (see for instance \cite[\S 5.3]{dyatlov_zworski_book}). The new operator $\mathcal{P}_h(\omega)$ is defined on the whole $X$, and, for $f$ smooth and compactly supported in $Y$, solving for $u$ the equation $\mathcal{P}_h(\omega) u =f$ with $u$ that satisifes a regularity condition near $\partial Y$ amounts to solve for $\tilde{u}$ the equation $(- h^2 \Delta - h^2 (n-1)^2/4 - \omega^2)\tilde{u} = \tilde{f}$ while imposing a certain behavior at infinity for $\tilde{u}$ (here $\tilde{f}$ depends on $f$ and is smooth and compactly supported in $M$).

A method to construct the scattering resolvent is then to construct a meromorphic inverse $\mathcal{P}_h(\omega)^{-1}$ for $\mathcal{P}_h(\omega)$. In Proposition \ref{proposition:general_statement} below, we give a new construction of this meromorphic inverse (maybe after modifying $\mathcal{P}_h(\omega)$ away from $\overline{Y}$, which is harmless since we only care about what happens on $Y$). This new construction is inspired by the method of Vasy \cite{vasy_method} (see also \cite[Chapter 5]{dyatlov_zworski_book}) with the addition of tools of real-analytic microlocal analysis near $\partial Y$.

Let us explain very briefly how it works. The idea is to set up a Fredholm theory for $\mathcal{P}_h(\omega)$. Inside $Y$, the operator $\mathcal{P}_h(\omega)$ is elliptic (due to assumption \ref{item:ellipticite_interieur}), so that there is no problem here. Outside of $\overline{Y}$, we are allowed to modify $\mathcal{P}_h(\omega)$, and we can consequently deal with this part of $X$ by adding to $\mathcal{P}_h(\omega)$ a well-chosen elliptic operator. This is similar to the addition of a complex absorbing potential in \cite{vasy_method}, and possible because of the hyperbolic structure of $\mathcal{P}_h(\omega)$ near $\partial Y$ in $X \setminus \overline{Y}$. Hence, the most important point is to understand what happens at $\partial Y$, where $\mathcal{P}_h(\omega)$ stops being elliptic. At that place, the operator $\mathcal{P}_h(\omega)$ has a source/sink structure on its characteristic set (this is a consequence of the assumptions \ref{item:structure_principal} and \ref{item:elliptique_bord}), so that one can use radial estimates (see for instance to \cite[\S E.4]{dyatlov_zworski_book}) to set up a Fredholm theory for $\mathcal{P}_h(\omega)$. However, the $C^\infty$ versions of the radial estimates are restricted by a threshold condition: they can be used to construct the scattering resolvent, but they do not give a bound on the number of resonances in disks as in Theorem \ref{theorem:main}. This is where real-analytic microlocal analysis becomes useful: using methods as in \cite{BJ20,galkowskiViscosityLimits0th2019} (see also \cite{galkowskiAnalyticHypoellipticityKeldysh2020,BGJ}), we are able to get an estimate which is in some sense a $C^\omega$ version of a radial estimate and allows us to prove Theorem \ref{theorem:main}. This estimate corresponds to the fourth and fifth case in the proof of Lemma \ref{lemma:high_frequencies}.

There are some technical reasons that make our set of assumptions very specific. The assumption \ref{item:ellipticite_bas} is rather artificial, this is just a way to ensure that our family of Fredholm operators will be invertible at a point. The assumptions \ref{item:structure_principal}, \ref{item:elliptique_bord} and \ref{item:structure_sous_principal} impose that the source/sink structure of $\mathcal{P}_h(\omega)$ on its characteristic set is very particular. This specific structure will allow us to work in the real-analytic category only near $\partial Y$, which is essential because we are not able to ensure that $\mathcal{P}_h(\omega)$ is analytic away from $\partial Y$. Concretely, this ensures that near $\partial Y$ in $X \setminus \overline{Y}$, the projection on $X$ of the bicharacteristics curve of $\mathcal{P}_h(\omega)$ that are contained in its characteristic set go either toward or away from $\partial Y$. This allows us to set up a propagation estimate by working on spaces weighted by $e^{\psi/h}$ where $\psi$ is a function on $X$ monotone along the projection to $X$ of the bicharacteristics of $\mathcal{P}_h(\omega)$. This estimate does not require real-analytic coefficients, so that it can be used to make the link between $\partial Y$ (where we really need real-analytic machinery) and the place in $X \setminus \overline{Y}$ where $\mathcal{P}_h(\omega)$ is artificially made elliptic by the addition of a differential operator with $C^\infty$ coefficients. 
\end{remark}

\subsection{General result}

The assumptions from \S \ref{subsection:general_assumption} allow us to state an abstract result from which Theorems \ref{theorem:main} and \ref{theorem:schwarzschild} follow.

\begin{proposition}\label{proposition:general_statement}
Under the assumptions from \S \ref{subsection:general_assumption}, we may modify the operator $\mathcal{P}_h(\omega)$ away from $\overline{Y}$ into a new operator $P_h(\omega)$ so that the following holds. There are two Hilbert spaces $\mathcal{H}_1,\mathcal{H}_2$ (depending on $h$) and a constant $\kappa >0$ (that does not depend on $h$) such that the following properties hold when $h$ is small enough:
\begin{enumerate}[label=(\roman*)]
\item for $j = 1,2$, there are continuous inclusions $C^\infty(X) \subseteq \mathcal{H}_j \subseteq \mathcal{D}'(X)$;\label{item:continuous_inclusions}
\item for $j = 1,2$, the elements of $\mathcal{H}_j$ are continuous on a neighbourhood of $\partial Y$;\label{item:continuous_boundary}
\item $P_h(\omega) : \mathcal{H}_1 \to \mathcal{H}_2$ is a holomorphic family of bounded operators;\label{item:holomorphic_family}
\item there is $\nu >0$ such that $P_h(i \nu) : \mathcal{H}_1 \to \mathcal{H}_2$ is invertible;\label{item:invertible}
\item for every open and relatively compact subset $V$ of $ \set{z \in \mathbb{C}: \im z > - \kappa}$, if $h$ is small enough then, for every $\omega \in V$, the operator $P_h(\omega) : \mathcal{H}_1 \to \mathcal{H}_2$ is Fredholm of index $0$. Moreover, this operator has a meromorphic inverse $\omega \mapsto P_h(\omega)^{-1}$ on $V$ with poles of finite rank;\label{item:fredholm_family}
\item if $\delta \in ]0,\kappa[$, there is $C > 0$, such that for every $h$ small enough, the number of $\omega$ in the disk of center $0$ and radius $\delta$ such that $P_h(\omega) : \mathcal{H}_1 \to \mathcal{H}_2$ is not invertible (counted with null multiplicity) is less than $C h^{-n}$.\label{item:counting_resonances}
\end{enumerate}
\end{proposition}

\begin{remark}
The notion of null multiplicity used in the statement of Proposition \ref{proposition:general_statement} is defined using the Gohberg--Sigal theory (see for instance \cite[\S C.4]{dyatlov_zworski_book}). In our context, we can use the following definition: if $\omega_0$ is such that the meromorphic inverse $\omega \mapsto P_h(\omega)$ is defined near $\omega_0$, then the null multiplicity of $P_h(\omega)$ at $\omega_0$ is the trace of the residue of $\omega \mapsto P_h(\omega)^{-1} \partial_\omega P_h(\omega)$ at $\omega_0$ (which is a finite rank operator).
\end{remark}

\begin{remark}
The modification of $\mathcal{P}_h(\omega)$ needed to get Proposition \ref{proposition:general_statement} will be obtained by modifying the coefficients of $P_0,P_1$ and $P_2$ away from $\overline{Y}$, so that the general assumption are still satisfied by $P_h(\omega)$ after this modification.
\end{remark}

\section{Schwarzschild--de Sitter spacetimes (proof of Theorem \ref{theorem:schwarzschild})}\label{section:SdS}

In this section, we explain how the general framework from \S \ref{section:general_statement} can be used to prove Theorem \ref{theorem:schwarzschild}. We start with this case because the setting is slightly simpler than in Theorem \ref{theorem:main} that we prove in \S \ref{section:asymptotically_hyperbolic} below. We recall a few basic facts about Schwarzschild--de Sitter spacetimes in \S \ref{subsec:model} and then apply Proposition \ref{proposition:general_statement} in \S \ref{subsection:quasinormal_frequencies}. Finally, in \S \ref{subsec:spherical_harmonics}, we discuss the number of resonances for the operators obtained by decomposing functions on Schwarzschild--de Sitter spacetimes on spherical harmonics.

\subsection{The model}\label{subsec:model}

We start by recalling the definition of Schwarzschild--de Sitter spacetimes and of the associated quasinormal frequencies. The interested reader may refer to \cite{dafermos_rodnianski} for the geometry of Schwarzschild--de Sitter spacetimes (and other notion from general realtivity). For the definition of the resonances, one may refer to \cite{sabarreto_zworski} or \cite{vasy_method}. Fix two constants
\begin{equation*}
M_0 > 0 \textup{ and } 0 < \Lambda < \frac{1}{9 M_0^2}.
\end{equation*}
The constant $M_0$ is called the \emph{mass of the black hole} and $\Lambda$ the \emph{cosmological constant}. We define the function
\begin{equation*}
G(r) = 1 - \frac{\Lambda r^2}{3} - \frac{2 M_0}{r} \quad \textup{ for } r > 0.
\end{equation*}
Let then $r_- < r_+$ be the positive roots of the polynomial $r G(r)$. Define $M = ]r_-,r_+[_r \times \mathbb{S}_y^2$ and $\widehat{M} = \mathbb{R}_t \times M$. Let $g$ be the Lorentzian metric
\begin{equation*}
g = - G \mathrm{d}t^2 + G^{-1} \mathrm{d}r^2 + r^2 g_S(y,\mathrm{d}y),
\end{equation*}
where $g_S$ denotes the standard metric on $\mathbb{S}^2$. The Lorentzian manifold $(\widehat{M},g)$ is then called a Schwarzschild--de Sitter spacetime. The hypersurfaces $\set{r_-} \times \mathbb{S}^2$ and $\set{r_+} \times \mathbb{S}^2$ are called respectively the event and the cosmological horizons.

In order to understand the asymptotic of the solution to the wave equation on $(\widehat{M},g)$, one studies the meromorphic continuation of the resolvant $(P_{\textup{SdS}} - \lambda^2)^{-1}$ where
\begin{equation*}
P_{\textup{SdS}} = G r^{-2} D_r(r^2 G) D_r - G r^{-2} \Delta_{\mathbb{S}^2}.
\end{equation*}
Here, $D_r = - i \partial_r$ and $\Delta_{\mathbb{S}^2}$ is the (non-positive) Laplace operator on the sphere $\mathbb{S}^2$. The operator $P_{\textup{SdS}}$ is self-adjoint and non-negative on the Hilbert space $L^2 \p{]r_-,r_+[ \times \mathbb{S}^2; G^{-1} r^2 \mathrm{d}r \mathrm{d}y}$, where $\mathrm{d}y$ denotes the standard volume form on $\mathbb{S}^2$. Consequently, the operator $(P_{\textup{SdS}} - \lambda^2)^{-1}$ is well-defined on this space when $\im \lambda > 0$. It is proven for instance in \cite[\S 2]{sabarreto_zworski} that $(P_{\textup{SdS}} - \lambda^2)^{-1}$ has a meromorphic continuation $R_{\textup{SdS}}(\lambda)$ to $\mathbb{C}$, with poles of finite rank, as an operator from $C_c^\infty\p{]r_-,r_+[ \times \mathbb{S}^2}$ to $\mathcal{D}'\p{]r_-,r_+[ \times \mathbb{S}^2}$. The poles of this meromorphic continuation are called the \emph{quasinormal frequencies} for the Schwarzschild--de Sitter spacetime. If $\lambda_0 \neq 0$ is a quasinormal frequency, we define its multiplicity as the rank of the operator
\begin{equation*}
\frac{i}{\pi} \int_\gamma \lambda R_{\textup{SdS}}(\lambda) \mathrm{d}\lambda,
\end{equation*}
where $\gamma$ is a positively oriented circle around $\lambda_0$, small enough so that the index of any other quasinormal frequency with respect to $\gamma$ is zero.

\subsection{Upper bound on the number of quasinormal frequencies}\label{subsection:quasinormal_frequencies}

Our proof of Theorem \ref{theorem:schwarzschild} is based on the method of Vasy \cite{vasy_method} to construct the resolvent $R_{\textup{SdS}}(\lambda)$, following mostly the exposition from \cite[Exercise 16 p.376]{dyatlov_zworski_book}. We start with a standard modification of the operator $P_{\textup{SdS}} - \lambda^2$, with some minor changes that will be convenient to check the assumptions from \S \ref{subsection:general_assumption}.

Let us embed a neighbourhood of $[r_-,r_+]$ in the circle $\mathbb{S}^1$ and set $X = \mathbb{S}^1 \times \mathbb{S}^2$ and $Y = ]r_-,r_+[ \times \mathbb{S}^2$. Let $\rho : ]r_-,r_+[ \to [-1,1]$ be a $C^\infty$ function, identically equal to $\pm 1$ near $r_\pm$. Let then $F : ]r_-,r_+[ \to \mathbb{R}$ be a primitive of 
\begin{equation}\label{eq:def_Fprime}
F'(r) = \rho(r)\p{ \frac{1}{G(r)} - \frac{1}{2(1 - (9 M_0^2 \Lambda)^{\frac{1}{3}})}}
\end{equation}
and introduce, for $\lambda \in \mathbb{C}$, the operator
\begin{equation*}
G^{-1} e^{- i \lambda F(r)} (P_{\textup{SdS}} - \lambda^2) e^{i \lambda F(r)},
\end{equation*}
which is explicitly given by the formula
\begin{equation}\label{eq:explicite_SdS}
\begin{split}
& G D_r^2 - r^{-2} \Delta_{\mathbb{S}^2} + \p{2 \lambda F' G - i \p{\frac{2G}{r} + G'}}D_r \\ & \qquad \qquad \qquad \qquad \qquad - i \lambda \p{\frac{2 G F'}{r} + G'F' + G F''} - \lambda^2 \frac{\p{1 - G^2(F')^2}}{G}.
\end{split}
\end{equation}
The coefficients of this differential operator extend as real-analytic functions near $r_-$ and $r_+$. Indeed, the definition of $F$ ensures that $F'G$ continues analytically passed $r_{-}$ and $r_+$. Moreover, near $r_{\pm}$ a direct computation yields
\begin{equation*}
G' F' + G F'' = \mp \frac{G'}{2(1 - (9 M_0^2 \Lambda)^{\frac{1}{3}})}
\end{equation*}
and
\begin{equation*}
\frac{1 - G^2(F')^2}{G} = \frac{1}{1 - (9 M_0^2 \Lambda)^{\frac{1}{3}}} - \frac{G}{4(1 - (9 M_0^2 \Lambda)^{\frac{1}{3}})^2}.
\end{equation*}

Letting $\chi$ be a $C^\infty$ function supported in a small neighbourhood of $[r_-,r_+]$ and identically equal to $1$ on a smaller neighbourhood of $[r_-,r_+]$, we can define a family of operators on $X$ by
\begin{equation*}
\Theta(\lambda) = \chi(r) \times \eqref{eq:explicite_SdS}.
\end{equation*}
Finally, for $\omega \in \mathbb{C}$, we define the semiclassical differential operator
\begin{equation*}
\begin{split}
\mathcal{P}_h(\omega) & \coloneqq h^2 \Theta(h^{-1} \omega) \\
\end{split}
\end{equation*}
Notice that this operator depends on the implicit semiclassical parameter $h$ as in \S \ref{section:general_statement}. It is of the form \eqref{eq:pinceau} with
\begin{equation*}
P_0 = \chi(r) \left(G h^2 D_r^2 - r^{-2} h^2 \Delta_{\mathbb{S}^2} - i \left(\frac{2G}{r}+ G'\right) h^2 D_r\right),
\end{equation*}
\begin{equation*}
P_1 = \chi(r) \left(2F'G h D_r -ih \left(\frac{2GF'}{r} + G'F' + GF'' \right)\right)
\end{equation*}
and
\begin{equation*}
P_2 = - \chi(r) \frac{1 - G^2(F')^2}{G},
\end{equation*}
where it is understood that the factor in parentheses continues analytically in $r$ passed $r_-$ and $r_+$. Let us check that the general assumptions from \S \ref{subsection:general_assumption} are satisfied by this family of operator.

We already mentioned that $\mathcal{P}_h(\omega)$ is of the form \eqref{eq:pinceau}, and it follows from the expression for the $P_j$'s given above that they are semiclassical differential operators of order $j$ with analytic coefficients on a neighbourhood of $\partial Y$. Moreover, the principal symbols of the $P_j$'s are given on $Y$ by
\begin{equation*}
\begin{split}
p_2(r,y,\rho,\eta) = G(r) \rho^2 + r^{-2} \eta^2, \quad p_1(r,y,\rho,\eta) = 2 F'(r) G(r) \rho
\end{split}
\end{equation*}
and
\begin{equation*}
p_0(r,y) = - \frac{1 - G(r)^2 F'(r)^2}{G(r)}.
\end{equation*}
We get the values of these symbols on a neighbourhood of $\overline{Y}$ by continuing these formulas analytically in $r$.

We can define the coordinates $(x_1,x')$ near $\partial Y$ by taking $x_1 = r - r_-$ (when $r$ is near $r_-$) or $x_1 = r_+ -r$ (when $r$ is near $r_+$) and $x' = y$. Beware here that this change of coordinates reverses the orientation of the real line near $r_+$. Then, we see that the assumption \ref{item:structure_principal} holds with $w(x_1) = G(r_{\pm} \mp x_1)$ and $q_1(x_1,y,\eta) = (r_{\pm} \mp x_1)^{-2} \eta^2$. In particular, we have $w'(0) = \mp G'(r_{\pm}) > 0$. The point \ref{item:elliptique_bord} follows from the definition of $q_1$. To get \ref{item:structure_sous_principal}, one only needs to notice that the value at $r_\pm$ of the real-analytic extension of $F'(r) G(r)$ is $\pm 1$ (and that our change of variable reverses orientation near $r_+$). Since $G$ is positive on $]r_-,r_+[$, we get \ref{item:ellipticite_interieur}. In order to check \ref{item:ellipticite_bas}, write
\begin{equation*}
p_0(r,y) = \frac{\rho(r)^2\p{1 - \frac{G(r)}{2(1 -(9 M_0^2 \Lambda)^{\frac{1}{3}})}}^2 - 1}{G(r)}.
\end{equation*}
Since $1 -(9 M_0^2 \Lambda)^{\frac{1}{3}}$ is an upper bound for $G$ on $]r_-,r_+[$, we find that $p_0(r,y) < 0$ for $r \in]r_-,r_+[$. Using that $\rho(r)^2$ is equal to $1$ when $r$ is near $r_\pm$, we find that
\begin{equation*}
p_0(r_\pm,y) = - \frac{1}{1 -(9 M_0^2 \Lambda)^{\frac{1}{3}}} < 0,
\end{equation*}
and thus \ref{item:ellipticite_bas} holds.

Consequently, we can modify $\mathcal{P}_h(\omega)$ away from $\overline{Y}$ in order to get a family of operator $P_h(\omega)$ that satisfies Proposition \ref{proposition:general_statement}. With $\kappa$ as in Proposition \ref{proposition:general_statement}, we let $V$ be a connected, relatively compact and open subset of $\set{z \in \mathbb{C} : \im z > - \kappa}$ that contains the closed disk of center $0$ and radius $\frac{3 \kappa}{4}$. Let $\iota_2$ denote the injection of $C_c^\infty(Y)$ in $\mathcal{H}_2$ and $\iota_1$ denote the map from $\mathcal{H}_1$ to $\mathcal{D}'(Y)$ obtained by composing the injection $\mathcal{H}_1 \to \mathcal{D}'(X)$ with the restriction map $\mathcal{D}'(X) \to \mathcal{D}'(Y)$.

If $\lambda \in h^{-1} V$, we define the resolvent
\begin{equation}\label{eq:resolvent_SdS}
R_h(\lambda) = e^{i \lambda F(r)} h^2 \iota_1 P_h(\lambda h)^{-1}\iota_2 e^{- i \lambda F(r)} G^{-1} : C_c^\infty(Y) \to \mathcal{D}'(Y).
\end{equation}
This is a meromorphic family of operators. We just got a new construction of the meromorphic continuation $R_{\textup{SdS}}(\lambda)$ of the $L^2$ resolvent $(P_{\textup{SdS}} - \lambda^2)^{-1}$, as we will now demonstrate.

\begin{lemma}\label{lemma:identification_resolvante}
If $\lambda \in h^{-1} V$ is such that $\im \lambda > 0$, then $R_h(\lambda)$ is the restriction to $C_c^\infty(Y)$ of the $L^2$ resolvent $(P_{SdS} - \lambda^2)^{-1}$. In particular $R_h(\lambda)$ does not depend on $h$.
\end{lemma}

\begin{proof}
Let $\lambda \in h^{-1} V$ be such that $\im \lambda > 0$. Let $u \in C_c^\infty(Y)$. Notice that 
\begin{equation*}
\begin{split}
& (P_{\textup{SdS}} - \lambda^2) R_h(\lambda)u \\ & \qquad  = G e^{i \lambda F(r)} G^{-1} e^{- i \lambda F(r)}(P_{\textup{SdS}} - \lambda^2) e^{i \lambda F(r)} h^2 \iota_1 P_h(\lambda h)^{-1}\iota_2 e^{- i \lambda F(r)} G^{-1} u \\
   & \qquad = G e^{i \lambda F(r)} P_h(\lambda h) \iota_1 P_h(\lambda h)^{-1}\iota_2 e^{- i \lambda F(r)} G^{-1} u = u,
\end{split}
\end{equation*}
where we used that $h^2 G^{-1} e^{- i \lambda F(r)} (P_{\textup{SdS}} - \lambda^2) e^{i \lambda F(r)} \iota_1 = P_h(\lambda h) \iota_1 = \iota_3 P_h(\lambda h)$, where $\iota_3$ is the map obtained by composing the injection $\mathcal{H}_2 \to \mathcal{D}'(X)$ with the restriction map $\mathcal{D}'(X) \to \mathcal{D}'(Y)$. Consequently, we only need to prove that the distribution $R_h(\lambda)u$ belongs to the space $L^2 \p{]r_-,r_+[ \times \mathbb{S}^2; G^{-1} r^2 \mathrm{d}r \mathrm{d}y}$.  Since $P_{\textup{SdS}}$ is elliptic, we know that $u$ is smooth, and thus bounded on all compact subsets of $Y$. It remains to understand the behaviour of $u$ near $\partial Y$.

Notice that $R_h(\lambda) u = e^{i \lambda F(r)} v$, where $v$ is the restriction to $Y$ of an element of $\mathcal{H}_1$. In particular, since the elements of $\mathcal{H}_1$ are continuous near $\partial Y$, there is a compact subset $K$ of $Y$ such that $v$ is continuous and bounded outside of $K$. Let us study for instance the behavior of $u$ near $r = r_-$ (the behavior near $r_+$ is similar). From \eqref{eq:def_Fprime}, we see that
\begin{equation*}
F(r) \underset{r \to r_-}{=} - \frac{\ln \va{r- r-}}{G'(r_-)} + \mathcal{O}(1).
\end{equation*}
Consequently, we have that $e^{i \lambda F(r)}$ is $\mathcal{O}(|r - r_-|^{\im \lambda/G'(r_-)})$ when $r$ tends to $r_-$. Working similarly near $r_+$, we find that $u$ belongs to the Hilbert space $L^2 \p{]r_-,r_+[ \times \mathbb{S}^2; G^{-1} r^2 \mathrm{d}r \mathrm{d}y}$.
\end{proof}

\begin{remark}\label{remark:multiplicity}
It follows from Lemma \ref{lemma:identification_resolvante} that $R_h(\lambda) = R_{\textup{SdS}}(\lambda)$ on $h^{-1} V$. In particular, $\lambda \in h^{-1} V $ is a quasinormal frequency if an only if it is a pole of $R_h(\lambda)$ and, if in addition $\lambda \neq 0$, its multiplicity is the rank of the operator
\begin{equation*}
\frac{i}{\pi}\int_\gamma \mu R_h(\mu) \mathrm{d}\mu,
\end{equation*}
where $\gamma$ is a small circle around $\lambda$.
\end{remark}

With this new construction of the resolvent $R_{\textup{SdS}}(\lambda)$ at our disposal, we are ready to prove Theorem \ref{theorem:schwarzschild}.

\begin{proof}[Proof of Theorem \ref{theorem:schwarzschild}]
Considering the bound on the number of points where $P_h(\omega)$ is not invertible given in Proposition \ref{proposition:general_statement}, we only need to prove that if $\lambda$ is a non-zero complex number of modulus less than $\frac{\kappa}{4h}$ then its multiplicity as a quasinormal frequency is less than the null multiplicity of $\omega \mapsto P_h(\omega)$ at $\lambda h$. 

Let us consider a quasinormal frequency $\lambda$ of modulus less than $\frac{\kappa}{4h}$. Since $P_h(\omega)$ is a holomorphic family of operators with a meromorphic inverse near $\lambda h$ (because $\lambda h$ belongs to $V$), it follows from the Gohberg--Sigal theory \cite[Theorem C.10]{dyatlov_zworski_book}, that there are holomorphic families of invertible operators $U_1(\omega)$ and $U_2(\omega)$ for $\omega$ near $\lambda h$, respectively on $\mathcal{H}_2$ and from $\mathcal{H}_1$ to $\mathcal{H}_2$, an integer $M \geq 0$, operators $P_0,\dots,P_M$ on $\mathcal{H}_2$ and non-zero integeres $k_1,\dots,k_M$ such that
\begin{equation}\label{eq:gohberg_sigal}
P_h(\omega) = U_1(\omega)\p{P_0 + \sum_{m = 1}^M (\omega - \lambda h)^{k_m} P_m} U_2(\omega),
\end{equation}
for $\omega$ near $\lambda h$. Moreover, $P_1,\dots, P_M$ are rank $1$ and $P_\ell P_m = \delta_{\ell,m} P_m$ for $0 \leq \ell,m \leq M$. We also have that $I = \sum_{m = 0}^M P_m$, since $P_h(\omega)$ is invertible for $\omega \neq \lambda h$ near $\lambda h$. Notice that the $k_m$'s must be positive, since $P_{h}(\omega)$ is holomorphic in $\omega$, and that the null multiplicity of $P_h(\omega)$ at $\lambda h$ is $\sum_{m = 1}^M k_m$.

It follows from \eqref{eq:gohberg_sigal} that
\begin{equation}\label{eq:gohberg_sigal_inverse}
P_h(\omega)^{-1} = U_2(w)^{-1} \p{P_0 + \sum_{m = 1}^M (\omega - \lambda h)^{-k_m} P_m} U_1(\omega)^{-1}.
\end{equation}
From \eqref{eq:resolvent_SdS} we get
\begin{equation*}
R_h(\mu) = A_1(\mu) + A_2(\mu), 
\end{equation*}
where $A_1$ and $A_2$ are obtained by replacing the inverse $P_h(\omega)^{-1}$ respectively by $U_2(\omega)^{-1} P_0 U_1(\omega)^{-1}$ and by $U_2(\omega)^{-1} \sum_{m = 1}^M (\omega- \lambda h)^{-k_m} P_m U_1(\omega)^{-1}$ in \eqref{eq:resolvent_SdS}, with $\omega= \mu h$. Notice that $A_1(\mu)$ is holomorphic in $\mu$, so that
\begin{equation}\label{eq:simplification_projecteur}
\int_\gamma \mu R_h(\mu) \mathrm{d}\mu = \int_\gamma \mu A_2(\mu) \mathrm{d}\mu.
\end{equation}
The operator $\mu A_2(\mu)$ is of the form $B_1(\mu) \p{\sum_{k = 1}^M (\mu -\lambda)^{-k_m}P_m} B_2(\mu)$, where $B_1(\mu)$ and $B_2(\mu)$ are holomorphic near $\lambda$. Writing the Taylor expansions for $B_1(\mu)$ and $B_2(\mu)$:
\begin{equation*}
B_j(\mu) = \sum_{\ell \geq 0}(\mu - \lambda)^\ell C_{j,l},
\end{equation*}
we find that the residue of $\mu A_2(\mu)$ at $\lambda$ is 
\begin{equation*}
\sum_{\substack{m,k,\ell \\ k+\ell = k_m - 1}} C_{1,k} P_m C_{2,\ell}.
\end{equation*}
This operator is the sum of $\sum_{m = 1}^M k_m$ operators of rank at most $1$, and thus is of rank at most $\sum_{m = 1}^M k_m$. It follows then from Remark \ref{remark:multiplicity} and \eqref{eq:simplification_projecteur} that the multiplicity of $\lambda$ as a scattering resonance is at most $\sum_{m = 1}^M k_m$, which is the null multiplicity of $\omega \mapsto P_h(\omega)$ at $\lambda h$.
\end{proof}

\subsection{Decomposition on spherical harmonics}\label{subsec:spherical_harmonics}

Notice that the Schwarzschild--de Sitter spacetime is radially symmetric. It is standard to use this kind of symmetries to study quasinormal frequencies by decomposing the operator $P_{\textup{SdS}}$ on spherical harmonics (see for instance \cite{sabarreto_zworski} or \cite{small_black_holes}). Let $\ell \in \mathbb{N}$ and $Y$ be a spherical harmonics satisfying $- \Delta_{\mathbb{S}^2} Y = \ell(\ell +1) Y$. The action of $P_{\textup{SdS}}$ on functions of the form $u(r) Y(y)$ is then equivalent to the action of the operator
\begin{equation*}
P_{\textup{SdS}}^\ell = G r^{-2} D_r(r^2 G) D_r + G r^{-2} \ell(\ell+1).
\end{equation*}
The operator $(P_{\textup{SdS}}^\ell - \lambda^2)^{-1}$ defined for $\im \lambda > 0$ by the spectral theory on $L^2(]r_-,r_+[; G^{-1} r^2 \mathrm{d}r)$ admits a meromorphic continuation to $\mathbb{C}$. The poles of this extension are quasinormal frequencies corresponding to angular momentum $\ell$.

We can then apply Proposition \ref{proposition:general_statement} as in \S \ref{subsection:quasinormal_frequencies} to get:

\begin{thm}
The number of quasinormal frequencies corresponding to the angular momentum $\ell$ of modulus less than $r$ is $\mathcal{O}(r)$ when $r$ tends to $+ \infty$.
\end{thm}

\section{Scattering on asymptotically hyperbolic manifolds (proof of Theorem \ref{theorem:main})}\label{section:asymptotically_hyperbolic}

In this section, we specify the geometric assumptions from Theorem \ref{theorem:main} and explain how one can use Proposition \ref{proposition:general_statement} to prove Theorem \ref{theorem:main}. In \S \ref{section:geometric_assumption} we describe the class of asymptotically hyperbolic manifolds with real-analytic ends that we are going to study. In \S \ref{subsec:even_extension} and \ref{subsection:modified_Laplacian}, we check the assumptions from \S \ref{subsection:general_assumption} in order to use Proposition \ref{proposition:general_statement} and prove Theorem \ref{theorem:main} in \S \ref{subsection:bound_resonances}.

The paragraphs \S \ref{section:geometric_assumption}, \S \ref{subsec:even_extension} and \S \ref{subsection:modified_Laplacian} are based on the exposition in \cite[Chapter 5]{dyatlov_zworski_book} of the method of Vasy \cite{vasy_method} to construct the scattering resolvent, with a few additional technicalities required to deal with real-analytic ends and apply Proposition \ref{proposition:general_statement}.

\subsection{Geometric assumptions}\label{section:geometric_assumption}

We explain here how the definition of asymptotically hyperbolic manifold may be modified to obtain the definition of asymptotically hyperbolic manifolds with real-analytic ends that appears in Theorem \ref{theorem:main}. Let us consider a Riemannian manifold $(M,g)$ where $M$ is a real-analytic manifold but the metric $g$ is \emph{a priori} only $C^\infty$. One could just say that $(M,g)$ is asymptotically hyperbolic with real-analytic ends if $M$ is the interior of a compact real-analytic manifold with boundary $\overline{M}$ such that $g$ may be put into the form \eqref{eq:normal_intro}, with $g_1$ real-analytic, near $\partial \overline{M}$, using a real-analytic diffeomorphism between $[0,\epsilon[ \times \partial \overline{M}$ and a neighbourhood of $\partial \overline{M}$. This is for instance the assumption that is made in \cite{zuily_analytic}. However, it may seem \emph{a priori} too restrictive to assume the existence of such coordinates defined on a neighbourhood of the whole $\partial \overline{M}$. Consequently, we will rather make a local assumption on $g$ and then see that it implies that $g$ takes the form \eqref{eq:normal_intro} in real-analytic coordinates.

\begin{definition}\label{definition:manifold}
Let $M$ be a real-analytic manifold and $g$ be a smooth ($C^\infty$) Riemannian metric on $M$. We assume that $M$ is the interior of a compact real-analytic manifold with boundary $\overline{M}$. Assume that, for every $x_0 \in \partial \overline{M}$, there is a neighbourhood $U$ of $x_0$ in $\overline{M}$ and a real-analytic function $y_1$ from $U$ to $\mathbb{R}$ such that:
\begin{enumerate}[label=(\roman*)]
\item $y_1 \geq 0$ on $U$ and $\partial \overline{M} \cap U = \set{y_1 = 0}$;
\item $\mathrm{d} y_1 (x) \neq 0$ for every $x \in \partial \overline{M} \cap U$;
\item $y_1^2 g$ extends to a real-analytic metric $\tilde{g}$ on $U$;
\item $|\mathrm{d}y_1|_{\tilde{g}} = 1$ on $\partial \overline{M} \cap U$.
\end{enumerate}
Then, we say that $(M,g)$ is an asymptotically hyperbolic manifold real-analytic near infinity.
\end{definition}

A function that satisfies (i) and (ii) is called a boundary defining function for $\overline{M}$. Notice that if $y_1$ and $\tilde{y}_1$ are two real-analytic boundary defining functions, then there is a real-analytic real-valued function $f$, defined wherever $y_1$ and $\tilde{y}_1$ are both defined, and such that $\tilde{y}_1 = e^f y_1$. In particular, the validity of (iii) and (iv) does not depend on the choice of the boundary defining function $y_1$. One can check that if $(M,g)$ is an asymptotically hyperbolic manifold real-analytic near infinity, then it is also an asymptotically hyperbolic manifold in the standard ($C^\infty$) sense (see for instance \cite[Definition 5.2]{dyatlov_zworski_book}).

Let us fix an asymptotically hyperbolic manifold real-analytic near infinity $(M,g)$, and let $\overline{M}$ be as in Definition \ref{definition:manifold}. The existence of a real-analytic boundary defining function defined on a neighbourhood of $\partial \overline{M}$ does not seem obvious, and will be established in Lemma \ref{lemma:normal_coordinates} below. However, notice that one easily shows that there are $C^\infty$ boundary defining functions defined on the whole $\overline{M}$ and let us define the conformal class of Riemannian metrics on $\partial M$:
\begin{equation*}
[g]_{\partial \overline{M}} = \set{(y_1^2 g)|_{\partial \overline{M}}: y_1 \in C^\infty(\overline{M}) \textup{ is a boundary defining function}}.
\end{equation*}

It will be convenient to know that:

\begin{lemma}
The conformal class $[g]_{\partial \overline{M}}$ admits a real-analytic representative.
\end{lemma}

\begin{proof}
Let $g_0$ be any $C^\infty$ representative of $[g]_{\partial \overline{M}}$. Let $\hat{g}$ be a real-analytic Riemannian metric on $\partial \overline{M}$ (whose existence is guaranteed by \cite{Morrey_embedding}). For every $x \in \partial \overline{M}$, let $B(x)$ be the self-adjoint (for $\hat{g}(x)$) endomorphism of $T_x \partial M$ such that $g_0(x) = \hat{g}(x)(B(x) \cdot,\cdot)$. Let $g_1$ be the metric defined by $g_1(x) = g_0(x) / \n{B(x)}$, where the operator norm of $B(x)$ is defined using the metric $\hat{g}(x)$. From its very definition, $g_1$ is a representative of $[g]_{\partial M}$. Let us prove that $g_1$ is real-analytic.

Let $x_0 \in \partial \overline{M}$. From our assumption above (Definition \ref{definition:manifold}), there is a neighbourhood $V$ of $x_0$ in $\partial \overline{M}$ and a real-analytic metric $g_2$ on $V$ such that $g_2$ is conformal to $g_0$ on $V$. We have $g_0 = e^{2f} g_2$ for some $C^\infty$ function $f$ on $V$. For $x \in V$, we have
\begin{equation*}
\begin{split}
g_1 (x) & = \frac{g_0(x)}{\n{B(x)}} = \hat{g}(x)\p{\frac{B(x)}{\n{B(x)}} \cdot,\cdot} \\
    & = \hat{g}(x)\p{\frac{e^{-2f(x)} B(x)}{\n{e^{-2f(x)} B(x)}} \cdot,\cdot}.
\end{split}
\end{equation*}
On the other hand, for $x \in V$, we have
\begin{equation*}
g_2(x) = \hat{g}(x)(e^{-2f(x)} B(x) \cdot,\cdot).
\end{equation*}
And since $g_2$ and $\hat{g}$ are real-analytic, it follows that $x \mapsto e^{-2f(x)} B(x)$ is real-analytic on $V$, and thus so is $g_1$.
\end{proof}

We can then establish the existence of a real-analytic diffeomorphism on a neighbourhood of $\partial \overline{M}$ that puts the metric $g$ into the form \eqref{eq:normal_intro} (this is also known as a canonical product structure). The $C^\infty$ version of this result is standard, see for instance \cite[Theorem 5.4]{dyatlov_zworski_book}.

\begin{lemma}\label{lemma:normal_coordinates}
Let $g_0$ be a real-analytic representative of $[g]_{\partial \overline{M}}$. Then there is a real-analytic boundary function $y_1$ defined on a neighbourhood $U$ of $\partial \overline{M}$ such that
\begin{equation}\label{eq:condition_normal}
\va{\mathrm{d}y_1}_{y_1^2 g} = 1 \textup{ on a neighbourhood of } \partial \overline{M} \textup{ and } g_0 = (y_1^2 g)|_{\partial \overline{M}}.
\end{equation}
Moreover, there is a real-analytic map $y'$ from $U$ to $\partial \overline{M}$ such that $y'$ is the identity on $\partial \overline{M}$, the map $\Psi = (y_1,y')$ is a diffeomorphism from $U$ to $[0,\epsilon[ \times \partial \overline{M}$ for some $\epsilon > 0$, and the pushforward of $g$ under this map has the form
\begin{equation*}
(\Psi^{-1})^*g = \frac{\mathrm{d}y_1^2 + g_1(y_1,y',\mathrm{d}y')}{y_1^2},
\end{equation*}
where $g_1(y_1,y',\mathrm{d}y')$ is a real-analytic family of Riemannian metrics on $\partial \overline{M}$.
\end{lemma}

\begin{proof}
We start by constructing $y_1$ locally. Let $x_0 \in \partial \overline{M}$. Let $\tilde{y}_1$ be a real-analytic boundary function defined on a neighbourhood $U$ of $x_0$ as in Definition \ref{definition:manifold}. Up to multiplying $\tilde{y}_1$ by a real-analytic function, we may assume that $(\tilde{y}_1^2 g)_{|\partial \overline{M} \cap U} = g_0$. We want to construct $y_1$ on a neighbourhood of $x_0$ of the form $y_1 = e^f \tilde{y}_1$ with $f$ real-analytic that vanishes on $\partial \overline{M}$. The condition $\va{\mathrm{d}y_1}_{y_1^2 g} = 1$ may be rewritten as an eikonal equation, $F(x,\mathrm{d}f(x)) = 0$, non-characteristic with respect to $\partial \overline{M}$, like in \cite[(5.1.11) - (5.1.12)]{dyatlov_zworski_book}, which in our case has real-analytic coefficients. We can then use \cite[Theorem 1.15.3]{taylor_book} to find a (unique) solution $f$ to this equation near $x_0$, which happens to be real-analytic. Thus, we constructed a boundary defining function $y_1$ that satisfies \eqref{eq:condition_normal} near $x_0$.

Notice that if $y_1$ and $y_2$ are boundary defining functions that satisfy \eqref{eq:condition_normal} on open sets $U_1$ and $U_2$ of $\overline{M}$, then $y_1$ and $y_2$ coincide on all the connected components of $U_1 \cap U_2$ that intersect $\partial \overline{M}$. Indeed, we can write $y_1 = e^f y_2$ with $f$ that satisfies an eikonal equation as above and vanishes on $\partial \overline{M}$, and there is only one solution to this equation near $\partial \overline{M}$. We get the coincidence of $y_1$ and $y_2$ on the whole connected component of $U_1 \cap U_2$ by analytic continuation. 

We can consequently glue the local solutions to \eqref{eq:condition_normal} to get a solution defined on a neighbourhood of the whole $\partial \overline{M}$.

Finally, we construct the normal coordinates $(y_1,y')$ by integrating the gradient vector field $\nabla^{y_1^2g} y_1$ starting on $\partial \overline{M}$ as in the proof of \cite[Theorem 5.4]{dyatlov_zworski_book}.
\end{proof}

\begin{definition} \label{definition:even}
Using the notation from Lemma \ref{lemma:normal_coordinates}, we say that $(M,g)$ is even if for every integer $k$, we have
\begin{equation}\label{eq:evenness}
\partial_{y_1}^{2k+1} g_1(0,y',\mathrm{d}y') = 0.
\end{equation}
\end{definition}

From now on, we will always assume that $(M,g)$ satisfies the evenness assumption Definition \ref{definition:even}. Notice that Definitions \ref{definition:manifold} and \ref{definition:even} together are the hypotheses from Theorem \ref{theorem:main}. It is also worth noticing that the eveness assumption \eqref{eq:evenness} does not depend on the choice of the canonical product structure, see \cite[Theorem 5.6]{dyatlov_zworski_book}.

\subsection{Even extension}\label{subsec:even_extension}

We define an even extension $X$ for $M$ in the following way. We fix a canonical product structure $(y_1,y')$ on a neighbourhood $U \simeq [0,\epsilon[ \times \partial \overline{M}$ of $\partial \overline{M}$, as in Lemma \ref{lemma:normal_coordinates}. Let us define the real-analytic diffeomorphisms
\begin{equation*}
\psi_+ : U \cap M \to ]0,\epsilon^2[ \times \partial \overline{M}, \quad x \mapsto (y_1(x)^2,y'(x))
\end{equation*}
and
\begin{equation*}
\psi_- : U \cap M \to ]-1 - \epsilon^2,-1[ \times \partial \overline{M}, \quad x \mapsto (-1 - y_1(x)^2,y'(x)).
\end{equation*}

We let $X$ be the closed real-analytic manifold obtained by gluing $]-1-\epsilon^2,\epsilon^2[ \times \partial \overline{M}$ with two distinct copies of $M$ using the maps $\psi_-$ and $\psi_+$. We let $x_1$ be the function on $X$ given by the first coordinate in $]-1-\epsilon^2,\epsilon^2[ \times \partial \overline{M}$. Up to making $\epsilon$ smaller, we extend $x_1$ to a smooth function on $X$, real-analytic on $]-1-\epsilon^2,\epsilon^2[ \times \partial \overline{M}$, and such that $]-1-\epsilon^2,\epsilon^2[ \times \partial \overline{M} = \set{-1 - \epsilon^2 < x_1 < \epsilon^2}$.

The features of the even extension $X$ of $M$ in $\set{x_1 < 0}$ are somehow irrelevant: we are only concerned by the analysis in $\set{x_1 \geq 0}$ (but it is more convenient to work on a closed real-analytic manifold). In particular, we will identify $Y \coloneqq \set{x_1 > 0}$ with $M$. We will never do that with $\set{x_1 < - 1}$. Notice however that $\overline{Y} \subseteq X$ does not have the same smooth structure as $\overline{M}$ as defined above (the manifold $\overline{Y}$ is the even compactification of $M$).

Notice that the diffeomorphism $\psi_+ : U \cap M \to ]0,\epsilon^2[_{x_1} \times \partial \overline{M}_{x'}$ puts the metric $g$ into the form
\begin{equation*}
(\psi_+^{-1})^* g = \frac{\mathrm{d}x_1^2}{4 x_1^2} + \frac{g_1(\sqrt{x_1},x',\mathrm{d}x')}{x_1}.
\end{equation*}
It follows from our evenness assumption, Definition \ref{definition:even}, that the family $x_1 \mapsto g_1(\sqrt{x_1},x',\mathrm{d}x')$ of real-analytic metrics on $\partial \overline{M}$ has a real-analytic extension to $\set{- \zeta < x_1 < \zeta}$ for some $\zeta > 0$.

\subsection{The modified Laplacian}\label{subsection:modified_Laplacian}

Let $\eta > 0$ be smaller than $\zeta/2,\epsilon^2/2$ and $1$ (where $\zeta$ and $\epsilon$ are defined in the previous section), and choose a function $\rho : \mathbb{R} \to \mathbb{R}$ such that $\rho(x) = x$ for $\va{x} \leq \eta$ and $\rho(x) = \pm 3 \eta /2$ for $\va{x} \geq 2 \eta$ (where $\pm$ is the sign of $x$). Notice that we can choose $\rho$ such that $\rho'(x) x/\rho(x) \leq 1$ for positive $x$. Define then the function
\begin{equation*}
\tilde{x}_1 = \rho\p{\frac{4 x_1}{(1+x_1)^2}}
\end{equation*}
on $X$. For $\lambda \in \mathbb{C}$, let us consider the operator on $M \simeq Y$
\begin{equation}\label{eq:modified_operator}
\tilde{x}_1^{\frac{i \lambda}{2} - \frac{n+3}{4}}\p{- \Delta_g - \frac{(n-1)^2}{4} - \lambda^2} \tilde{x}_1^{\frac{n-1}{4} - \frac{i \lambda}{2}},
\end{equation}
where $\Delta_g$ is the (non-positive) Laplacian on $M$. Using $\psi_+$ to identify the set $\set{0 < x_1 < \eta}$ with $]0,\eta[_{x_1} \times \partial \overline{M}_{x'}$, we see that the operator \eqref{eq:modified_operator} takes the form 
\begin{equation}\label{eq:mal_de_tete}
\begin{split}
& - x_1 (1+x_1)^2 \partial_{x_1}^2 - \frac{(1+x_1)^2}{4} \Delta_{g_1}\\
   & \qquad \qquad + (1+x_1)\p{ (n-2 - i \lambda) x_1 + i \lambda - 1 - \gamma x_1 (1+x_1) } \partial_{x_1} \\
   & \qquad \qquad - \p{\frac{n-1}{2} - i \lambda}\Bigg( x_1\frac{n-1}{2} + i \lambda - 1 \\ & \qquad \qquad \qquad \qquad \qquad \qquad \qquad \qquad \quad - \gamma \frac{ (1+x_1)(1-x_1)}{2}\Bigg)
\end{split}
\end{equation}
there. Here $\Delta_{g_1}$ is the Laplacian for the metric $g_1(\sqrt{x_1},x',\mathrm{d}x')$ on $\partial \overline{M}$, the function $\gamma$ is the logarithmic derivative $J^{-1} \frac{\partial J}{\partial x_1}$ with respect to $x_1$ of the Jacobian $J$ of the metric $g_1(\sqrt{x_1},x',\mathrm{d}x')$ on $\partial \overline{M}$. The Jacobian $J$ may be defined by taking local coordinates on $\partial \overline{M}$. While $J$ depends on the choice of coordinates, the logarithmic derivative $\gamma$ does not. It follows from our evenness assumption that $\gamma$ extends to a real-analytic function on $\set{- \eta < x_1 < \eta} \subseteq X$. Notice that the expression \eqref{eq:mal_de_tete} extends real-analytically to $\set{- \eta < x_1 < \eta} \subseteq X$.

\begin{remark}
Here, we differ from the exposition in \cite[Chapter 5]{dyatlov_zworski_book} where, instead of \eqref{eq:modified_operator}, the operator
\begin{equation}\label{eq:modified_original}
x_1^{\frac{i \lambda}{2} - \frac{n+3}{4}}\p{- \Delta_g - \frac{(n-1)^2}{4} - \lambda^2} x_1^{\frac{n-1}{4} - \frac{i \lambda}{2}}
\end{equation}
is considered. This is an artificial modification that we introduce in order to be able to check assumption \ref{item:ellipticite_bas} from \S \ref{subsection:general_assumption}. The formula \eqref{eq:mal_de_tete} for \eqref{eq:modified_operator} can be deduced from the formula for \eqref{eq:modified_original} given in \cite[Lemma 5.10]{dyatlov_zworski_book}.
\end{remark}

Let $\chi : \mathbb{R} \to [0,1]$ be a smooth function such that $\chi(t) = 0$ for $t \leq -2\eta/3$ and $\chi(t) = 1$ for $t \geq - \eta/3$. Define then for $\lambda \in \mathbb{C}$ the differential operator $P(\lambda)$ on $X$ by
\begin{equation*}
P(\lambda) = \tilde{x}_1^{\frac{i \lambda}{2} - \frac{n+3}{4}}\p{- \Delta_g - \frac{(n-1)^2}{4} - \lambda^2} \tilde{x}_1^{\frac{n-1}{4} - \frac{i \lambda}{2}} \textup{ on } Y \simeq M,
\end{equation*}
\begin{equation*}
P(\lambda) = \chi(x_1)\times \eqref{eq:mal_de_tete} \textup{ on } \set{ - \eta < x_1 < \eta} ,
\end{equation*}
and
\begin{equation*}
P(\lambda) = 0 \textup{ on } \set{x_1 < - \frac{2\eta}{3}}.
\end{equation*}

Notice that the differential operator $P(\lambda)$ has real-analytic coefficients on the set $\set{ - \eta/3 < x_1 < \eta}$. Let us define for $\omega \in \mathbb{C}$ and $h > 0$ the semiclassical operator
\begin{equation*}
\mathcal{P}_h(\omega) = h^2 P(\omega/h).
\end{equation*}
Let us check that this family of operators satisfy the general assumptions from \S \ref{subsection:general_assumption}. We recall that the manifold $X$ and its open subset $Y$ have been defined at the end of \S \ref{subsec:even_extension}. It follows from \eqref{eq:mal_de_tete} that $\mathcal{P}_h(\omega)$ is of the form \eqref{eq:pinceau} with $P_0,P_1$ and $P_2$ that have real-analytic coefficients in the neighbourhood $\set{-  \eta/3 < x_1 <  \eta}$ of $\partial Y$. 

Let $p_j$ denote the principal symbol of $P_j$ for $j = 0,1,2$. For $x$ in the interior of $Y$, we have
\begin{equation*}
p_2(x,\xi) = \frac{(1+x_1)^2}{4 x_1} \va{\xi}_{g(x)}^2, \quad p_1(x,\xi) = - \frac{(1+x_1)^2}{4 x_1} \brac{\xi, \frac{\mathrm{d}\tilde{x}_1}{\tilde{x}_1}}_{g(x)}
\end{equation*}
and 
\begin{equation*}
p_0(x) = \frac{(1+x_1)^2}{4 x_1}\p{\va{\frac{\mathrm{d}\tilde{x}_1}{2 \tilde{x}_1}}_{g(x)}^2 - 1}.
\end{equation*}
Near $\partial Y$, we can express this symbols in the $(x_1,x')$ coordinates to find
\begin{equation*}
p_2(x_1,x,\xi_1,\xi') = x_1 (1+x_1)^2 \xi_1^2 + \frac{(1+x_1)^2}{4}\va{\xi'}_{g_1(\sqrt{x_1},x')}^2,
\end{equation*}
\begin{equation*}
p_1(x_1,x,\xi_1,\xi') = -(1+x_1)(1-x_1) \xi_1, \quad \textup{ and } \quad p_0(x_1,x') = - 1.
\end{equation*}

We are now in position to check that the assumptions from \S \ref{subsection:general_assumption} are satisfied. We see that \ref{item:structure_principal} holds with $w(x_1) = x_1(1+x_1)$ and $q_1(x_1,x',\xi') =  \frac{1+x_1}{4}\va{\xi'}_{g_1(\sqrt{x_1},x')}^2$. It is clear from the definition of $q_1$ that \ref{item:elliptique_bord} also holds. The validity of \ref{item:structure_sous_principal} and \ref{item:ellipticite_interieur} follows immediately from the formulae for $p_1(x_1,x,\xi_1,\xi')$ and $p_2(x,\xi)$ above.

It remains to prove \ref{item:ellipticite_bas}, that is that $p_0$ is negative on a neighbourhood of $\overline{Y}$. It is clear that $p_0$ is negative on a neighbourhood of $\partial Y$ from the formula above, so that we only need to check that 
\begin{equation*}
\va{\frac{\mathrm{d}\tilde{x}_1}{2 \tilde{x}_1}}_{g(x)} < 1
\end{equation*}
on the interior of $Y$.

Notice that we have
\begin{equation*}
\frac{\mathrm{d} \tilde{x}_1}{ 2 \tilde{x}_1} = \frac{\rho'\p{\frac{4 x_1}{(1+x_1)^2}}}{\rho\p{\frac{4 x_1}{(1+x_1)^2}}} \frac{4 x_1}{(1+x_1)^2} \frac{1 - x_1}{1+x_1} \frac{\mathrm{d}x_1}{2 x_1}.
\end{equation*}
Since $\va{\frac{\mathrm{d}x_1}{2 x_1}}_{g(x)} = 1$ when $0 < x_1 < 2 \eta$, we get
\begin{equation*}
\va{\frac{\mathrm{d}\tilde{x}_1}{2 \tilde{x}_1}}_{g(x)} = \va{\frac{\rho'\p{\frac{4 x_1}{(1+x_1)^2}}}{\rho\p{\frac{4 x_1}{(1+x_1)^2}}} \frac{4 x_1}{(1+x_1)^2}} \frac{1 - x_1}{1+x_1} \leq \frac{1 - x_1}{1+x_1},
\end{equation*}
and the validity of the assumption \ref{item:ellipticite_bas} follows.

\subsection{Upper bound on the number of resonances}\label{subsection:bound_resonances}

Since the assumption from \S \ref{subsection:general_assumption} are satisfied by the operator $\mathcal{P}_h(\omega)$ introduced in \S \ref{subsection:modified_Laplacian}, we may modify $\mathcal{P}_h(\omega)$ to get an operator $P_h(\omega)$ that satisfies Proposition \ref{proposition:general_statement}.

From here, the strategy to prove Theorem \ref{theorem:main} is the same as in \S \ref{subsection:quasinormal_frequencies}. We let $\kappa$ be as in Proposition \ref{proposition:general_statement} and choose a connected, relatively compact and open subset $V$ of $\set{z \in \mathbb{C} : \im z > - \kappa}$ that contains the closed disk of center $0$ and radius $\frac{3 \kappa}{4}$. We write $\iota_2$ for the inclusion of $C_c^\infty(M)$ in $\mathcal{H}_2$ and $\iota_1$ for the map obtained by composition of the inclusion of $\mathcal{H}_1$ in $\mathcal{D}'(X)$ and the restriction map $\mathcal{D}'(X) \to \mathcal{D}'(M)$.

For $\lambda \in h^{-1} V$, define the resolvent
\begin{equation*}
R_h(\lambda) = \tilde{x}_1^{\frac{n-1}{4} - \frac{i \lambda}{2}} \iota_1 h^{2} P_h(h \lambda)^{-1} \iota_2 \tilde{x}_1^{\frac{i \lambda}{2} - \frac{n+3}{4}} : C_c^\infty\p{M} \to \mathcal{D}'(M).
\end{equation*}

As in \S \ref{subsection:quasinormal_frequencies}, we get
\begin{lemma}\label{lemma:vrai_resolvante}
If $h$ is small enough, $\lambda$ is in $h^{-1}V$ and $\im \lambda > 0$, then $R_h(\lambda)$ coincides with the inverse of $- \Delta_g - \frac{(n-1)^2}{4} - \lambda^2$ on $L^2(M)$. In particular, $R_h(\lambda)$ does not depend on $h$ for $\lambda \in h^{-1} K$.
\end{lemma}

\begin{proof}
The proof is the same as for Lemma \ref{lemma:identification_resolvante}. One just needs to notice that if $\im \lambda > 0$ then the function $\tilde{x}_1^{\frac{n-1}{4} - \frac{i \lambda}{2}}$ belongs to $L^2(M)$.
\end{proof}

Notice that Lemma \ref{lemma:vrai_resolvante} implies that for $\lambda \in h^{-1} V$ the scattering resolvent $R_{\scat}(\lambda)$ coincides with $R_h(\lambda)$. With Proposition \ref{proposition:general_statement} and Lemma \ref{lemma:vrai_resolvante} at our disposal, the proof of Theorem \ref{theorem:main} follows exactly the same lines as the proof of Theorem \ref{theorem:schwarzschild} given in \S \ref{subsection:quasinormal_frequencies}. Consequently, we do not repeat it.

\section{Real-analytic Fourier--Bros--Iagolnitzer transform}\label{section:FBI_transform}

In this section, we detail the tools of real-analytic microlocal analysis that will be used in the proof of Proposition \ref{proposition:general_statement} in \S \ref{section:general_construction}. The main ingredient that we need is a real-analytic Fourier--Bros--Iagolnitzer transform as we studied in \cite{BJ20} with Guedes Bonthonneau. 

In \S \ref{subsec:generality}, we recall the main feature of such an FBI transform, and prove a slight generalization, Proposition \ref{proposition:toeplitz}, of \cite[Proposition 2.10]{BJ20}. In \S \ref{subsection:duality_statement}, we give a description, Proposition \ref{proposition:duality}, of the dual of a Hilbert space defined in \S \ref{subsec:generality}. This result will be useful to construct the injection of the spaces $\mathcal{H}_1$ and $\mathcal{H}_2$ in $\mathcal{D}'(X)$ in the proof of Proposition \ref{proposition:general_statement} (see Proposition \ref{proposition:inclusions_naturelles}) and to reuse results from \cite{BGJ} in \S \ref{subsec:logarithmic}, where we study the specificities of certain spaces defined using FBI transform and logarithmic weights (rather than weight of order $1$ as in \cite{BJ20}).

\subsection{Generality}\label{subsec:generality}

Let us recall the tools from \cite{BJ20} that we need for the proof of Proposition \ref{proposition:general_statement}. As in \S \ref{section:general_statement}, we let $X$ be a closed real-analytic manifold, and we endow it with a real-analytic metric $g_X$ (which is possible due to \cite{Morrey_embedding}). We endow $T^* X$ with an associated metric $g_{KN}$ which is given, using the decomposition into horizontal and vertical direction $T_\alpha(T^* X) \simeq T_{\alpha_x} X \oplus T^*_{\alpha_x} X \simeq T_{\alpha_x} X \oplus T_{\alpha_x} X$ for $\alpha = (\alpha_x,\alpha_\xi) \in T^*M$, by the formula 
\begin{equation*}
g_{KN,\alpha}((u,v),(u,v)) = g_{X,\alpha_x}(u,u) + \frac{g_{X,\alpha_x}(v,v)}{1 + g_{X,\alpha_x}(\alpha_\xi,\alpha_\xi)}
\end{equation*}
for $(u,v) \in T_{\alpha_x} X \oplus T_{\alpha_x} X$. This metric can be used to give a characterization of Kohn--Nirenberg symbols (see for instance \cite[Remark 2.5]{BJ20}), and we will consequently call it a Kohn--Nirenberg metric. Let $\widetilde{X}$ be a complexification of $X$ (endowed with any smooth distance) and $T^* \widetilde{X}$ its cotangent bundle. If $r > 0$ is small, we let $(X)_r$ denote the Grauert tube (see for instance \cite{grauert_tube_I,grauert_tube_II}) of size $r$ for $X$, that is the image of 
\begin{equation}\label{eq:boule}
\set{ (x,v) \in TX : g_{X,x}(v,v) \leq r^2}
\end{equation}
by the map
\begin{equation*}
(x,v) \mapsto \exp_x(iv),
\end{equation*}
which is well-defined on \eqref{eq:boule} if $r$ is small enough (here we use the holomorphic extension of the exponential map for $g_X$). We define similarly the Grauert tube $(T^*X)_r \subseteq T^* \widetilde{X}$ by using the Kohn--Nirenberg metric on $T^* X$. Because of the non-compactness of $T^* X$, it is not clear \emph{a priori} that $(T^* X)_r$ is well-defined. However, one can reduce the study of the Kohn--Nirenberg metric on $T^* X$ to its study near the zero section and the study of its pullbacks by the dilations $(\alpha_x,\alpha_\xi) \mapsto (\alpha_x,\lambda \alpha_\xi)$ for $\lambda \geq 1$ on a bounded subset of $T^* X$ (for instance the space between the spheres of radii $1$ and $2$ in each fiber). Since these pullbacks are uniformly analytic and positive definite, we see in particular that $(T^* X)_{r}$ is well-defined when $r$ is small enough.

Working in the holomorphic extension of real-analytic coordinates on $X$, we get a holomorphic trivialization $(\tilde{x},\tilde{\xi}) = (x+iy,\xi+i \eta)$ of $T^*\widetilde{X}$ in which $T^* X$ is described by $\set{y= \eta = 0}$. Using the same dilation trick as above, one may then check that, for every compact subset $K$ of the domain of the coordinate patch $\tilde{x}$, there is $C > 0$ such that, for every $r> 0$ small enough, the image of $(T^* X)_{r}$ above $K$ in this trivialization is intermediate between $$T^*_K \widetilde{X} \cap \set{\va{y} \leq C^{-1} r, \va{\eta} \leq C^{-1} \p{1 + \va{\xi}}  r}$$ and $$T_K^* \widetilde{X} \cap \set{\va{y} \leq C r, \va{\eta} \leq C \p{1 + \va{\xi} } r}.$$
Here, we write $T^*_K \widetilde{X}$ for the reciprocal image of $K$ by the canonical projection $T^* \widetilde{X} \to \widetilde{X}$.

If $m$ is a real number, $r > 0$ is small and $a$ is a smooth function on $(T^* X)_{r}$, we say that $a \in S_{KN}^m((T^* X)_{r})$ is a Kohn--Nirenberg symbol of order $m$ on $(T^* X)_{r}$ if, for every compact subset of the domain of a coordinate patch as above and every $k,k',\ell,\ell' \in \mathbb{N}^n$ there is a constant $C > 0$ such that on the image of $T^*_K \widetilde{X} \cap (T^* X)_{r}$ by the trivialization of $T^* \widetilde{X}$ associated to the coordinate patch, we have
\begin{equation*}
\va{\partial_x^k \partial_y^{k'} \partial_\xi^{\ell}\partial_\eta^{\ell'} a (\tilde{x},\tilde{\xi})} \leq C \p{1 + \va{\xi}}^m \p{1 + \va{\xi}}^{- \va{\ell} - \va{\ell'}}.
\end{equation*}
We define similarly symbols of logarithmic order by replacing $\p{1 + \va{\xi}}^m$ by $\log(2 + \va{\xi})$.

Let us fix a \emph{real} $C^\infty$ metric $\tilde{g}$ on the vector bundle $T^* \widetilde{X} \to X$ (seen as a real vector bundle) and define for $\alpha = (x,\xi) \in T^* \widetilde{X}$ the japanese bracket
\begin{equation*}
\brac{\va{\alpha}} = \sqrt{2 + \tilde{g}_x(\xi)}.
\end{equation*}
This is just a more convenient way to denote the size of $\alpha$ than taking the norm of $\xi$ directly, notice in particular that $\brac{\va{\alpha}}$ and $\log \brac{\va{\alpha}}$ are bounded from below. Notice that if $r > 0$ is small enough, then the function $\alpha \mapsto \brac{\va{\alpha}}$ is a Kohn--Nirenberg symbol of order $1$ on $(T^* X)_r$, as defined above.

It will also be useful to endow $T^* \widetilde{X}$ with a distance adapted to Kohn--Nirenberg symbols. One way to do that is to endow $T \widetilde{X}$ with a smooth Hermitian metric, which gives an identification of $T^* \widetilde{X}$ with $T \widetilde{X}$. Then, one may define a Kohn--Nirenberg metric on $T \widetilde{X}$ as above when $\widetilde{X}$, seen as a real manifold, is endowed with a smooth Riemannian metric (e.g. the real part of the Hermitian metric). We let $d_{KN}$ denote the associated distance. Restricting to a compact subset $K$ of $\widetilde{X}$, one may check that $\alpha,\beta \in T^*_K \widetilde{X}$ are close for $d_{KN}$ if their position variables are close to each other and, in local coordinates, their momentum variables have the same order of magnitude and the Euclidean distance between them is small with respect to this order of magnitude. This can be proved using a rescaling argument as described above.

For $R \gg 1$, so that $(X)_{\frac{1}{R}}$ is defined, we let $\widetilde{E}_R(X)$ denote the space of bounded holomorphic functions on the interior of $(X)_{\frac{1}{R}}$, endowed with the supremum norm. Then, we let $E_{R}(X)$ denote the closure of $\widetilde{E}_{R'}(X)$ in $E_{R}(X)$ for any $R'< R$ large enough so that $(X)_{\frac{1}{R'}}$ is well-defined. It follows from the Oka--Weil Theorem \cite[Theorems 2.3.1 and 2.5.2]{forstneric_book} that the space $E_R(X)$ does not depend on the choice of $R'$. Let $E_{R}'(X)$ denote the dual of $E_R(X)$, and notice that if $R > R'$ are such that $(X)_{\frac{1}{R}}$ and $(X)_{\frac{1}{R'}}$ are well-defined, then the injection of $E_{R'}(X)$ in $E_{R}(X)$ has dense image (because it contains $\widetilde{E}_{R''}(X)$ for some $R'' < R'$), so that the adjoint of this map defines an injection of $E_{R}'(X)$ into $E_{R'}'(X)$.

Then, we choose a real-analytic FBI transform $T : \mathcal{D}'(X) \to C^\infty(T^* X)$ on $X$, as defined in \cite[Definition 2.1]{BJ20}. This is a transform defined by a real-analytic kernel $K_T$:
\begin{equation*}
Tu(\alpha) = \int_X K_T(\alpha,x) u(x) \mathrm{d}x,
\end{equation*}
for $u \in \mathcal{D}'(X)$ and $\alpha \in T^* X$. Here, $\mathrm{d}x$ denotes the Lebesgue density associated to the Riemannian metric $g_X$ on $X$. The kernel $K_T$, and thus $T$, depends on the implicit semiclassical parameter $h > 0$ introduced in the beginning of \S \ref{subsection:general_assumption}. Unless the opposite is explicitly stated, all the estimates below will be uniform in $h$. The fact that $T$ is a real-analytic FBI transform \cite[Definition 2.1]{BJ20} means that the kernel $K_T$ has a holomorphic extension to $(T^* X)_r \times (X)_r$ for some small $r > 0$, which satisfies the following properties:
\begin{itemize}
\item for every $\delta > 0$, there is $r' > 0$ such that if $(\alpha,x) \in (T^* X)_{r'} \times (X)_{r'}$ are such that $d(\alpha_x,x) \geq \delta$ then
\begin{equation}\label{eq:small_away_diagonal}
\va{K_T(\alpha,x)} \leq (r')^{-1} \exp\p{- r' \frac{\brac{\va{\alpha}}}{h}};
\end{equation}
\item there is $\delta >0$ and $r' >0$ such that if $(\alpha,x) \in (T^* X)_{r'} \times (X)_{r'}$ are such that $d(\alpha_x,x) \leq \delta$ then
\begin{equation}\label{eq:local_behavior}
\va{K_T(\alpha,x) - e^{i \frac{\Phi_T(\alpha,x)}{h}} a(\alpha,x)} \leq (r')^{-1} \exp\p{- r' \frac{\brac{\va{\alpha}}}{h}}.
\end{equation}
\end{itemize}
Here, $a(\alpha,x)$ is an analytic symbol defined near the diagonal, elliptic in the class of $h^{-\frac{3n}{4}} \brac{\va{\alpha}}^{\frac{n}{4}}$, meaning that for $r',\delta > 0$ small enough, there is a constant $C > 0$ such that $a(\alpha,x)$ is holomorphic in $\set{(\alpha,x) \in (T^*X)_{r'} \times (X)_{r'} : d(\alpha_x,x) < \delta}$ and satisfies on that set the estimate
\begin{equation*}
 C^{-1} h^{- \frac{3n}{4}} \brac{\va{\alpha}}^{\frac{n}{4}} \leq \va{a(\alpha,x)} \leq C h^{- \frac{3n}{4}} \brac{\va{\alpha}}^{\frac{n}{4}}.
\end{equation*} 
The phase $\Phi_T (\alpha,x)$ from \eqref{eq:local_behavior} is an analytic symbol of order $1$ on the set $\set{(\alpha,x) \in (T^*X)_{r'} \times (X)_{r'} : d(\alpha_x,x) < \delta}$ (it is holomorphic and bounded by $C\brac{\alpha}$ for some $C > 0$), which satisfies in addition the following properties
\begin{itemize}
\item for $\alpha \in T^* X $, we have $\Phi_T(\alpha,\alpha_x) = 0$;
\item for $\alpha  \in T^* X$, we have $\mathrm{d}_x \Phi_T(\alpha,\alpha_x) = - \alpha_\xi$;
\item there is $C > 0$ such that, if $(\alpha,x) \in T^*X \times X$ and $d(\alpha_x,x) < \delta$, then
\begin{equation}\label{eq:imaginary_coercivity}
\im\p{\Phi_T(\alpha,x)} \geq C^{-1} \brac{\va{\alpha}} d(\alpha_x,x)^2.
\end{equation}
\end{itemize}

According to \cite[Theorem 6]{BJ20}, such a FBI transform exists. Moreover, if we endow $T^* X$ with the volume associated to the canonical symplectic form, then we may assume that the formal adjoint $S \coloneqq T^*$ of $T$ is a left inverse for $T$, i.e. that $T$ is an isometry on its image. Notice that $S$ has a real-analytic kernel $K_S$ that satisfies for $\alpha$ and $x$ real
\begin{equation*}
K_S(x,\alpha) = \overline{K_T(\alpha,x)}.
\end{equation*}
In particular, $K_S$ is negligible away from the diagonal, and may be described near the diagonal in a similar fashion as $K_T$.

Let us fix some small $r > 0$, and let $G_0$ be a Kohn--Nirenberg symbol of order $1$ on $(T^* X)_r$ and set $G = \tau G_0$ for some small $\tau > 0$ (the function $G$ is sometimes called an escape function). We let $\Lambda = \Lambda_G$ be the submanifold of $(T^* X)_r$ defined by
\begin{equation}\label{eq:defLambda}
\Lambda = e^{H_G^{\omega_I}} T^* X,
\end{equation}
where $H_G^{\omega_I}$ is the Hamiltonian vector field of $G$ for the symplectic form $\omega_I = \im \omega$, where $\omega$ denotes the canonical complex symplectic form on $T^* \widetilde{X}$. By taking $\tau$ small, we ensure that $\Lambda$ is $C^\infty$ close to $T^* X$ (this statement can be made uniform by pulling back $\Lambda$ to a bounded subset of $T^* \widetilde{X}$ using dilation in the fibers as above). Notice that in \cite[Definition 2.2]{BJ20}, the symbol $G_0$ was assumed to be supported in $(T^* X)_{r'}$ for some $r' < r$. The only reason for that was to ensure that the flow of $H_G^{\omega_I}$ is complete, which implies that \eqref{eq:defLambda} makes sense. However, taking $\tau$ small (which we will always do) is enough to ensure that \eqref{eq:defLambda} is well-defined. Moreover, we see that $\Lambda$ only depends on the values of $G$ on $(T^* X)_{r'}$ for some $r' < r$, so that the assumption on the support of $G_0$ from \cite[Definition 2.2]{BJ20} may be lifted without harm.

We will say that a smooth function $a$ on $\Lambda$ is a symbol of order $m \in \mathbb{R}$, and write $a \in S_{KN}^m(\Lambda)$, if the function $a \circ e^{H_G^{\omega_I}}$ is a symbol of order $m$, in the standard Kohn--Nirenberg class on $T^* X$. We define similarly symbols on $\Lambda \times \Lambda$.

On $\Lambda$, we can construct a real-valued symbol $H$ of order $1$ such that $\mathrm{d}H = -\im \theta$ where $\theta$ denotes the canonical complex $1$-form on $T^* \widetilde{X}$ (see \cite[\S 2.1.1]{BJ20}, in particular equation (2.9) there). Notice also that $\omega_R = \re \omega$ is a symplectic form on $\Lambda$ if $\tau$ is small enough. We let $\mathrm{d}\alpha = \omega_R^n /n!$ denote the associated volume form.

Notice that if $u \in E_R'(X)$ with $R$ large enough, then $Tu$ is well-defined and holomorphic on $(T^* X)_r$ for some small $r > 0$, so that if $\tau$ is small enough, $Tu$ is defined on $\Lambda$. We can consequently define the FBI transform $T_\Lambda$ associated to $\Lambda$ by restriction $T_\Lambda u = (T u)_{|\Lambda}$. Notice that since the kernel of $S$ is holomorphic, we also have an operator\label{page:Slambda} $S_\Lambda$ that is a left inverse for $T_\Lambda$ (see \cite[Lemma 2.7]{BJ20}). We will work with the spaces
\begin{equation*}
L_k^2(\Lambda) \coloneqq L^2(\Lambda, \brac{\va{\alpha}}^{2k} e^{- \frac{2H}{h}} \mathrm{d}\alpha)  \quad \textup{ for } \quad k \in \mathbb{R}
\end{equation*}
and
\begin{equation*}
\mathcal{H}_\Lambda^k \coloneqq \set{u \in E_{R}'(X) : T_\Lambda u \in L_k^2(\Lambda)}.
\end{equation*}
Here, $R$ needs to be large enough so that $E_R(X)$ is well-defined, and $\tau$ small enough depending on $R$ (but the particular choice of $R$ is irrelevant when $\tau$ is small). According to \cite[Corollary 2.2]{BJ20}, we know that $\mathcal{H}_\Lambda^k$ is a Hilbert space. We let also $\mathcal{H}_{\Lambda,\textup{FBI}}^k \subseteq L_k^2(\Lambda)$ denote the (closed) image of $\mathcal{H}_\Lambda^k$ by $T_\Lambda$. The structure of the projector $\Pi_\Lambda \coloneqq T_\Lambda S_\Lambda$ on the image of $T_\Lambda$ has been studied in \cite[\S 2.2]{BJ20}. The orthogonal projector $B_\Lambda$ on $\mathcal{H}_{\Lambda,\textup{FBI}}^0$ in $L_0^2(\Lambda)$ is studied in \cite[\S 2.3]{BJ20}. Notice that in order to prove Proposition \ref{proposition:general_statement}, we will work with a symbol $G_0$ which is of logarithmic order. As explained in \S \ref{subsec:logarithmic} (see also \cite{BGJ}), it implies that $\mathcal{H}_\Lambda^k$ is in fact a space of distributions. Consequently, we could have worked from the beginning only with distributions (and avoid the introduction of the space $E_R'(X)$). However, we decided to start from the context of \cite{BJ20} and then specify to the case of logarihtmic weights in \S \ref{subsec:logarithmic}. This is because we will need some extensions of the results from \cite{BJ20} that are not made easier by assuming that $G_0$ is of logarithmic order. It is also useful to see the case of logarithmic weights as a particular case of \cite{BJ20}, as it allows us to use the results from this reference.

Assume that $A(\alpha,\beta)$ is a smooth function on $\Lambda \times \Lambda$ and let $A$ be the associated operator
\begin{equation*}
A u (\alpha) = \int_\Lambda A(\alpha,\beta) u(\beta) \mathrm{d}\beta \textup{ for } \alpha \in \Lambda.
\end{equation*}
The operator $A$ may be defined for instance as an operator from the space of smooth compactly supported functions $u$ on $\Lambda$ to the space of smooth functions on $\Lambda$. In order to understand the action of $A$ on $L_0^2(\Lambda)$, one has to study the reduced kernel of $A$:
\begin{equation*}
A_{\textup{red}}(\alpha,\beta) = A(\alpha,\beta) e^{\frac{H(\beta) - H(\alpha)}{h}}.
\end{equation*}
To study the action of $A$ from $L_k^2(\Lambda)$ to $L_\ell^2(\Lambda)$, one can study the kernel $A_{\textup{red}}(\alpha,\beta) \brac{\va{\beta}}^{- 2 k} \brac{\va{\alpha}}^{2\ell}$. We will say that the kernel $A$ is negligible if
\begin{equation}\label{eq:reduced_kernel_small}
A_{\textup{red}}(\alpha,\beta) = \mathcal{O}_{C^\infty}\p{h^\infty (\brac{\va{\alpha}} + \brac{\va{\beta}})^{- \infty}}.
\end{equation}
Here, the $C^\infty$ estimates may be understood by identifying $\Lambda$ with $T^* X$ using $e^{H_G^{\omega_I}}$, taking a trivialization for $T^* X$ and then asking for all partial derivatives of $A_{\textup{red}}$ to be $\mathcal{O}(h^\infty (\brac{\va{\alpha}} + \brac{\va{\beta}})^{- \infty})$. We do not need to ask for symbolic estimates in that case, as it is automatic for something that decays that fast. Notice that an operator whose reduced kernel satisfies \eqref{eq:reduced_kernel_small} is bounded from $L_k^2(\Lambda)$ to $L_\ell^2(\Lambda)$ for every $k,\ell \in \mathbb{R}$, with norm $\mathcal{O}(h^\infty)$. An operator whose reduced kernel satisfy \eqref{eq:reduced_kernel_small} will be called a negligible operator. 

Recall the phase $\Phi_{TS}(\alpha,\beta)$ from \cite[\S 2.2]{BJ20}, which is the critical value of $y \mapsto \Phi_T(\alpha,y) + \Phi_S(y,\beta)$. Here, $\Phi_S$ is the phase that appear when describing the kernel $K_S(y,\beta)$ of $S$ locally as we do for $K_T$ in \eqref{eq:local_behavior}. That is $\Phi_S(y,\beta) = -\overline{\Phi_T(\bar{\beta},\bar{y})}$. The following fact follows from the analysis in \cite{BJ20}.

\begin{lemma}\label{lemma:factorisation}
Let $\delta > 0$ be small enough. Assume that $\tau > 0$ and $h > 0$ are small enough. Assume that $A(\alpha,\beta)$ is a smooth function on $\Lambda \times \Lambda$ and let $A$ be the associated operator. Let $m \in \mathbb{R}$. Assume that there is a symbol $a \in S_{\textup{KN}}^m(\Lambda \times \Lambda)$ supported in $\set{(\alpha,\beta) \in \Lambda \times \Lambda, d_{KN}(\alpha,\beta) < \delta}$ such that
\begin{equation}\label{eq:desired_reduced}
A_{\textup{red}}(\alpha,\beta) =\frac{1}{(2\pi h)^n} e^{\frac{H(\beta) + i \Phi_{TS}(\alpha,\beta) - H(\alpha)}{h}} a(\alpha,\beta) + \mathcal{O}_{C^\infty}\p{h^\infty (\brac{\va{\alpha}} + \brac{\va{\beta}})^{- \infty}}.
\end{equation}
Then, $A$ is bounded from $L_k^2(\Lambda)$ to $L_{k-m}^2(\Lambda)$ for every $k \in \mathbb{R}$, and there is a symbol $\sigma \in S_{KN}^m(\Lambda)$ such that the operators $B_\Lambda A B_\Lambda$ and $B_\Lambda \sigma B_\Lambda$ differ by a negligible operator. 

Moreover, $\sigma$ coincides with  $\alpha \mapsto a(\alpha,\alpha)$ up to $\mathcal{O}(h)$ in $S_{\textup{KN}}^{m-1}(\Lambda)$.
\end{lemma}

Indeed, the boundedness statement follows from the proof of \cite[Proposition 2.4]{BJ20}. Our assumption on the kernel of $A$ implies that $A$ belongs to the class of FIO from \cite[Definition 2.5]{BJ20}, and thus the proof of \cite[Proposition 2.10]{BJ20} may be rewritten replacing the operator ``$f T_\Lambda PS_\Lambda$'' by the operator $A$. This gives the symbol $\sigma$ such that $B_\Lambda A B_\Lambda - B_\Lambda \sigma B_\Lambda$ is a negligible operator. The proof gives that $\sigma$ coincides with $\alpha \mapsto g_0(\alpha) a (\alpha,\alpha)$ for a symbol $g_0$ of order $0$ that does not depend on $A$. To see that one can take $g_0 = 1$, just notice that the operator $\Pi_\Lambda = T_\Lambda S_\Lambda$ satisfies the hypotheses from Lemma \ref{lemma:factorisation} with $\alpha \mapsto a(\alpha,\alpha)$ identically equal to $1$ up to $\mathcal{O}(h)$ in $S_{KN}^{-1}(\Lambda)$, according to \cite[Lemma 2.10]{BJ20}, and that $B_\Lambda \Pi_\Lambda B_\Lambda = B_\Lambda B_\Lambda$. Moreover, one may retrieve the leading part of a symbol $\sigma$ from restriction to the diagonal of the kernel of the operator $B_\Lambda \sigma B_\Lambda$ (the kernel may be computed by the stationary phase method as in \cite[Lemma 2.16]{BJ20}).

We need to extend certain results from \cite{BJ20} to a slightly more general context in order to prove Proposition \ref{proposition:general_statement}. Let $P$ be a semiclassical differential operator of order $m$ with $C^\infty$ coefficients and let $p$ be the principal symbol of $P$. We make the following assumption
\begin{equation}\label{eq:assumption}
\begin{split}
& \textup{ for every } x \in X \textup{ either } G_0(y,\xi)=0  \textup{ for every } y \textup{ near } x \\ & \textup{ and } \xi \in T^*_y X, \textup{ or } P \textup{ has real-analytic coefficients near } x.
\end{split}
\end{equation}
Notice that under the assumption \eqref{eq:assumption} the principal symbol $p$ of $P$ may be restricted to $\Lambda$ provided $\tau$ is small enough. Indeed, for every $x \in X$, either $p$ has a holomorphic extension near $T^*_x X$ or $\Lambda$ coincides with $T^* X$ near $T^*_x X$. We let $p_\Lambda$ denote this restriction. If $P$ is an operator that satisfies \eqref{eq:assumption}, we may define $T_\Lambda P S_\Lambda$ as the operator with kernel
\begin{equation}\label{eq:def_kernel_TPS}
T_\Lambda P S_\Lambda(\alpha,\beta) = \int_M K_T(\alpha,y) P_y\p{K_S(y,\beta)} \mathrm{d}y.
\end{equation}
The reason for which we use this definition is because since $P$ is \emph{a priori} not an operator with real-analytic coefficients, it is not straightforward to define the action of $P$ on elements of $E_R'(X)$. Notice that the following result allows to define $P$ as an operator from $\mathcal{H}_\Lambda^k$ to $\mathcal{H}_\Lambda^{k - m}$. When we will specify to the case of logarithmic weights in \S \ref{subsec:logarithmic}, the spaces $\mathcal{H}_\Lambda^k$'s will be included in $\mathcal{D}'(M)$, and the natural relation $T_\Lambda P u = T_\Lambda P S_\Lambda T_\Lambda u$ will be satisfied, see Lemma \ref{lemma:seems_simple}.

\begin{proposition}\label{proposition:toeplitz}
Under the assumption \eqref{eq:assumption}, if $\tau$ is small enough, then the operator $T_\Lambda P S_\Lambda$ is bounded from $L_k^2(\Lambda)$ to $L_{k-m}^2(\Lambda)$. Moreover, if $\ell \in \mathbb{R}$ and $f \in S_{KN}^\ell(\Lambda)$, there is a symbol $\sigma \in S_{KN}^{m + \ell}(\Lambda)$ and an operator $L$ with negligible kernel such that
\begin{equation*}
B_\Lambda f T_\Lambda P S_\Lambda B_\Lambda = B_\Lambda \sigma B_\Lambda + L.
\end{equation*}
In addition, $\sigma$ coincides with $f p_\Lambda$ up to $\mathcal{O}(h\brac{\va{\alpha}}^{m+ \ell-1})$. 
\end{proposition}

The proof of Proposition \ref{proposition:toeplitz} is based on applications of the stationary and non-stationary phase methods with complex phase. We will apply both the $C^\infty$ and the holomorphic versions of these methods. We are not aware of a reference stating the $C^\infty$ version of the non-stationary phase method with complex phase that would cover all the cases we are going to consider (for the stationary phase method, see \cite{melin_sjostrand}, and for a standard version of the non stationary phase method with complex phase, see \cite[Theorem 7.7.1]{horm1}), so that we prove here a statement adapted to our needs. This result and its proof should be no surprise for specialists.

\begin{lemma}\label{lemma:non_stationary_C_infty}
Let $m,n$ be integer. Let $U,V$ be open subsets respectively of $\mathbb{R}^m$ and $\mathbb{R}^n$. Let $\Phi : U \times V \to \mathbb{C}$ be a $C^\infty$ function. Let $K_1$ and $K_2$ be compact subsets respectively of $U$ and $V$. Assume that for every $(x,y) \in K_1 \times K_2$ we have $\im \Phi(x,y) \geq 0$ and $\mathrm{d}_y \Phi(x,y) \neq 0$. Then, for every $L,N > 0$, there are constants $k \in \mathbb{N}$ and $\lambda_0 > 0$ such that for every $\lambda \geq \lambda_0$, every $C^k$ function $u$ supported in $K_2$ and every $x \in U$ such that $d(x,K_1) \leq L \log \lambda / \lambda$, we have
\begin{equation*}
\va{\int_V e^{i \lambda \Phi(x,y)} u(y) \mathrm{d}y} \leq \lambda^{-N} \n{u}_{C^k}.
\end{equation*}
\end{lemma}

\begin{proof}
Let $x \in U$ be such that $d(x,K_1) \leq L \log \lambda / \lambda$. From our non-stationary assumption, we see that if $\lambda$ is large enough then $\mathrm{d}_y \Phi(x,y) \neq 0$ for every $y \in K_2$. We can consequently introduce the differential operator
\begin{equation*}
L_x = -i \sum_{j = 1}^n \frac{\overline{\partial_{y_j} \Phi(x,y)}}{\va{\nabla_y \Phi(x,y)}^2} \partial_{y_j},
\end{equation*}
and notice that $L_x(e^{i \lambda \Phi(x,y)}) = \lambda e^{i \lambda \Phi(x,y)}$. Letting $k$ be a large integer and ${}^t L_x$ denote the formal adjoint of $L_x$, we find that
\begin{equation*}
\int_V e^{i \lambda \Phi(x,y)} u(y) \mathrm{d}y = \lambda^{-k} \int_{V} e^{i \lambda \Phi(x,y)} {}^t L_x^k u (y) \mathrm{d}y.
\end{equation*}
Then, we notice that the $L^\infty$ norm of ${}^t L_x^k u$ is controlled by the $C^k$ norm of $u$. Moreover, since $d(x,K_1) \leq L \log \lambda / \lambda$, we find that for every $y \in K_2$, if $\lambda$ is large enough, we have $\im \Phi(x,y) \geq - C_\Phi L \log \lambda / \lambda$ for some constant $C_\Phi$ that does not depend on $k$ nor $u$. Consequently, we have for $\lambda$ large:
\begin{equation*}
\va{\int_V e^{i \lambda \Phi(x,y)} u(y) \mathrm{d}y} \leq C \lambda^{-k+C_\Phi L} \n{u}_{C_k}.
\end{equation*}
Here the constant $C$ may depend on $k$ and $\Phi$, but not on $\lambda$ nor $u$. Taking $k$ large enough, we ensure that $k - C_\Phi L > N$ and the result follows.
\end{proof}

We have now at our disposal all the tools to prove Proposition \ref{proposition:toeplitz}.

\begin{proof}[Proof of Proposition \ref{proposition:toeplitz}]
We want to apply Lemma \ref{lemma:factorisation} to the operators $T_\Lambda P S_\Lambda$ and $f T_\Lambda P S_\Lambda$. Let us introduce the open sets
\begin{equation*}
U_1 = \set{x \in X : G_0(y,\xi) = 0 \textup{ for every } y \textup{ near } x \textup{ and } \xi \in T_y^* X}
\end{equation*}
and
\begin{equation*}
U_2 = \set{ x \in X : P \textup{ has real-analytic coefficients near } x}.
\end{equation*}
By assumption $X = U_1 \cup U_2$. We start by proving that for every $\delta > 0$, provided $\tau$ is small enough, we have
\begin{equation}\label{eq:undiagonal}
T_\Lambda P S_\Lambda(\alpha,\beta)e^{\frac{H(\beta)-H(\alpha)}{h}} = \mathcal{O}\p{h^\infty (\brac{\va{\alpha}} + \brac{\va{\beta}})^{- \infty}}
\end{equation}
whenever $\alpha,\beta \in \Lambda$ are such that $d_{KN}(\alpha,\beta) \geq \delta$. Let us write $\alpha = e^{H_G^{\omega_I}}(x,\xi)$ and $\beta = e^{H_G^{\omega_I}} (y,\eta)$ where $(x,\xi)$ and $(y,\eta)$ are in $T^* X$. Assume first that $x$ and $y$ are at distance larger than $\delta/L$ for some large constant $L \gg 1$. We can then write
\begin{equation}\label{eq:decomposition_kernel_undiagonal}
\begin{split}
& T_\Lambda P S_\Lambda(\alpha,\beta) \\ & \quad = \p{\int_{D(x,\delta/10L)} + \int_{D(y,\delta/10L)} + \int_{X \setminus (D(x,\delta/10L) \cup D(y,\delta/10L)}} \\ & \qquad \qquad \qquad \qquad \qquad \qquad \qquad \qquad \qquad \qquad K_T(\alpha,z) P_z\p{K_S(z,\beta)} \mathrm{d}z.
\end{split}
\end{equation}
We write $D(w,r)$ for the ball of center $w$ and radius $r$ in $X$. Notice that, provided $\tau$ is small enough, the third integral in \eqref{eq:decomposition_kernel_undiagonal} is $\mathcal{O}\p{\exp\p{- \frac{\brac{\va{\alpha}} + \brac{\va{\beta}}}{C h}}}$ since the kernel $K_T$ and $K_S$ are negligible away from the diagonal \eqref{eq:small_away_diagonal}. Since $e^{\frac{H(\beta) - H(\alpha)}{h}}$ is $\mathcal{O}\p{\exp\p{C \tau (\brac{\va{\alpha}} + \brac{\va{\beta}})}}$, we see that for $\tau$ small enough we have
\begin{equation*}
\begin{split}
& e^{\frac{H(\beta) - H(\alpha)}{h}} \int_{X \setminus (D(x,\delta/10L) \cup D(y,\delta/10L)} K_T(\alpha,z) P_z\p{K_S(z,\beta)} \mathrm{d}z \\ & \qquad \qquad \qquad \qquad \qquad \qquad \qquad \qquad \qquad = \mathcal{O}\p{\exp\p{- \frac{\brac{\va{\alpha}} + \brac{\va{\beta}}}{C h}}} \\ & \qquad \qquad \qquad \qquad \qquad \qquad \qquad \qquad \qquad = \mathcal{O}\p{h^\infty (\brac{\va{\alpha}} + \brac{\va{\beta}})^{- \infty}},
\end{split}
\end{equation*}
and we only need to care about the two other terms.

Let us deal with the first term in \eqref{eq:decomposition_kernel_undiagonal}. Up to a negligible term, it is given by
\begin{equation}\label{eq:premier_terme}
\int_{D(x,\delta/10L)} e^{i \frac{\Phi_T(\alpha,z)}{h}} a(\alpha,z) P_z(K_S(z,\beta)) \mathrm{d}z.
\end{equation}
By taking $L$ large enough, we have either $D(x,\delta/10L) \subseteq U_1$ or $D(x,\delta/10L) \subseteq U_2$.

Let us begin with the case of $D(x,\delta/10L) \subseteq U_1$. In that case, the differential operator $P$ has \emph{a priori} only $C^\infty$ coefficients on $D(x,\delta/10L)$ so that we find that $P_z(K_S(z,\beta))$ is $\mathcal{O}\p{\exp(- \brac{\va{\beta}}/Ch)}$ in $C^\infty$. Notice also that $\mathrm{d}_y \Phi_T(\alpha,\alpha_x) = - \alpha_\xi$ and that the imaginary part of $\Phi_T(\alpha,z)$ is non-negative when $z \in D(x,\delta/10L)$. Hence, provided $L$ is large enough, we can use the $C^\infty$ non-stationary phase method (apply Lemma \ref{lemma:non_stationary_C_infty} with a rescaling argument) to find that \eqref{eq:premier_terme} is an $$\mathcal{O}\p{h^\infty \brac{\va{\alpha}}^{- \infty}\exp(- \brac{\va{\beta}}/Ch)}.$$ Here, the integrand is not supported away from the boundary of the domain of integration, but since the imaginary part of the phase is larger than $C^{-1} \brac{\va{\alpha}}/h$ near the boundary of the domain of integration, we may just introduce a bump function to fix that. The same trick allows to remove the dependence on $x$ of the domain of integration. Using that $x \in U_1$, we find that $\alpha = (x,\xi)$ and that $H(\alpha) = 0$ (see \cite[(2.9)]{BJ20}), so that $e^{\frac{H(\beta) - H(\alpha)}{h}} = e^{\frac{H(\beta)}{h}} = \mathcal{O}\p{\exp\p{C \tau \frac{\brac{\va{\beta}}}{h}}}$. Hence, for $\tau$ small enough, we find that
\begin{equation}\label{eq:reduced_premier_terme}
\begin{split}
e^{\frac{H(\beta) - H(\alpha)}{h}}\int_{D(x,\delta/10L)} e^{i \frac{\Phi_T(\alpha,z)}{h}} a(\alpha,z) & P_z(K_S(z,\beta)) \mathrm{d}z \\ & \quad \qquad = \mathcal{O}\p{h^\infty (\brac{\va{\alpha}} + \brac{\va{\beta}})^{- \infty}}.
\end{split}
\end{equation}

When $D(x,\delta/10L) \subseteq U_2$, the coefficients of $P$ are analytic, and $P_z(K_S(z,\beta))$ is $\mathcal{O}\p{\exp(- \brac{\va{\beta}}/Ch)}$ as a real-analytic function. Hence, provided $L$ is large enough, we can use the holomorphic non-stationary phase method (see for instance \cite[Proposition 1.1]{BJ20}, and use a rescaling argument) as in the proof of \cite[Lemma 2.9]{BJ20} to see that \eqref{eq:premier_terme} is $\mathcal{O}\p{\exp\p{- \frac{\brac{\va{\alpha}}+ \brac{\va{\beta}}}{h}}}$, provided $\tau$ is small enough. Hence, if $\tau$ is small enough, this is enough to beat the potential growth of the factor $e^{\frac{H(\beta) - H(\alpha)}{h}}$, so that we also have \eqref{eq:reduced_premier_terme} in that case.

We deal similarly with the second term in \eqref{eq:decomposition_kernel_undiagonal}, distinguishing the cases $D(y,\delta/10L) \subseteq U_1$ and $D(y,\delta/10L) \subseteq U_2$.

Let us now prove \eqref{eq:undiagonal} when the distance between $x$ and $y$ is less than $\delta/L$ (and consequently $\xi$ and $\eta$ are away from each other in a trivialization of $T^* X$). As above, we can discard the $z$'s that are away from $x$ (and thus from $y$) and write up to a negligible term the kernel of $T_\Lambda P S_\Lambda$ as (the error term coming from the approximation \eqref{eq:local_behavior} is dealt with by an application of the non-stationary pahse method as in the previous case)
\begin{equation}\label{eq:almost_diagonal}
\int_{D(x,10\delta/L)} e^{i \frac{\Phi_T(\alpha,z) + \Phi_S(z,\beta)}{h}} a(\alpha,z) \tilde{b}(z,\beta) \mathrm{d}z,
\end{equation} 
where the symbol $\tilde{b}$ is defined by
\begin{equation*}
\tilde{b}(z,\beta) = e^{- i \frac{\Phi_S(z,\beta)}{h}} P_z\p{e^{i \frac{\Phi_S(z,\beta)}{h}} b(z,\beta)}.
\end{equation*}
Notice that the phase in \eqref{eq:almost_diagonal} is holomorphic and non-stationary. Indeed, working in coordinates and assuming that $L$ is large enough, we find that for some $C > 0$ and every $z \in D(x,10\delta/L)$:
\begin{equation*}
\begin{split}
\va{\nabla_z\p{\Phi_T(\alpha,z) + \Phi_S(z,\beta)}} & = \va{\beta_\xi - \alpha_\xi} + \mathcal{O}\p{\frac{\max(\brac{\va{\alpha}}, \brac{\va{\beta}})}{L}} \\
    & \geq C^{-1} \max(\brac{\va{\alpha}},\brac{\va{\beta}}).
\end{split}
\end{equation*}
Moreover, provided $\tau$ is small enough, the imaginary part of the phase is larger than $C^{-1} \max(\brac{\va{\alpha}},\brac{\va{\beta}})$ when $z$ is on the boundary of $D(x,10\delta/L)$ (because $z$ is away from $\alpha_x$ and $\beta_x$), and is always non-negative when $D(x,10\delta/L) \subseteq U_1$. We can apply the $C^\infty$ non-stationary phase method when $D(x,10\delta/L) \subseteq U_1$ and the holomorphic non-stationary phase method when $D(x,10\delta/L) \subseteq U_2$ (for this second case, see the similar computation in the proof of \cite[Lemma 2.9]{BJ20}). Indeed, in the latter case $\tilde{b}$ is holomorphic in $z$, while in the first case it is only $C^\infty$. In the first case, we get that \eqref{eq:almost_diagonal} is $\mathcal{O}(h^\infty (\brac{\va{\alpha}} + \brac{\va{\beta}})^{- \infty})$ and in the second case that it is $\mathcal{O}\p{\exp\p{- \frac{\brac{\va{\alpha}} + \brac{\va{\beta}}}{Ch}}}$. Noticing that in the first case $H(\alpha) = H(\beta) = 0$, we find that \eqref{eq:undiagonal} holds.

Notice that differentiating the kernel of $K_T$ or of $K_S$ (in a local trivialization of $T^* X$) amount to replace the symbols $a$ and $b$ by symbols of higher orders (in terms of $\alpha,\beta$ and $h$). Thus, all the estimates that we established when $\alpha$ and $\beta$ are away from each other actually hold in $C^\infty$.

We must now understand what happens when $\alpha$ and $\beta$ are close to each other. We write as above $\alpha = e^{H_G^{\omega_I}}(x,\xi)$ and $\beta = e^{H_G^{\omega_I}} (y,\eta)$ where $(x,\xi)$ and $(y,\eta)$ are in $T^* X$. Then, up to negligible terms, the kernel of $T_\Lambda P S_\Lambda$ at $(\alpha,\beta)$ is given as above, for some small $\delta > 0$, by
\begin{equation*}
\int_{D(x,\delta)} e^{i \frac{\Phi_T(\alpha,z) + \Phi_S(z,\beta)}{h}} a(\alpha,z) \tilde{b}(z,\beta) \mathrm{d}z.
\end{equation*}
As above, the error coming from the approximation \eqref{eq:local_behavior} is dealt with by an application of the non-stationary phase method. The asymptotic of this integral when $\brac{\va{\alpha}}/h$ tends to $+ \infty$ is given by the stationary phase method. Indeed, when $\alpha = \beta$, the rescaled phase $y \mapsto (\Phi_T(\alpha,y) + \Phi_S(y,\beta))/\brac{\va{\alpha}}$ has a uniformly non-degenerate critical point at $y = \alpha_x = \beta_x$, as a consequence of \eqref{eq:imaginary_coercivity}. Moreover, when $D(x,\delta) \subseteq U_1$, the imaginary part of this phase is non-negative on $D(x,\delta)$, provided the distance between $\alpha_x$ and $\beta_x$ is way smaller than $\delta$. When $D(x,\delta) \subseteq U_2$, we may ensure that the imaginary part of the (rescaled) phase is uniformly positive on the boundary of $D(x,\delta)$ by taking $\tau$ small enough. As above, we apply the stationary phase method in the $C^\infty$ category (see \cite[\S 2]{melin_sjostrand}) when $D(x,\delta) \subseteq U_1$ and in the $C^\omega$ category when $D(x,\delta) \subseteq U_2$ (see \cite[\S 2]{sjostrand_asterisque} for the general method and the proof of \cite[Lemma 2.10]{BJ20} in the case $s = 1$, page 111, for the details of the computation in our particular setting). In both cases, we can use the fact that the imaginary part of the phase is positive on the boundary of the domain of integration to remove the dependence of this domain on $x$. In the first case we get an expansion with an error term of the form $\mathcal{O}(h^\infty \brac{\va{\alpha}}^{-\infty})$ and in the second case of the form $\mathcal{O}\p{\exp\p{- \brac{\va{\alpha}}/h}}$. Since in the first case we have $H(\alpha) = H(\beta) = 0$, we see that in both cases we get the desired expansion \eqref{eq:desired_reduced} for the reduced kernel of $T_\Lambda P S_\Lambda$, with an error term of the required size. 

We can then apply Lemma \ref{lemma:factorisation} to end the proof. Indeed, we just saw that the kernel of $f T_\Lambda P S_\Lambda$ is of the form \eqref{eq:desired_reduced}. Moreover, it follows from the application of the stationary phase method that, up to $\mathcal{O}(h)$ in $S_{KN}^{m-1}\p{\Lambda}$, the symbol $\alpha \mapsto a(\alpha,\alpha)$ coincides with $f p_\Lambda g_0$, where $g_0$ is a symbol of order $0$ that does not depend on $P$. Thus, the operator $f T_\Lambda P S_\Lambda - f p_\Lambda \Pi_\Lambda$ is also of the form \eqref{eq:desired_reduced} but with an $a$ such that $a \mapsto a(\alpha,\alpha)$ is $\mathcal{O}(h)$ in $S_{KN}^{m+ \ell-1}\p{\Lambda}$. Consequently, there is a symbol $\tilde{\sigma} \in h S_{KN}^{m+ \ell-1} (\Lambda)$ such that $B_\Lambda(f T_\Lambda P S_\Lambda - f p_\Lambda \Pi_\Lambda)B_\Lambda - B_\Lambda \tilde{\sigma} B_\Lambda = B_\Lambda f T_\Lambda P S_\Lambda B_\Lambda - B_\Lambda (f p_\Lambda + \tilde{\sigma}) B_\Lambda$ is a negligible operator. We get the announced result with $\sigma = f p_\Lambda + \tilde{\sigma}$.
\end{proof}

\subsection{Duality statement}\label{subsection:duality_statement}

In \cite[Lemma 2.24]{BJ20}, an identification between $\mathcal{H}_{\Lambda}^{-k}$ and the dual of $\mathcal{H}_\Lambda^k$ is given. However, the pairing used to define this identification is not the $L^2$ pairing. We explain here how to describe the dual of $\mathcal{H}_\Lambda^k$ using the $L^2$ pairing. This will allow us in particular to reuse results from \cite{BGJ} in \S \ref{subsec:logarithmic}.

Let us first recall that there is an anti-holomorphic involution $\alpha \mapsto \bar{\alpha}$ on $(T^* X)_r$ such that $\set{\alpha \in (T^* X)_r: \alpha = \bar{\alpha}} = T^* X$, see \cite{grauert_tube_I}. Let $G$ be a symbol of order $1$ on $(T^* X)_r$ as above (of the form $G = \tau G_0$ with $\tau$ small) and $\Lambda$ be defined by \eqref{eq:defLambda}. Let us introduce a new symbol $G^*(\alpha) = - G(\bar{\alpha})$, and notice that the Lagrangian associated to $G^*$ by \eqref{eq:defLambda} is $\overline{\Lambda}$, that is the image of $\Lambda$  by the involution $\alpha \mapsto \bar{\alpha}$. Notice also that changing $G$ to $G^*$, we have to replace $H$ by the function $H^*$ on $\overline{\Lambda}$ given by $H^*(\alpha) = - H(\bar{\alpha})$.

Consequently, if $u \in \mathcal{H}_\Lambda^k$ and $v \in \mathcal{H}_{\overline{\Lambda}}^{-k}$, we may define the pairing
\begin{equation}\label{eq:pairing}
\langle u,v \rangle = \int_\Lambda T_\Lambda u(\alpha) \overline{T_{\overline{\Lambda}} v(\bar{\alpha})} \mathrm{d}\alpha,
\end{equation}
for which we can prove:

\begin{proposition}\label{proposition:duality}
Let $R \gg 1$. Assume that $\tau$ is small enough. The pairing \eqref{eq:pairing} induces an indentification between $\mathcal{H}_{\overline{\Lambda}}^{-k}$ and the dual of $\mathcal{H}_\Lambda^k$. Moreover, if $u$ or $v$ belongs to $E_{R}(X)$ then \eqref{eq:pairing} is just the natural (sesquilinear) pairing between elements of $E_R(X)$ and $E_R'(X)$.
\end{proposition}

\begin{proof}
Assume that $u$ is in $E_{R}(X)$ and that $v \in \mathcal{H}_{\overline{\Lambda}}^{-k}$. Since $T$ is an isometry on its image, we know that
\begin{equation}\label{eq:self_adjoint}
\int_X u \bar{v} \mathrm{d}x = \int_{T^* X} T u \overline{T v} \mathrm{d}\alpha.
\end{equation}
Notice that the function $\alpha \mapsto Tu(\alpha) \overline{Tv (\bar{\alpha})}$ is holomorphic on $(T^* X)_r$. Moreover, from \cite[Lemmas 2.4 and 2.5, Corollary 2.2]{BJ20}, we see that, provided $\tau$ is small enough, there is $r > 0$ such that $Tu(\alpha) \overline{Tv (\bar{\alpha})}$ decays exponentially fast in $(T^* X)_{r}$. This allows us to shift contour in \eqref{eq:self_adjoint} to find that $\int_X u \bar{v} \mathrm{d}x$ coincides with \eqref{eq:pairing}, provided $\tau$ is small enough. By symmetry, we have the same equality when $v$ is assumed to belong to $E_{R}(X)$.

Consequently, the (antilinear) map from $\mathcal{H}_{\overline{\Lambda}}^{-k}$ to the dual of $\mathcal{H}_\Lambda^k$ induced by the pairing \eqref{eq:pairing} is injective. Let us prove that it is surjective. Let $l$ be a continuous linear form on $\mathcal{H}_\Lambda^k$. It follows from \cite[Proposition 2.4]{BJ20} that $S_\Lambda$ is bounded from $L_k^2(\Lambda)$ to $\mathcal{H}_\Lambda^k$, and we can thus define a linear form $\tilde{l}$ on $L_k^2(\Lambda)$ by the formula $\tilde{l}(w) = l(S_\Lambda w)$. Notice that if $u \in \mathcal{H}_\Lambda^k$ then $l(u) = \tilde{l}(T_\Lambda u)$. Let then $h_1$ be the element of $L_k^2(\Lambda)$ such that
\begin{equation*}
\tilde{l}(w) = \int_\Lambda w(\alpha) \overline{h_1(\alpha)} \brac{\va{\alpha}}^{2k} e^{- \frac{2 H(\alpha)}{h}} \mathrm{d}\alpha
\end{equation*}
for every $w \in L_k^2(\Lambda)$. Let us define the function $h_2$ on $\overline{\Lambda}$ by
\begin{equation*}
h_2(\alpha) = h_1(\bar{\alpha}) \brac{\va{\bar{\alpha}}}^{2k} e^{- \frac{2 H(\bar{\alpha})}{h}}, 
\end{equation*}
and notice that $h_2$ belongs to $L_{-k}^2 (\overline{\Lambda})$, so that $v \coloneqq S_{\overline{\Lambda}} h_2$ belongs to $\mathcal{H}_{\overline{\Lambda}}^{-k}$. Let $u \in E_{R}(X)$, then with the pairing above, we have
\begin{equation*}
\langle u,v \rangle = \int_\Lambda T_\Lambda u (\alpha) \overline{\Pi_{\overline{\Lambda}} h_2 (\overline{\alpha})} \mathrm{d}\alpha.
\end{equation*}
Notice that the kernel of the operators $\Pi_\Lambda$ and $\Pi_{\overline{\Lambda}}$ are obtained by restricting respectively to $\Lambda \times \Lambda$ and $\overline{\Lambda} \times \overline{\Lambda}$ the holomorphic kernel of the operator $\Pi = TS$. We write $\Pi(\alpha,\beta)$ for this kernel. Since $S$ is the adjoint of $T$, we find by analytic continuation that $\overline{\Pi(\bar{\alpha},\bar{\beta})} = \Pi(\beta,\alpha)$. It follows then from Fubini's theorem that
\begin{equation*}
\begin{split}
\langle u,v \rangle & = \int_{\Lambda} \Pi_\Lambda T_\Lambda u (\alpha) \overline{h_2(\bar{\alpha})} \mathrm{d}\alpha \\
    & = \int_\Lambda T_\Lambda u (\alpha) \overline{h_1(\alpha)} \brac{\va{\alpha}}^{2k} e^{- 2 H(\alpha)} \mathrm{d}\alpha = l(u).
\end{split}
\end{equation*}
The equality on the first line can be proved first by replacing $h_2$ by a rapidly decaying function and then using an approximation argument. It follows from \cite[Corollary 2.3]{BJ20} and the Oka-Weil theorem that $E_R(X)$ is dense in $\mathcal{H}_\Lambda^k$ and the result follows.
\end{proof}

\subsection{Particularity of logarithmic weights}\label{subsec:logarithmic}

When applying the FBI transform techniques that we describe here in \S \ref{section:general_construction}, the weight $G_0$ will be of logarithmic order. This is a strategy that we already applied in \cite{BGJ}. It amounts to doing $C^\infty$ microlocal analysis with respect to the large parameter $\brac{\va{\alpha}}$ but real-analytic microlocal analysis with respect to the small parameter $h$. 

Using a logarithmic weight allows us to construct spaces that are intermediate between $C^\infty(X)$ and $\mathcal{D}'(X)$.

\begin{proposition}\label{proposition:inclusions_naturelles}
Assume that $G_0$ has logarithmic order. Assume that $\tau$ and $h$ are small enough. Then, for every $k \in \mathbb{R}$, there are continuous injections $C^\infty(X) \subseteq \mathcal{H}_\Lambda^k \subseteq \mathcal{D}'(X)$. Moreover, these injections are natural in the following sense: the diagram
\begin{equation*}
\begin{tikzcd}
C^\infty(X) \arrow[r] \arrow[rd] & \mathcal{H}_\Lambda^k(X) \arrow[r] \arrow[d] & E_R'(X) \\
                       & \mathcal{D}'(X) \arrow[ru] &
\end{tikzcd}
\end{equation*}
is commutative, with $R$ as in the definition of $\mathcal{H}_\Lambda^k$. The arrows that are not given by the proposition are the standard injections.
\end{proposition}

\begin{proof}
It follows from Lemma \ref{lemma:non_stationary_C_infty}, using for instance \cite[(2.9)]{BJ20} to bound $H$, that $C^\infty(X)$ is contained in $\mathcal{H}_\Lambda^k$, where we identify an element of $C^\infty(X)$ with an element of $E_R'(X)$ using the $L^2$ pairing (see also \cite[Lemma 4.10]{BGJ}). The proof of this result actually proves that the injection is continuous (even if the estimates are not uniform in $h$). Notice that $C^\infty(X)$ is dense in $\mathcal{H}_\Lambda^k$ as a consequence of \cite[Corollary 2.3]{BJ20}. Replacing $G$ by $G^*$, we find that $C^\infty(X)$ is also a dense subset of $\mathcal{H}_{\overline{\Lambda}}^{-k}$, with continuous injection. Consequently, the pairing \eqref{eq:pairing} induces a continuous injection of $\mathcal{H}_\Lambda^k$ into $\mathcal{D}'(X)$ according to Proposition \ref{proposition:duality}. Since the pairing \eqref{eq:pairing} coincides with the $L^2$ pairing when $u$ or $v$ is in $E_R(X)$, we see that the diagram above is indeed commutative.
\end{proof}

\begin{remark}\label{remark:pairing_natural}
It follows from Propositions \ref{proposition:duality} and \ref{proposition:inclusions_naturelles} that if $u \in \mathcal{H}_\Lambda^k$ and $v \in \mathcal{H}_{\overline{\Lambda}}^{-k}$ are such that $u$ or $v$ is in $C^\infty(X)$, then the pairing \eqref{eq:pairing} coincides with the natural pairing between a smooth function and a distribution.
\end{remark}

When $G_0$ is of logarithmic order, we may identify the $\mathcal{H}_\Lambda^k$'s with spaces of distributions, and consequently it makes sense to let a differential operator $P$ with $C^\infty$ coefficients act on the elements of the $\mathcal{H}_\Lambda^k$'s. In the following lemma, we see that under the assumption \eqref{eq:assumption} we can relate the action of $P$ on these spaces with the action of the operator $T_\Lambda P S_\Lambda$ that we studied in Proposition \ref{proposition:toeplitz}.

\begin{lemma}\label{lemma:seems_simple}
Assume that $G_0$ has logarithmic order. Let $P$ be a semiclassical operator of order $m \in \mathbb{N}$ that satisfy \eqref{eq:assumption}. Assume that $\tau$ is small enough. Then, for every $k \in \mathbb{R}$, the operator $P$ is bounded from $\mathcal{H}_\Lambda^k$ to $\mathcal{H}_\Lambda^{k-m}$ and for every $u \in \mathcal{H}_\Lambda^k$ we have
\begin{equation*}
T_\Lambda P u = (T_\Lambda P S_\Lambda)T_\Lambda u,
\end{equation*}
where we recall that $T_\Lambda P S_\Lambda$ is the operator with kernel \eqref{eq:def_kernel_TPS}.
\end{lemma}

\begin{proof}
For $\alpha \in \Lambda$, we have by definition
\begin{equation}\label{eq:def_TPu}
T_\Lambda P u(\alpha) = \int_M P u(y) K_T(\alpha,y) \mathrm{d}y = \int_M u(y) {}^t P_y(K_T(\alpha,y)) \mathrm{d}y,
\end{equation}
where ${}^t P$ denotes the adjoint of $P$ for the bilinear (rather than sesquilinear) $L^2$ pairing on $M$. Notice that for $\alpha \in \Lambda$, the function $h_\alpha : y \mapsto {}^t P_y(K_T(\alpha,y))$ is $C^\infty$. Consequently, one may use the $C^\infty$ non-stationary phase method, Lemma \ref{lemma:non_stationary_C_infty}, to find that ${}^t S h_\alpha(\beta)$ decays faster than the inverse of any polynomial when $\beta$ becomes large while its imaginary part remains bounded (from the Kohn--Nirenberg point of view) by $L \log \brac{\va{\beta}} / \brac{\va{\beta}}$ (for any large constant $L$). Notice however that this estimate is not uniform in $h$ (we apply Lemma \ref{lemma:non_stationary_C_infty} with $h$ fixed and $\brac{\va{\beta}}$, rather than $\brac{\va{\beta}}/h$, as a large parameter). Consequently, we can shift contour in the integral equality ${}^t T {}^t S h_\alpha = {}^t (ST) h_\alpha = h_\alpha$ to find 
\begin{equation*}
\begin{split}
h_{\alpha}(x) & = \int_\Lambda K_T(\beta,x) \p{\int_M K_S(y,\beta) h_\alpha(y) \mathrm{d}y} \mathrm{d}\beta \\ & = \int_\Lambda K_T(\beta,x) T_\Lambda P S_\Lambda(\alpha,\beta) \mathrm{d}\beta.
\end{split}
\end{equation*}
Using the fast decay of ${}^t S h_\alpha$, we see that this integral actually converges in $C^\infty(X)$, and plugging this equality in \eqref{eq:def_TPu}, we get $T_\Lambda P u = (T_\Lambda P S_\Lambda) T_\Lambda u$. It follows then from Proposition \ref{proposition:toeplitz} that $T_\Lambda P u \in L_{k-m}^2(\Lambda)$, that is $P u \in \mathcal{H}_\Lambda^{k-m}(\Lambda)$.
\end{proof}

The following result will be used in the demonstration of Proposition \ref{proposition:general_statement} to prove that the elements of the spaces $\mathcal{H}_1$ and $\mathcal{H}_2$ are bounded near $\partial Y$.

\begin{proposition}\label{proposition:continuity}
Let $K$ be a compact subset of $X$. Assume that $G_0$ has logarithmic order and that there is $C > 0$ such that if $\alpha \in T^*_{K} X$ is large enough then
\begin{equation*}
G_0(\alpha) \leq - C^{-1} \log \brac{\va{\alpha}}
\end{equation*}
Assume that $\tau$ is small enough. Then, for every $k \in \mathbb{R}$, if $h$ is small enough then the elements of $\mathcal{H}_\Lambda^k$ are continuous on a neighbourhood of $K$.
\end{proposition}

\begin{proof}
Let $N > n$. It follows from \cite[Lemmas 4.2 and 4.9]{BGJ} that, for $h$ small enough, there is a neighbourhood $U$ of $K$ such that if $v$ is in $H^{-N}(X)$ and supported in $U$ then $v$ belongs to $\mathcal{H}_{\overline{\Lambda}}^{-k}$ and its norm in this space is less than $C \n{v}_{H^{-N}}$, where the constant $C$ may depend on $h$ but not on $v$.

Let $u \in \mathcal{H}_\Lambda^k$. If $\chi$ is a $C^\infty$ function supported in the intersection of $U$ with a coordinates patch, then we see that in this coordinates the Fourier transform of $\chi u$ decays faster than $\brac{\xi}^{-N}$. Indeed, the $H^{-N}$ norm of the functions given in coordinates by $\chi(x) e^{i x \xi}$ decays like $\brac{\xi}^{-N}$ when $\xi$ tends to $+ \infty$. Thus, the same is true for the norm of these functions in $\mathcal{H}_{\overline{\Lambda}}^{-k}$. It follows then from Remark \ref{remark:pairing_natural} that $\widehat{\chi u}(\xi)$, which is the $L^2$ pairing of $u$ with one of these functions, decays like $\brac{\xi}^{-N}$ when $\xi$ tends to $+ \infty$. Consequently, the distribution $\chi u$ is a continuous function, and the result follows by a partition of unity argument.
\end{proof}

\section{General construction (proof of Proposition \ref{proposition:general_statement})}\label{section:general_construction}

The aim of this section is to prove Proposition \ref{proposition:general_statement}. We will use the notation that we introduced in \S \ref{subsection:general_assumption}.

In \S \ref{subsec:choice_parameters}, we fix the value of certain parameters that play an important role in the proof of Proposition \ref{proposition:general_statement} and define the modification $P_h(\omega)$ of $\mathcal{P}_h(\omega)$. In \S \ref{subsec:spaces}, we define the spaces that will be $\mathcal{H}_1$ and $\mathcal{H}_2$ in Proposition \ref{proposition:general_statement}, and explain how the action of $P_h(\omega)$ on these spaces is related to the values of a certain symbol (Proposition \ref{proposition:multiplication_formula}). In \S \ref{subsec:ellipticity_estimates}, we prove ellipticity estimates on this symbol (Lemmas \ref{lemma:high_frequencies} and \ref{lemma:everywhere}). In \S \ref{subsection:invertibility_and_Fredholm}, we use these estimates to study the functional analytic properties of $P_h(\omega)$ acting on the spaces defined in \S \ref{subsec:ellipticity_estimates}: we prove that $P_h(\omega)$ is Fredholm by proving that it is invertible after pertubration by a compact operator (Lemmas \ref{lemma:estimee_fredholm} and \ref{lemma:estimee_fredholm_adjoint} and Proposition \ref{proposition:meromorphic_inverse}), and that $P_h(i\nu)$ is invertible for some $\nu > 0$ (Lemmas \ref{lemma:estimee_invertibility} and \ref{lemma:estimee_invertibility_adjoint} and Proposition \ref{proposition:meromorphic_inverse}). In \S \ref{subsec:counting_resonances}, we prove the crucial point \ref{item:counting_resonances} from Proposition \ref{proposition:general_statement} (from which our upper bounds on resonances, Theorems \ref{theorem:main} and \ref{theorem:schwarzschild}, follow). This is done by evaluating the trace class norm of the compact perturbation that we use to make $P_h(\omega)$ invertible (Lemmas \ref{lemma:bound_extended_resonances} and \ref{lemma:trace_class}). Finally, in \S \ref{subsec:summary}, we put all these information together in order to get a full proof of Proposition \ref{proposition:general_statement}.

Notice that, in most of this section, we are not working directly with the operator $P_h(\omega)$, but rather with an operator $\widetilde{P}_h(\omega)$, defined in \S \ref{subsec:choice_parameters}, which is conjugated to $P_h(\omega)$, but simpler to apprehend.

\subsection{Choice of parameters and modification of the operator}\label{subsec:choice_parameters}

We use the notation from \S \ref{subsection:general_assumption}. Up to making $\epsilon$ smaller, we may assume that $w'(x_1) > \epsilon$ for every $x_1 \in ]-\epsilon,\epsilon[$ and $p_0(x) < - \epsilon$ for every $x \in U \cup Y$ (this second point is a consequence of assumption \ref{item:ellipticite_bas}). We may also assume that $x_1$ extends to a smooth function on the whole $X$ (analytic on $U$) such that $U = \set{ - \epsilon < x_1 < \epsilon}$, $Y = \set{x_1 > 0}$ and $X \setminus \overline{Y} = \set{x_1 < 0}$.

Let us introduce on $T^* U \simeq T^* (]-\epsilon,\epsilon[)_{(x_1,\xi_1)} \times T^* \partial Y_{(x',\xi')}$ the symbol of logarithmic order
\begin{equation*}
G_1(x_1,x',\xi_1,\xi') = \log\p{2 + \xi_1^2 + \va{\xi'}^2}
\end{equation*}
and denote by $H_{G_1}$ the Hamiltonian flow of $G_1$ for the canonical symplectic form on $T^* U$. Here, the quantity $\va{\xi'}^2$ is computed using any smooth Riemannian metric on $\partial Y$, e.g. the restriction of $g_X$. Let us compute $H_{G_1} p_2$ where we recall that $p_2$ is the principal symbol of the order $2$ differential operator $P_2$ from \eqref{eq:pinceau}. Using local coordinates on $\partial Y$, we find that
\begin{equation*}
\begin{split}
& H_{G_1} p_2(x_1,x',\xi_1,\xi') = \frac{2 \xi_1}{2 + \xi_1^2 + \va{\xi'}^2} w'(x_1) \xi_1^2 \\ & \qquad \qquad + \frac{2 \xi_1}{2 + \xi_1^2 + \va{\xi'}^2} \frac{\partial q_1}{\partial x_1}(x_1,x',\xi') + \frac{\nabla_{\xi'}(\va{\xi'}^2)}{2 + \xi_1^2 + \va{\xi'}^2} \cdot \nabla_{x'}q_1(x_1,x',\xi') \\ & \qquad \qquad - \frac{\nabla_{x'}(\va{\xi'}^2)}{2 + \xi_1^2 + \va{\xi'}^2} \cdot \nabla_{\xi'}q_1(x_1,x',\xi').
\end{split}
\end{equation*}
Since $w'(x_1) > \epsilon$, the term on the first line is elliptic of order $1$ whenever $\xi_1$ is larger than a fixed proportion of $\va{\xi'}$. Moreover, this term has the same sign as $\xi_1$. The terms on the second and third line are also of order $1$, and they can be made arbitrarily small by assuming that $\xi_1$ is much larger than $\xi'$. Hence, there is some small $\epsilon_1 \in ]0,\epsilon[$ such that if $(x_1,x',\xi_1,\xi') \in T^* U$ and $\va{\xi_1} \geq \epsilon_1^{-1}(1+ \va{\xi'})$ we have
\begin{equation}\label{eq:escape}
\frac{H_{G_1} p_2(x_1,x',\xi_1,\xi')}{\xi_1} \geq C^{-1},
\end{equation}
for some constant $C > 0$.

Let then $C_0$ be a bound for the derivative of $w$ on $]-\epsilon,\epsilon[$. We choose $\epsilon_0 \in ]0,\epsilon[$ so small that if $(x_1,x',\xi_1,\xi') \in T^* U$ and $\va{\xi_1} \leq 2 \epsilon_1^{-1}(1+ \va{\xi'})$ we have, with $\xi = (\xi_1,\xi')$,
\begin{equation}\label{eq:epsilon0_c1}
-C_0 \epsilon_0 \xi_1 ^2 + \va{\xi}^2 + 1 \geq C^{-1}\p{1 + \va{\xi}^2}
\end{equation}
and
\begin{equation}\label{eq:epsilon0_c2}
-C_0 \epsilon_0 \xi_1^2 + q_1(x_1,x',\xi') +1 \geq C^{-1}\p{1 + \va{\xi}^2}.
\end{equation}
Here, $\va{\xi}^2$ is defined using the metric $g_X$ on $X$, and we used the ellipticity condtion on $q_1$ (assumption \ref{item:elliptique_bord}).

Let $\chi : \mathbb{R} \to [0,1]$ be a smooth function such that $\chi(t) = 0$ for $t \leq -\epsilon_0$ and $\chi(t) = 1$ for $t \geq - 5 \epsilon_0/6$. Let $\psi$ be a real-analytic function from $\mathbb{R}$ to $\mathbb{R}$ such that $t \psi'(t) \leq 0$ for $t \in \mathbb{R}$ and $t \psi'(t) < 0$ for $t \neq 0$. One can take for instance $\psi(t) = - t^2/2$.

Let $Q$ be a semiclassical differential operator of order $2$ with principal symbol $q$. We assume that $Q$ has the following properties:
\begin{itemize}
\item the coefficients of $Q$ are supported in $\set{x_1 < - \epsilon_0/2}$;
\item the principal symbol of $Q$ is
\begin{equation*}
q(x,\xi) = \chi_1(x_1)(1+ \va{\xi}^2),
\end{equation*}
where $\chi_1 : \mathbb{R} \to [0,1]$ is a smooth function supported in $] - \infty, - \epsilon_0/2]$ and that takes value $1$ on $]- \infty, - 2 \epsilon_0/3]$. For instance, one can take $Q = \chi_1(I - h^2 \Delta_{g_X})$.
\end{itemize}

The modification $P_h(\omega)$ of the operator $\mathcal{P}_h(\omega)$ for which Proposition \ref{proposition:general_statement} will be established is
\begin{equation}\label{eq:modification_operators}
P_h(\omega) = \chi(x_1) \mathcal{P}_h(\omega) + e^{- \frac{\psi(x_1)}{h}} Q e^{\frac{\psi(x_1)}{h}},
\end{equation}
but we will rather study the conjugated operator
\begin{equation*}
\widetilde{P}_h(\omega) = e^{\frac{\psi(x_1)}{h}}\chi(x_1) \mathcal{P}_h(\omega)e^{- \frac{\psi(x_1)}{h}} + Q.
\end{equation*}

For $(x,\xi) = (x_1,x',\xi_1,\xi') \in T^* U \simeq T^*(]-\epsilon,\epsilon[ \times \partial Y)$, the principal symbol of $\tilde{p}(x,\xi;\omega)$ of $\widetilde{P}_h(\omega)$ is given by
\begin{equation}\label{eq:ptilde}
\begin{split}
\tilde{p}(x,\xi;\omega)& = \chi(x_1)\Big(w(x_1) \xi_1^2 + q_1(x_1,x',\xi') + 2 i w(x_1)\psi'(x_1) \xi_1 + \omega p_1(x_1,\xi_1) \\ & \qquad \qquad \qquad \qquad  + i \omega p_1(x_1,\psi'(x_1)) + \omega^2 p_0(x) - \psi'(x_1)^2 w(x_1)\Big) \\ & \qquad \qquad \qquad + \chi_1(x_1) (1+ \va{\xi}^2).
\end{split}
\end{equation}

Finally, let $\phi$ be a $C^\infty$ function from $\mathbb{R}$ to $[0,1]$, supported in $]- \epsilon_0/3,\epsilon_0[$, such that $\phi(t) = 1$ for $t \in [- \epsilon_0/6,2\epsilon_0/3]$ and $t \phi'(t) \leq 0$ for every $t \in \mathbb{R}$. 

Our choices of parameters are summed up in Figure \ref{figure:choice_parameters}, where the black line represent the $x_1$-axis. The colored zones in this drawing correspond to places where we will use different mechanisms to prove the Fredholm property for $\widetilde{P}_h(\omega)$. In the purple zone (which is compactly contained in $Y$), we will use the ellipticity of $\widetilde{P}_h(\omega)$, which follows from our assumption \ref{item:ellipticite_interieur} in \S \ref{subsection:general_assumption}. In the green zone (which is away from $\overline{Y}$), the operator $\widetilde{P}_h(\omega)$ is also elliptic, but this is just because $Q$ is. Finally, the most interesting part is the blue zone, where two phenomena occur: in some places $\widetilde{P}_h(\omega)$ is elliptic and in other places we need to use propagation and radial estimates to get the Fredholm property. See \S \ref{subsec:ellipticity_estimates} for the details on how Figure \ref{figure:choice_parameters} can be turned into actual estimates.

\begin{figure}[h]
\centering
\includegraphics[scale=0.55]{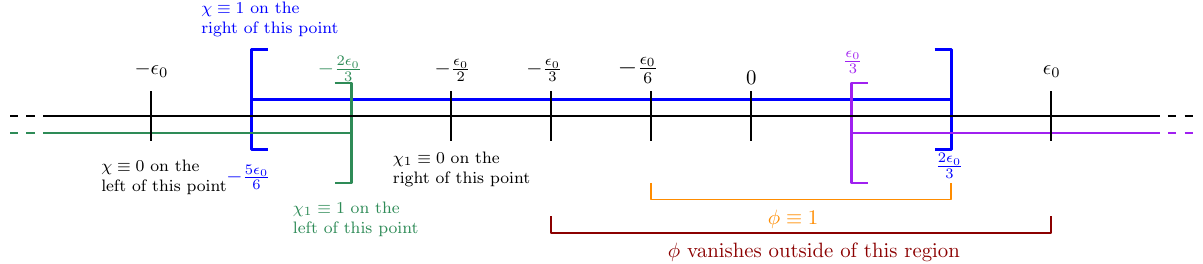}
\caption{Some relevant places near $\partial Y$}
\label{figure:choice_parameters}
\end{figure}

\subsection{Definition of the spaces}\label{subsec:spaces}

We define the symbol $G_0$ on $T^* X$ by
\begin{equation*}
G_0(x,\xi) = - \phi(x_1) G_1(x,\xi) \textup{ for } (x,\xi) \in T^* X.
\end{equation*}
Then, for some small $r > 0$, we extend $G_0$ to $(T^* X)_r$ as a symbol of logarithmic order. The particular features of the extension are irrelevant as soon as we have symbolic estimates, and that $G_0$ is identically equal to $0$ away from a small neighbourhood of the support of $\phi$, so that all derivatives of $G_0$ vanish at any point of $T^* X$ such that $x_1 \leq - \epsilon_0/3$ or $x_1 \geq \epsilon_0$ (even derivatives in directions that are not tangent to $T^* X$). As above we define the escape function $G = \tau G_0$ for some small $\tau > 0$. We let $\Lambda = e^{H_G^{\omega_I}} T^* M$ be the associated Lagrangian deformation and $(\mathcal{H}_\Lambda^k)_{k \in \mathbb{R}}$ the associated family of Hilbert spaces (see \S \ref{section:FBI_transform}). Notice that these are spaces of distributions according to Proposition \ref{proposition:inclusions_naturelles}. For $k \in \mathbb{R}$, define the Hilbert space
\begin{equation}\label{eq:deffk}
\mathcal{F}_k = \set{ u \in \mathcal{D}'(M) : e^{\frac{\psi}{h}} u \in \mathcal{H}_\Lambda^k},
\end{equation}
where we recall that $\psi$ is defined in \S \ref{subsec:choice_parameters}. The spaces $\mathcal{H}_1$ and $\mathcal{H}_2$ in Proposition \ref{proposition:general_statement} will be respectively $\set{u \in \mathcal{F}_k : P_h(0) u \in \mathcal{F}_{k-1}}$ and $\mathcal{F}_{k-1}$.

Notice that it is equivalent to study $P_h(\omega)$ acting on the $\mathcal{F}_k$'s or $\widetilde{P}_h(\omega)$ acting on the $\mathcal{H}_\Lambda^k$'s. Also, we can write $\widetilde{P}_h(\omega) = \widetilde{P}_2 + \omega \widetilde{P}_1 + \omega^2 \widetilde{P}_0$ where the $\widetilde{P}_j$'s are semiclassical differential operator with analytic coefficients near the support of $\phi$. Consequently, these operators satisfy the assumption \eqref{eq:assumption} and it makes sense to restrict their principal symbols (and thus the principal symbol $\tilde{p}(\cdot;\omega)$ of $\widetilde{P}_h(\omega)$) to $\Lambda$, see the remark below \eqref{eq:assumption}. Applying Proposition \ref{proposition:toeplitz} and Lemma \ref{lemma:seems_simple} to the operators $\widetilde{P}_0, \widetilde{P}_1$ and $\widetilde{P}_2$, we find:

\begin{proposition}\label{proposition:multiplication_formula}
Assume that $\tau$ and $h$ are small enough. Let $m \in \mathbb{R}$ and $f$ be a symbol of order $m$ on $\Lambda$. Let $\omega \in \mathbb{C}$. Let $k_1$ and $k_2$ be such that $k_1 + k_2 = m +1 $. Then, there is a constant $C$ such that for every $u,v \in \mathcal{H}_\Lambda^{\infty}$, we have
\begin{equation*}
\va{\int_\Lambda f T_\Lambda \widetilde{P}_h(\omega) u \overline{T_\Lambda v} \mathrm{d}\alpha - \int_\Lambda f (\alpha) \tilde{p}(\alpha;\omega) T_\Lambda u \overline{T_\Lambda v} \mathrm{d}\alpha} \leq C h \n{u}_{\mathcal{H}_\Lambda^{k_1}} \n{v}_{\mathcal{H}_\Lambda^{k_2}}.
\end{equation*}
Here, the constant $C$ depends continuously on $\omega$ and $\tilde{p}(\cdot;\omega)$ denotes the principal symbol of $\widetilde{P}_h(\omega)$. We also wrote $\mathcal{H}_\Lambda^\infty$ for $\bigcap_{k \in \mathbb{R}} \mathcal{H}_\Lambda^k$.
\end{proposition}

Another consequence of Proposition \ref{proposition:toeplitz} and Lemma \ref{lemma:seems_simple} that it will be useful to remember is that, under the assumptions of Proposition \ref{proposition:multiplication_formula}, for every $k \in \mathbb{R}$ the family $\omega \mapsto \widetilde{P}_h(\omega) - \widetilde{P}_h(0)$ is a holomorphic family of bounded operators from $\mathcal{H}_\Lambda^{k}$ to $\mathcal{H}_\Lambda^{k-1}$.

\subsection{Ellipticity estimates}\label{subsec:ellipticity_estimates}

In order to use Proposition \ref{proposition:multiplication_formula}, let us introduce the following subsets of $T^* X$:
\begin{equation*}
\begin{split}
& V_R = \set{x_1 \leq - 2 \epsilon_0 /3} \cup \set{x_1 \geq \epsilon_0/3} \\ & \quad \qquad \qquad \qquad \qquad \bigcup \p{\set{- 5\epsilon_0/6 \leq x_1 \leq 2 \epsilon_0/3} \cap \set{\va{\xi_1} \leq 2\epsilon_1^{-1}(1+ \va{\xi'})}},
\end{split}
\end{equation*}
\begin{equation*}
V_+ = \set{- 5\epsilon_0/6 \leq x_1 \leq 2\epsilon_0/3} \cap \set{\xi_1 \geq \epsilon_1^{-1}(1 + \va{\xi'})}
\end{equation*}
and
\begin{equation*}
V_- = \set{ - 5\epsilon_0/6 \leq x_1 \leq 2\epsilon_0/3} \cap \set{\xi_1 \leq -\epsilon_1^{-1} (1 + \va{\xi'})}.
\end{equation*}

Here, we see the function $x_1 : X \to \mathbb{R}$ as a function on $T^* X$ by composition by the canonical projection $T^* X \to X$, and the constant $\epsilon_1$ has been defined in \S \ref{subsec:choice_parameters}. Notice that a point of $T^* X$ for which $x_1$ is between $ - \epsilon$ and $\epsilon$ is in $T^* U \simeq T^* (]- \epsilon,\epsilon[) \times T^* (\partial Y)$, and may consequently be written as $(x_1,x',\xi_1,\xi')$ with $(x_1,\xi_1) \in T^* (]- \epsilon,\epsilon[)$ and $(x',\xi') \in T^* \partial Y$. This is how we make sense of $\xi'$ in the equation above. We use the same metric on $T^* \partial Y$ as in \S \ref{subsec:choice_parameters}. 

We let also $W_R,W_+$ and $W_-$ denote the images respectively of $V_R,V_+$ and $V_-$ by $e^{H_G^{\omega_I}}$. Notice that $T^* X = V_R \cup V_+ \cup V_-$ so that $\Lambda = W_R \cup W_+ \cup W_-$. We are going to prove two ellipticity estimates, Lemmas \ref{lemma:high_frequencies} and \ref{lemma:everywhere}, that will be used in \S \ref{subsection:invertibility_and_Fredholm} below to prove that $\widetilde{P}_h(\omega)$ is Fredholm for a certain range of $\omega$'s and invertible for at least one of these $\omega$'s.

\begin{lemma}\label{lemma:high_frequencies}
Let $\tau > 0$ be small and fixed. There is $\kappa > 0$ (depending on $\tau$) such that the following holds. For every compact subset $K$ of $\set{z \in \mathbb{C} : \im z \geq -\kappa}$, there is a constant $C_K$ such that for $\omega \in K$ and $\alpha \in \Lambda$ such that $\brac{\va{\alpha}} \geq C_K$ we have
\begin{itemize}
\item if $\alpha \in W_R$ then
\begin{equation*}
\re \tilde{p}(\alpha;\omega) \geq C_K^{-1} \brac{\va{\alpha}};
\end{equation*}
\item if $\alpha \in W_+$ then
\begin{equation*}
\im \tilde{p}(\alpha;\omega) \leq - C_K^{-1} \brac{\va{\alpha}};
\end{equation*}
\item if $\alpha \in W_-$ then
\begin{equation*}
\im \tilde{p}(\alpha;\omega) \geq  C_K^{-1} \brac{\va{\alpha}}.
\end{equation*}
\end{itemize}
\end{lemma}

\begin{proof}
Let us write $\alpha  = e^{H_G^{\omega_I}} (x,\xi)$ for $(x,\xi) \in T^* X$. We will distinguish the different cases that appear in the definitions of $V_R,V_+$ and $V_-$.

\noindent \textbf{First case:} $(x,\xi) \in \set{x_1 \leq - 2 \epsilon_0 /3}$.

In that case, we see that $G_0$ is null on a neighbourhood of $(x,\xi)$ so that $\alpha = (x,\xi)$. Moreover $\chi_1(x_1) = 1$, so that $q(x,\xi) = 1 + \va{\xi}^2$. Using that $\chi(x_1) = 0$ for $x_1 \leq -\epsilon_0$, we find from \eqref{eq:ptilde} that the real part of $\tilde{p}(x,\xi;\omega)$ is greater than
\begin{equation}\label{eq:real_part_1}
- C_0 \epsilon_0 \xi_1^2 - C(1 + \va{\xi_1}) + \va{\xi}^2.
\end{equation}
Here, $C > 0$ is some constant that depends continuously on $\omega$ (and does not depend on $\alpha$). Thanks to our assumption \eqref{eq:epsilon0_c1} on $\epsilon_0$, we see that \eqref{eq:real_part_1} is larger than $C^{-1} \brac{\va{\alpha}}^2$ and hence that $C^{-1} \brac{\va{\alpha}}$ when $\alpha$ is large enough.

\noindent \textbf{Second case:} $(x,\xi) \in \set{x_1 \geq \epsilon_0/3}$.

Notice that $H_G^{\omega_I} \tilde{p}(\alpha;h)$ is $\mathcal{O}\p{\brac{\va{\alpha}} \log \brac{\va{\alpha}}}$, with symbolic estimates (it follows from the fact that $G$ has logarithmic order). Consequently, we have
\begin{equation*}
\begin{split}
\tilde{p}\p{\alpha;h} & = \tilde{p}\p{x,\xi;\omega} + \mathcal{O}\p{\brac{\va{\alpha}} \log \brac{\va{\alpha}}} \\
& = p\p{x,\xi + i \psi'(x_1) \mathrm{d}x_1;\omega} + \mathcal{O}\p{\brac{\va{\alpha}} \log \brac{\va{\alpha}}} \\ & = p_2(x,\xi) + \mathcal{O}\p{\brac{\va{\alpha}} \log \brac{\va{\alpha}}}.
\end{split}
\end{equation*}
Thanks to our assumption \ref{item:ellipticite_interieur} of ellipticity on $p_2$ in \S \ref{subsection:general_assumption}, we see that this quantity is larger than $C^{-1} \brac{\va{\alpha}}^2$ and hence that $C^{-1} \brac{\va{\alpha}}$ when $\alpha$ is large enough.

\noindent \textbf{Third case:} $(x,\xi) \in \set{- 5\epsilon_0/6 \leq x_1 \leq 2\epsilon_0/3} \cap \set{\va{\xi_1} \leq 2\epsilon_1^{-1} (1 + \va{\xi'})}$.

In that case, we notice that $\chi(x_1) = 1$ and that $\tilde{p}\p{\alpha;\omega} = \tilde{p}\p{x,\xi;\omega} + \mathcal{O}\p{\brac{\va{\alpha}} \log \brac{\va{\alpha}}}$ as above. Using the formula \eqref{eq:ptilde}, we find that
\begin{equation*}
\re \tilde{p}\p{\alpha;\omega} \geq q_1(x_1,x',\xi') - C_0 \epsilon_0 \xi_1^2 - C \brac{\va{\alpha}} \log \brac{\va{\alpha}}.
\end{equation*}
Using that $\va{\xi_1} \leq 2\epsilon_1^{-1} (1 + \va{\xi'})$ and our assumption \eqref{eq:epsilon0_c2} on $\epsilon_0$, we see that for $\alpha$ large enough, this real part is larger than $C^{-1}\brac{\va{\alpha}}^2$ and hence that $C^{-1}\brac{\va{\alpha}}$.

\noindent \textbf{Fourth case:} $(x,\xi) \in V_+$.

Notice that we have
\begin{equation}\label{eq:taylor_symbol}
\im \tilde{p}(\alpha;\omega) = \im \tilde{p}(x,\xi;\omega) + \tau H_{G_0}^{\omega_I} \im \tilde{p}(x,\xi;\omega) + \mathcal{O}\p{(\log \brac{\va{\alpha}})^2}.
\end{equation}
We want to estimate $H_{G_0}^{\omega_I} \im \tilde{p}(x,\xi;\omega)$. First, notice that $\tilde{p}(\cdot;\omega) - p_2$ is a symbol of order $1$ on a neighbourhood of the support of $G_0$, so that
\begin{equation*}
\begin{split}
H_{G_0}^{\omega_I} \im \tilde{p}(x,\xi;\omega)& = H_{G_0}^{\omega_I} \im p_2(x,\xi) + \mathcal{O}\p{\log\brac{\va{\alpha}}} \\
    & = -H_{\im p_2}^{\omega_I} G_0(x,\xi) + \mathcal{O}\p{\log\brac{\va{\alpha}}}.
\end{split}
\end{equation*}
Notice that the symbol $\im p_2$ vanishes on the real cotangent bundle $T^* X$, which is a Lagrangian submanifold for the symplectic form $\omega_I$. Consequently, the Hamiltonian vector field $H_{\im p_2}^{\omega_I}$ is tangent to $T^* X$ (this is why we only care about the value of $G_0$ on $T^* X$). Recall that $\omega_R$ denotes the real part of the canonical symplectic form $\omega$ on $(T^* X)_r$. For $u$ tangent to $T^* X$, we have
\begin{equation}\label{eq:symplectic_calcul}
\begin{split}
\omega_R(u, H_{\im p_2}^{\omega_I}) & = \im\p{i \omega(u,H_{\im p_2}^{\omega_I})} = \im\p{\omega(iu,H_{\im p_2}^{\omega_I})} \\
    & = \mathrm{d}(\im p_2) \cdot (iu) = \mathrm{d}(\re p_2) \cdot u \\
    & = \omega_R(u, H_{p_2}),
\end{split}
\end{equation}
where $H_{p_2}$ is the Hamiltonian vector field of $p_2$ for the (real) canonical symplectic form on the real cotangent bundle $T^* X$. We used the Cauchy--Riemann equation on the second line of \eqref{eq:symplectic_calcul}. On the last line, we used the fact that $p_2$ is real-valued on $T^* X$ and that the pullback of $\omega_R$ on $T^* X$ is the canonical symplectic form on $T^* X$. Since $\omega_R$ is symplectic on $T^*X$ and the vector fields $H_{p_2}$ and $H_{\im p_2}^{\omega_I}$ are parallel to $T^* X$, we find that $H_{\im p_2}^{\omega_I}$ coincides with $H_{p_2}$ on $T^* X$. It follows that
\begin{equation*}
\begin{split}
H_{G_0}^{\omega_I} \im \tilde{p}(x,\xi;&\omega) = H_{G_0} p_2 (x,\xi) + \mathcal{O}\p{\log \brac{\va{\alpha}}} \\
    & = - \phi(x_1) H_{G_1} p_2(x,\xi) + 2 w(x_1) \phi'(x_1) G_1(x,\xi) \xi_1 + \mathcal{O}\p{\log \brac{\va{\alpha}}} \\
    & \leq - C^{-1} \phi(x_1) \xi_1 + C_\omega \log \brac{\va{\alpha}},
\end{split}
\end{equation*}
for some constant $C_\omega > 0$ that depends continuously on $\omega$ and some constant $C$ that does not depend on $\omega$. Here, we used \eqref{eq:escape}, which is valid thanks to the assumption $(x,\xi) \in V_+$, and the fact that $w(x_1) \phi'(x_1) \leq 0$. Then, we plug this estimate in \eqref{eq:taylor_symbol} to find that
\begin{equation*}
\begin{split}
\im \tilde{p}(\alpha;\omega) & \leq 2 w(x_1) \psi'(x_1) \xi_1 + \im \omega p_1(x_1,\xi_1) - C^{-1} \tau \phi(x_1) \xi_1 + C_\omega (\log \brac{\va{\alpha}})^2 \\
    & \leq - C^{-1} \p{ -w(x_1) \psi'(x_1) + \tau \phi(x_1) + \im \omega}\xi_1 + C_\omega (\log \brac{\va{\alpha}})^2,
\end{split}
\end{equation*}
where the constants $C$ may change from one line to another but still does not depend on $\omega$. We used here that $p_1(x_1,\xi_1)$ is elliptic of order $1$ with the same sign as $ - \xi_1$, that is our assumption \ref{item:structure_sous_principal} from \S \ref{subsection:general_assumption}. Notice that $w(x_1) \psi'(x_1)$ has the same sign as $x_1 \psi'(x_1)$, and consequently there is a constant $\kappa > 0$ such that $-w(x_1) \psi'(x_1) + \tau \phi(x_1) > \kappa$ for $- \frac{5 \epsilon_0}{6} < x_1 < \frac{2\epsilon_0}{3}$. Hence, if $\im \omega > - \kappa$, we see that $\im \tilde{p}(\alpha;\omega)$ is less than $ - C^{-1} \brac{\va{\alpha}}$ when $\alpha$ is large.

\noindent \textbf{Fifth case:} $(x,\xi) \in V_-$.

This is the same as the fourth case up to a few sign flips.
\end{proof}

\begin{remark}
Let us point out how the five cases in the proof of Lemma \ref{lemma:high_frequencies} correspond to different places in Figure \ref{figure:choice_parameters}. The first and the second cases correspond respectively to the green and the purple zone. The last three cases correspond to the blue zone (to distinguish these cases one need to consider the momentum variable which is not represented on Figure \ref{figure:choice_parameters}).
\end{remark}

\begin{lemma}\label{lemma:everywhere}
Assume that $\nu$ is large enough. Assume that $\tau$ is small enough (depending on $\nu$). Then there is a constant $C > 0$ such that for every $\alpha \in \Lambda$ we have
\begin{itemize}
\item if $\alpha \in W_R$ then
\begin{equation*}
\re \tilde{p}(\alpha;i\nu) \geq C^{-1} \brac{\va{\alpha}};
\end{equation*}
\item if $\alpha \in W_+$ then
\begin{equation*}
\im \tilde{p}(\alpha;i\nu)  \leq - C^{-1} \brac{\va{\alpha}};
\end{equation*}
\item if $\alpha \in W_+$ then
\begin{equation*}
\im \tilde{p}(\alpha;i\nu) \geq C^{-1} \brac{\va{\alpha}}.
\end{equation*}
\end{itemize}
\end{lemma}

\begin{proof}
We write as above $\alpha  = e^{H_G^{\omega_I}} (x,\xi)$ for $(x,\xi) \in T^* X$. We review the same five cases as in the proof of Lemma \ref{lemma:high_frequencies}, with the additional assumption that $\omega = i \nu$ with $\nu > 0 $ large.

\noindent \textbf{First case:} $(x,\xi) \in \set{x_1 \leq - 2 \epsilon_0 /3}$.

The symbol $q(x,\xi)$ is still $1 + \va{\xi}^2$ in that case. Notice that we have here $p_0(x) < - \epsilon$ (see the beginning of \S \ref{subsec:choice_parameters}). Looking at \eqref{eq:ptilde}, we find that, for some constant $C > 0$, the real part of $\tilde{p}(x,\xi;i\nu)$ is larger than 
\begin{equation*}
\chi(x_1)\p{ - C_0 \epsilon_0 \xi_1^2 + \epsilon \nu^2  - C(1+ \nu)} + (1 + \va{\xi}^2) \geq  - C_0 \epsilon_0 \xi_1^2 + (1 + \va{\xi}^2),
\end{equation*}
provided $\nu$ is large enough so that $\epsilon \nu^2 - C(1 + \nu) \geq 0$. Using our assumption \eqref{eq:epsilon0_c1}, we see that the real part of $\tilde{p}(\alpha;i\nu)$ is indeed larger than $C^{-1} \brac{\va{\alpha}}$.

\noindent \textbf{Second case:} $(x,\xi) \in \set{x_1 \geq \epsilon_0/3}$.

Notice that $\tilde{p}(\cdot;i \nu) = \tilde{p}_2 + i \nu\tilde{p}_1 - \nu^2 \tilde{p}_0$, where for $j = 0,1,2$, the principal symbol $\tilde{p}_j$ of $\widetilde{P}_j$ is a symbol of order $j$ that does not depend on $\nu$. It follows that $H_{G}^{\omega_I} \tilde{p}(\cdot;i \nu)$ is $\mathcal{O}\p{\tau\p{\brac{\va{\alpha}} \log \brac{\va{\alpha}} + \nu^2 \frac{\log \brac{\va{\alpha}}}{\brac{\va{\alpha}}}}}$, uniformly in $\nu$ and $\tau$ and with symbolic estimates.

Consequently, we have in this second case
\begin{equation*}
\begin{split}
\re \tilde{p}(\alpha; i \nu) & = \re \tilde{p}(x,\xi;i \nu) + \mathcal{O}\p{\tau\p{\brac{\va{\alpha}} \log \brac{\va{\alpha}} + \nu^2 \frac{\log \brac{\va{\alpha}}}{\brac{\va{\alpha}}}}} \\
    & = \re p\p{x,\xi + i \psi'(x_1) \mathrm{d}x_1;i \nu}  + \mathcal{O}\p{\tau \nu^2 \brac{\va{\alpha}} \log \brac{\va{\alpha}}} \\
    & = p_2(x,\xi) - \nu^2 p_0(x) +  \mathcal{O}\p{\nu + \tau \nu^2 \brac{\va{\alpha}} \log \brac{\va{\alpha}}},
\end{split}
\end{equation*}
uniformly in $\tau$ and $\nu$. We start by taking $\nu$ large enough so that $p_2(x,\xi) - \nu^2 p_0(x) +  \mathcal{O}\p{\nu}$ is larger than $C^{-1} \brac{\va{\alpha}}^2$ (which is possible by our ellipticity assumptions on $p_2$ and $p_0$). Then, by taking $\tau$ small enough, we ensure that $\mathcal{O}\p{\tau \nu^2 \brac{\va{\alpha}} \log \brac{\va{\alpha}}}$ is smaller than $C^{-1} \brac{\va{\alpha}}^2$, which gives the required estimate. Let us point out here that how small $\tau$ needs to be depend on $\nu$, but how large $\nu$ has to be does not depend on $\tau$.

\noindent \textbf{Third case:} $(x,\xi) \in \set{- 5\epsilon_0/6 \leq x_1 \leq 2\epsilon_0/3} \cap \set{\va{\xi_1} \leq 2\epsilon_1^{-1} (1+\va{\xi'})}$.

As in the previous case, we notice that
\begin{equation*}
\tilde{p}(\alpha;i \nu) = \tilde{p}(x,\xi;i \nu) + \mathcal{O}\p{\tau \nu^2 \brac{\va{\alpha}}\log \brac{\va{\alpha}}}.
\end{equation*}
Then, we use \eqref{eq:ptilde} to find that
\begin{equation*}
\re \tilde{p}(x,\xi;i\nu) \geq - C_0 \epsilon_0 \xi_1^2 + q_1(x_1,x',\xi') + \epsilon \nu^2 - C(1 + \nu),
\end{equation*}
for some $C>0$ that does not depend on $\nu$, nor $\tau$. Using \eqref{eq:epsilon0_c2}, we find that if $\nu$ is large enough we have
\begin{equation*}
\re \tilde{p}(x,\xi;i\nu) \geq C^{-1}\p{1+\va{\xi}^2}.
\end{equation*}
Consequently, we have
\begin{equation*}
\re \tilde{p}(\alpha;i \nu)  \geq C^{-1} \p{1+\va{\xi}^2} - C \tau \nu^2 \brac{\va{\alpha}} \log \brac{\va{\alpha}},
\end{equation*}
where the constant $C > 0$ may have changed, but still does not depend on $\nu$, nor $\tau$. Taking $\tau$ small enough (depending on $\nu$), we get rid of the term $ - C \tau \nu^2 \brac{\va{\alpha}} \log \brac{\va{\alpha}}$. Thus, we get the required estimate. As above, it is crucial here that how small $\tau$ needs to be depend on $\nu$, but how large $\nu$ has to be does not depend on $\tau$.

\noindent \textbf{Fourth case:} $(x,\xi) \in V_+$.

Writing $\tilde{p}(\cdot; i \nu) = \tilde{p}_2 + i \nu \tilde{p}_1 - \nu^2 \tilde{p}_0$, we find that
\begin{equation*}
\im \tilde{p}(\alpha;i \nu) = \im \tilde{p}(x,\xi;i \nu) + \tau H_{G_0}^{\omega_I} \im \tilde{p}(x,\xi;i \nu) + \mathcal{O}\p{\tau^2 \nu^2 (\log \brac{\va{\alpha}})^2}.
\end{equation*}
Then, we notice that on a neighbourhood of the support of $G_0$, the symbol $\tilde{p}(\cdot;i \nu) - p_2$ is the sum of a symbol of order $1$, a symbol of order $1$ multiplied by $\nu$ and a symbol of order $0$ multiplided by $\nu^2$. Consequently, we have
\begin{equation*}
H_{G_0}^{\omega_I} \im \tilde{p}(x,\xi;i \nu) = H_{G_0}^{\omega_I} \im p_2 (x,\xi) + \mathcal{O}\p{\nu^2 \log \brac{\va{\alpha}}}.
\end{equation*}
Using \eqref{eq:escape} as in the proof of Lemma \ref{lemma:high_frequencies}, we see that $H_{G_0}^{\omega_I} \im p_2 (x,\xi)$ is non-positive. We recall \eqref{eq:ptilde} and the fact that $w(x_1) \psi'(x_1)$ is non-positive, to find, for some $C > 0$ that does not depend on $\tau$ nor $\alpha$, that
\begin{equation}\label{eq:estimee_touffue}
\im \tilde{p}(\alpha;i \nu) \leq \nu p_1(x_1,\xi_1) + C \tau \nu^2 (\log \brac{\va{\alpha}})^2.
\end{equation}
Since $(x,\xi) \in V_+$, we know that $\xi_1$ is non-zero. Moreover, $p_1(x_1,\xi_1)$ is negative and elliptic. Thus, we only need to take $\tau$ small enough to get rid of the last term in \eqref{eq:estimee_touffue} and the required estimate follows. Once again here, see that $\tau$ depends on $\nu$, but $\nu$ does not depend on $\tau$.

\noindent \textbf{Fifth case:} $(x,\xi) \in V_-$.

This is the same as the fourth case up to a few sign flips.
\end{proof}

\subsection{Invertibility and Fredholm properties}\label{subsection:invertibility_and_Fredholm}

With the estimates from \S \ref{subsec:ellipticity_estimates}, we are now ready to study the functional analytic properties of $\widetilde{P}_h(\omega)$ acting on suitable spaces.

Let $\nu$ be large enough and $\tau$ be small enough so that Lemma \ref{lemma:everywhere} and Proposition \ref{proposition:multiplication_formula} hold. Let then $\kappa$ be as in Lemma \ref{lemma:high_frequencies}. Let $\delta \in ]0,\kappa[$ and $V$ be a relatively compact open subset of $\set{ z \in \mathbb{C}: \im z > - \kappa}$. Without loss of generality, we may assume that $V$ is connected and contains the compact set $\set{ z \in \mathbb{C}: \va{\re z} \leq \nu + \kappa, - \delta \leq \im z \leq 2 \nu + \kappa}$. Let then $C_K$ be the constant from Lemma \ref{lemma:high_frequencies} applied with $K = \overline{V}$. We shall always assume that $h$ is small enough so that Proposition \ref{proposition:multiplication_formula} holds. Let $k$ be any real number.

Let $a$ be a compactly supported smooth function from $\Lambda$ to $\mathbb{R}_+$ such that

\begin{equation}\label{eq:condition_a}
\inf_{\omega \in K} \inf_{\substack{\alpha \in \Lambda \\ \brac{\va{\alpha}} \leq 2 C_K }} a(\alpha) + \re \tilde{p}(\alpha;\omega) > 0.
\end{equation}
We let then $A$ be the operator
\begin{equation}\label{eq:defA}
A \coloneqq S_\Lambda B_\Lambda a B_\Lambda T_\Lambda,
\end{equation}
where we recall that $S_\Lambda$ is a left inverse for $T_\Lambda$, and $B_\Lambda$ is the orthogonal projector on $\mathcal{H}_{\Lambda,\FBI}^0$ in $L_0^2(\Lambda)$ (see page \pageref{page:Slambda}). The operator $A : C^\infty(X) \to \mathcal{D}'(X)$ extends to a bounded operator from $\mathcal{H}_\Lambda^{m}$ to $\mathcal{H}_\Lambda^{\ell}$ for every $m,\ell \in \mathbb{R}$, see for instance \cite[Proposition 2.4 and Remark 2.20]{BJ20}.

Let us define the domain of $\widetilde{P}_h(\omega)$ on $\mathcal{H}_\Lambda^k$ as
\begin{equation*}
D_k = \set{ u \in \mathcal{H}_\Lambda^k : \widetilde{P}_h(0)u \in \mathcal{H}_\Lambda^{k-1}}.
\end{equation*}
We put a Hilbert space structure on $D_k$ by endowing it with the norm
\begin{equation*}
\n{u}_{D_k}^2 = \n{u}_{\mathcal{H}_\Lambda^k}^2 + \|\widetilde{P}_h(0)u\|_{\mathcal{H}_\Lambda^{k-1}}^2.
\end{equation*}

We will need the following approximation result.

\begin{lemma}\label{lemma:approximation}
Let $u \in D_k$. Then $\widetilde{P}_h(\omega)u$ belongs to $\mathcal{H}_\Lambda^{k-1}$ and there is a sequence $(u_n)_{n \in \mathbb{N}}$ of elements of $\mathcal{H}_\Lambda^{\infty}$ such that $(u_n)_{n \in \mathbb{N}}$ tends to $u$ in $\mathcal{H}_\Lambda^k$ and $(\widetilde{P}_h(\omega)u_n)_{n \in \mathbb{N}}$ converges to $\widetilde{P}_h(\omega)u$ in $\mathcal{H}_\Lambda^{k-1}$.
\end{lemma}

\begin{proof}
Start by noticing that
\begin{equation*}
\widetilde{P}_h(\omega)u = (\widetilde{P}_h(\omega) - \widetilde{P}_h(0))u + \widetilde{P}_h(0)u.
\end{equation*}
Since $\widetilde{P}_h(\omega) - \widetilde{P}_h(0)$ is bounded from $\mathcal{H}_\Lambda^k$ to $\mathcal{H}_\Lambda^{k-1}$, we see that $\widetilde{P}_h(\omega)u$ belongs to $\mathcal{H}_\Lambda^{k-1}$ when $u$ belongs to $D_k$.

Let then $I_\epsilon$ be the operator
\begin{equation*}
I_\epsilon = S_\Lambda B_\Lambda s_\epsilon B_\Lambda T_\Lambda,
\end{equation*}
where $s_\epsilon$ is a symbol on $\Lambda$ defined by $s_\epsilon(\alpha) = \theta(\epsilon \brac{\va{\alpha}})$, where $\theta$ is a compactly supported function in $\mathbb{R}$, identically equal to $0$ near $1$. It follows for instance from \cite[Proposition 2.4 and Remark 2.20]{BJ20} that if $u \in \mathcal{H}_\Lambda^k$ then $I_\epsilon u \in \mathcal{H}_\Lambda^\infty$. We see that for $u \in \mathcal{H}_\Lambda^k$, we have
\begin{equation*}
\begin{split}
\n{I_\epsilon u - u}_{\mathcal{H}_\Lambda^k} & = \n{\Pi_\Lambda B_\Lambda s_ \epsilon T_\Lambda u - T_\Lambda u}_{L_k^2(\Lambda)} \\
& = \n{B_\Lambda(s_\epsilon - 1) T_\Lambda u}_{L_k^2(\Lambda)} \leq C\n{(s_\epsilon - 1) T_\Lambda u}_{L_k^2(\Lambda)}.
\end{split}
\end{equation*}
It follows that $I_\epsilon u$ converges to $u$ in $\mathcal{H}_\Lambda^k$ when $\epsilon$ tends to $0$.

If $u$ belongs to $D_k$, we see that
\begin{equation*}
\widetilde{P}_h(\omega) I_\epsilon u = I_\epsilon \widetilde{P}_h(\omega) u + \left[\widetilde{P}_h(\omega), I_\epsilon\right] u.
\end{equation*}
From the analysis above, we have that $I_\epsilon \widetilde{P}_h(\omega) u$ converges to $\widetilde{P}_h(\omega) u$ in $\mathcal{H}_\Lambda^{k-1}$ when $\epsilon$ tends to $0$. Notice that the symbol $s_\epsilon$ is uniformly bounded as a symbol of order $0$ on $\Lambda$. Hence, it follows from \cite[Proposition 2.12]{BJ20}, as in the proof of \cite[Lemma 3.4]{BJ20}, that the operator $[\widetilde{P}_h(\omega), I_\epsilon]$ is uniformly bounded from $\mathcal{H}_\Lambda^k$ to $\mathcal{H}_\Lambda^{k-1}$ when $\epsilon$ tends to $0$. If $u$ is in $\mathcal{H}_\Lambda^\infty$, the analysis above implies that $[\widetilde{P}_h(\omega), I_\epsilon]u$ tends to $0$ in $\mathcal{H}_\Lambda^{k-1}$. Thanks to the uniform boundedness of $[\widetilde{P}_h(\omega), I_\epsilon]$ when $\epsilon$ tends to $0$, we see that the same holds when $u$ is only in $\mathcal{H}_\Lambda^{k-1}$.
\end{proof}

We first use Lemma \ref{lemma:high_frequencies} to find:

\begin{lemma}\label{lemma:estimee_fredholm}
There is $C > 0$ such that for $h$ small enough and every $\omega \in V$ and $u \in D_k$ we have
\begin{equation*}
\n{u}_{\mathcal{H}_\Lambda^k} \leq C \n{\p{\widetilde{P}_h(\omega) +A} u}_{\mathcal{H}_\Lambda^{k-1}}.
\end{equation*}
\end{lemma}

\begin{proof}
Thanks to Lemma \ref{lemma:approximation}, we only need to prove this estimate for $u \in \mathcal{H}_\Lambda^\infty$. Let $f_+,f_-,f_R$ and $f_a$ be symbols of order $0$ on $\Lambda$ such that $f_+ + f_-+f_R + f_a = 1$. Moreover, we assume that $f_+,f_-$ and $f_R$ are supported in the intersection of $\set{\brac{\va{\alpha}} \geq C_K}$ respectively with $W_+,W_-$ and $W_R$ and that $f_a$ is supported in $\set{\brac{\va{\alpha}} \leq 2 C_K}$.

For $u \in \mathcal{H}_\Lambda^\infty$ and $\omega \in K$, we have from Proposition \ref{proposition:multiplication_formula}:
\begin{equation*}
\begin{split}
& \re\p{\int_\Lambda f_R(\alpha) \brac{\va{\alpha}}^{2k-1} T_\Lambda (\widetilde{P}_h(\omega) +A) u \overline{T_\Lambda u} \mathrm{d}\alpha} \\ & \qquad \qquad \geq \int_{\Lambda} f_R(\alpha) \brac{\va{\alpha}}^{2k-1} \re\p{\tilde{p}(\alpha;\omega) +a(\alpha)} \va{T_\Lambda u(\alpha)}^2 \mathrm{d}\alpha - C h \n{u}_{\mathcal{H}_\Lambda^k}^2 \\
   & \qquad \qquad \geq C^{-1} \int_\Lambda f_R(\alpha) \brac{\va{\alpha}}^{2k} \va{T_\Lambda u(\alpha)}^2 \mathrm{d}\alpha - C h \n{u}_{\mathcal{H}_\Lambda^k}^2,
\end{split}
\end{equation*}
where we used Lemma \ref{lemma:high_frequencies} in the last line (since $a$ takes positive values it does not harm the positivity of the real part of $\tilde{p}$). From Cauchy--Schwarz inequality, we find that
\begin{equation*}
\begin{split}
& \int_\Lambda f_R(\alpha) \brac{\va{\alpha}}^{2k} \va{T_\Lambda u(\alpha)}^2 \mathrm{d}\alpha \\
     & \qquad \qquad \leq C \n{(\widetilde{P}_h(\omega) +A) u}_{\mathcal{H}_\Lambda^{k-1}} \n{u}_{\mathcal{H}_\Lambda^k} + C h \n{u}_{\mathcal{H}_\Lambda^k}^2.
\end{split}
\end{equation*}
Replacing the real part by an imaginary part, and varying the sign, we get the same estimates with $f_R$ replaced by $f_+$ and $f_-$. Using \eqref{eq:condition_a}, we get the same estimates with $f_R$ replaced by $f_a$. Summing these four estimates, we find that
\begin{equation*}
\n{u}_{\mathcal{H}_\Lambda^k}^2 \leq C \n{(\widetilde{P}_h(\omega) + A)u}_{\mathcal{H}_\Lambda^{k-1}}\n{u}_{\mathcal{H}_\Lambda^k} + C h \n{u}_{\mathcal{H}_\Lambda^k}^2.
\end{equation*}
When $h$ is small enough, we can get rid of the second term in the right hand side. Dividing by $\n{u}_{\mathcal{H}_\Lambda^k}$ the result follows (the result is trivial when $u = 0$).
\end{proof}

The same proof using Lemma \ref{lemma:everywhere} instead of Lemma \ref{lemma:high_frequencies} gives:

\begin{lemma}\label{lemma:estimee_invertibility}
There is $C > 0$ such that for $h$ small enough and every $u \in D_k$ we have
\begin{equation*}
\n{u}_{\mathcal{H}_\Lambda^k} \leq C \n{\widetilde{P}_h(i \nu) u}_{\mathcal{H}_\Lambda^{k-1}}.
\end{equation*}
\end{lemma}

Applying Proposition \ref{proposition:toeplitz} as in the justification of Proposition \ref{proposition:multiplication_formula}, we find that, for every $\omega \in V$, there is a symbol $\sigma_\omega$ of order $2$ on $\Lambda$ and an operator $Z$ with negligible kernel on $\Lambda \times \Lambda$ such that
\begin{equation*}
B_\Lambda T_\Lambda \widetilde{P}_h(\omega)S_\Lambda B_\Lambda = B_\Lambda \sigma_\omega B_\Lambda + Z = B_\Lambda \sigma_\omega B_\Lambda + B_\Lambda Z B_\Lambda.
\end{equation*}
Let us identify the dual of $\mathcal{H}_\Lambda^k$ with $\mathcal{H}_\Lambda^{-k}$ as in \cite[Lemma 2.24]{BJ20}, that is using the pairing
\begin{equation}\label{eq:other_pairing}
\brac{u,v}_\Lambda \coloneqq \int_\Lambda T_\Lambda u \overline{T_\Lambda v} e^{ - \frac{2H}{h}} \mathrm{d}\alpha.
\end{equation}
Notice that it is \emph{a priori} not the $L^2$ pairing (recall that the $L^2$ pairing identifies the dual of $\mathcal{H}_\Lambda^k$ with $\mathcal{H}_{\overline{\Lambda}}^{-k}$, see \S \ref{subsection:duality_statement}). Under this identification, the formal adjoint of $\widetilde{P}_h(\omega)$ may be defined as
\begin{equation*}
\widetilde{P}_h(\omega)^* = S_\Lambda \p{B_\Lambda \bar{\sigma}_\omega B_\Lambda + B_\Lambda Z^* B_\Lambda} T_\Lambda
\end{equation*}
By this, we just mean that if $u,v \in \mathcal{H}_{\Lambda}^{ \infty}$ then
\begin{equation*}
\brac{\widetilde{P}_h(\omega)u,v}_\Lambda = \brac{u,\widetilde{P}_h(\omega)^* v}_\Lambda.
\end{equation*}
Notice that we do not claim that $\widetilde{P}_h(\omega)^*$ is the adjoint of $\widetilde{P}_h(\omega)$ for a Hilbert space structure. We define the domain of $\widetilde{P}_h(\omega)^* $ as
\begin{equation*}
D_{-k}^* = \set{ u \in \mathcal{H}_\Lambda^{-k+1} : \widetilde{P}_h(0)^* u \in \mathcal{H}_\Lambda^{-k}}.
\end{equation*}

Notice that we have $\bar{\sigma}_\omega(\alpha) = \overline{\tilde{p}(\alpha;\omega)} + \mathcal{O}\p{h \brac{\va{\alpha}}}$. Hence, the operator $\widetilde{P}_h(\omega)^* $ satisfies Proposition \ref{proposition:multiplication_formula} with $\tilde{p}$ replaced by $\bar{\tilde{p}}$. In order to introduce the symbol $f$, one may use \cite[Proposition 2.12]{BJ20}. Consequently, we can use Lemmas \ref{lemma:high_frequencies} and \ref{lemma:everywhere}, as in the proofs of Lemmas \ref{lemma:estimee_fredholm} and \ref{lemma:estimee_invertibility}, to get:

\begin{lemma}\label{lemma:estimee_fredholm_adjoint}
There is $C > 0$ such that for $h$ small enough and every $\omega \in V$ and $u \in D_{-k}^*$ we have
\begin{equation*}
\n{u}_{\mathcal{H}_\Lambda^{-k+1}} \leq C \n{\p{\widetilde{P}_h(\omega) +A}^* u}_{\mathcal{H}_\Lambda^{-k}}.
\end{equation*}
\end{lemma}

In this statement, $\p{\widetilde{P}_h(\omega) +A}^* = \widetilde{P}_h(\omega)^* +A$ is the formal adjoint of $\widetilde{P}_h(\omega) +A$ for the pairing \eqref{eq:other_pairing}.

\begin{lemma}\label{lemma:estimee_invertibility_adjoint}
There is $C > 0$ such that for $h$ small enough and every $u \in D_{-k}^*$ we have
\begin{equation*}
\n{u}_{\mathcal{H}_\Lambda^{-k+1}} \leq C \n{\widetilde{P}_h(i \nu)^* u}_{\mathcal{H}_\Lambda^{-k}}.
\end{equation*}
\end{lemma}

\begin{remark}
Here, we used \eqref{eq:other_pairing} rather than the $L^2$ pairing to describe the dual of $\mathcal{H}_\Lambda^k$ because this identification makes $A$ self-adjoint, so that we can reuse directly the estimates from Lemmas \ref{lemma:high_frequencies} and \ref{lemma:everywhere}. We expect however that the $L^2$ pairing studied in \S \ref{subsection:duality_statement} would allow to get similar estimates that we could also use in the proofs below.
\end{remark}

From Lemmas \ref{lemma:estimee_fredholm}, \ref{lemma:estimee_invertibility}, \ref{lemma:estimee_fredholm_adjoint} and \ref{lemma:estimee_invertibility_adjoint}, we deduce:

\begin{proposition}\label{proposition:inverse}
There is $C > 0$ such that for $h$ small enough and $\omega \in V$ the operators $\widetilde{P}_h(\omega) +A$ and $\widetilde{P}_h(i \nu)$ are invertible as operators from $D_k$ to $\mathcal{H}_\Lambda^{k-1}$. Moreover, the operator norms of their inverses is bounded by $C$.
\end{proposition}

\begin{proof}
From Lemma \ref{lemma:estimee_fredholm}, we find that $\widetilde{P}_h(\omega) +A$ is injective on $D_k$ and that its image is closed in $\mathcal{H}_\Lambda^{k-1}$.

Let us consider and element $v \in \mathcal{H}_\Lambda^{-k+1}$ such that $\brac{u,v}_\Lambda = 0$ for every $u \in \mathcal{H}_\Lambda^{k-1}$ in the image of $\widetilde{P}_h(\omega) +A$. In particular, if $u \in \mathcal{H}_\Lambda^\infty$, we have $\brac{(\widetilde{P}_h(\omega) +A)u,v} = 0$. Notice that $\mathcal{H}_\Lambda^\infty$ is dense in $\mathcal{H}_\Lambda^{-k+1}$ (for instance because it contains all real-analytic functions due to Proposition \ref{proposition:duality}, and they form a dense subset of $\mathcal{H}_\Lambda^{-k+1}$ according to \cite[Corollary 2.3]{BJ20}, one can also work as in Lemma \ref{lemma:approximation}). Consequently, we have
\begin{equation*}
\brac{(\widetilde{P}_h(\omega) +A)u,v}_\Lambda = \brac{u, (\widetilde{P}_h(\omega) +A)^* v}_\Lambda
\end{equation*}
for every $u \in \mathcal{H}_\Lambda^\infty$, since this equality holds when $v \in \mathcal{H}_\Lambda^\infty$. Hence, we have $\brac{u, (\widetilde{P}_h(\omega) +A)^* v} = 0$ for every $u \in \mathcal{H}_\Lambda^\infty$, and thus $(\widetilde{P}_h(\omega) +A)^* v = 0$. It follows from Lemma \ref{lemma:estimee_fredholm_adjoint} that $v = 0$. 

We just proved that the image of $\widetilde{P}_h(\omega) +A$ is dense in $\mathcal{H}_\Lambda^{k-1}$, and thus $\widetilde{P}_h(\omega) +A$ is invertible. The estimate on the operator norm of the inverse immediately follows from Lemma \ref{lemma:estimee_fredholm}.

The argument to invert $\widetilde{P}_h(i \nu)$  is the same using Lemmas \ref{lemma:estimee_invertibility} and \ref{lemma:estimee_invertibility_adjoint} instead of Lemmas \ref{lemma:estimee_fredholm} and \ref{lemma:estimee_fredholm_adjoint}.
\end{proof}

The analytic Fredholm theory then implies that:

\begin{proposition}\label{proposition:meromorphic_inverse}
Assume that $h$ is small enough. For every $\omega \in V$, the operator $\widetilde{P}_h(\omega) : D_k \to \mathcal{H}_\Lambda^{k-1}$ is Fredholm of index $0$. Moreover, the operator $\widetilde{P}_h(\omega): D_k \to \mathcal{H}_\Lambda^{k-1}$ has a meromorphic inverse $\omega \mapsto \widetilde{P}_h(\omega)^{-1}$ with poles of finite rank on $V$.
\end{proposition}

\begin{proof}
From \cite[Proposition 21.3]{BJ20} or Lemma \ref{lemma:trace_class} below, we find that $A$ is a compact operator from $D_k$ to $\mathcal{H}_\Lambda^{k-1}$. Hence, it follows from Proposition \ref{proposition:inverse} that $\widetilde{P}_h(\omega) : D_k \to \mathcal{H}_\Lambda^{k-1}$ is Fredholm for $\omega \in V$.

Since $\widetilde{P}_h(\omega) - \widetilde{P}_h(0)$ is a holomorphic family of bounded operators from $\mathcal{H}_\Lambda^k$ to $\mathcal{H}_\Lambda^{k-1}$, we see that $\widetilde{P}_h(\omega)$ is a holomorphic family of operators from $D_k$ to $\mathcal{H}_\Lambda^{k-1}$, for $\omega$ in $V$. Since this operator is invertible for $\omega = i \nu$ and $V$ is connected, we find that the index of $\widetilde{P}_h(\omega)$ is $0$. Finally, the analytic Fredholm theorem \cite[Theorem C.8]{dyatlov_zworski_book} implies the existence of the meromorphic inverse $\omega \mapsto \widetilde{P}_h(\omega)^{-1}$, with poles of finite rank. 
\end{proof}

\subsection{Counting resonances}\label{subsec:counting_resonances}

We will now use the functional analytic framework from \S \ref{subsection:invertibility_and_Fredholm} to prove the point \ref{item:counting_resonances} in Proposition \ref{proposition:general_statement}. The bounds on the number of resonances from Theorems \ref{theorem:main} and \ref{theorem:schwarzschild} ultimately come from the following lemma.

\begin{lemma}\label{lemma:bound_extended_resonances}
Recall that $\delta \in ]0,\kappa[$. There is $C> 0$ such that, for every $h$ small enough, the number of $\omega$'s in the disk of center $0$ and radius $\delta$ such that $\widetilde{P}_h(\omega) : D_k \to \mathcal{H}_\Lambda^{k-1}$ is not invertible (counted with null multiplicity) is less than $C h^{-n}$.
\end{lemma}

Before being able to prove Lemma \ref{lemma:bound_extended_resonances}, we need to establish a bound on the trace class opertor norm of $A$, which is defined by \eqref{eq:defA}.

\begin{lemma}\label{lemma:trace_class}
The operator $A : D_k \to \mathcal{H}_\Lambda^{k-1}$ is trace class, with trace class norm $\mathcal{O}(h^{-n})$.
\end{lemma}

\begin{proof}
We only need to prove that the operator $\widetilde{A} = B_\Lambda a B_\Lambda$ is trace class from $L_k^2(\Lambda)$ to $L_{k-1}^2(\Lambda)$, with trace class norm $\mathcal{O}(h^{-n})$.

For every $N > 0$, introduce the operator $\Box_N \coloneqq B_\Lambda \brac{\va{\alpha}}^N B_\Lambda$. Using \cite[Proposition 2.12]{BJ20} to make a parametrix construction, we see that there is a symbol $\sigma_N$ of order $-N$ on $\Lambda$ such that $\Box_N  B_\Lambda \sigma_N B_\Lambda - B_\Lambda$ and $ B_\Lambda \sigma_N B_\Lambda \Box_N - B_\Lambda$ are negligible operators, in particular they are $\mathcal{O}(h^\infty)$ as operators from $L_{s_1}^2(\Lambda) \to L_{s_2}^2(\Lambda)$ for any $s_1,s_2 \in \mathbb{R}$. Hence, for $h$ small enough, we get an inverse $\Box_N^{-1} : \mathcal{H}_{\Lambda,\FBI}^0 \to \mathcal{H}_{\Lambda,\FBI}^N$ for $\Box_N$, which is bounded uniformly in $h$ and satisfies the equation
\begin{equation*}
\begin{split}
\Box_N^{-1} B_\Lambda = B_\Lambda \sigma_N B_\Lambda & + B_\Lambda \sigma_N B_\Lambda(B_\Lambda - \Box_N B_\Lambda \sigma_N B_\Lambda) \\ & \qquad \qquad + (B_\Lambda - B_\Lambda \sigma_N B_\Lambda \Box_N) \Box_N^{-1} (B_\Lambda - \Box_N B_\Lambda \sigma_N B_\Lambda).
\end{split}
\end{equation*}
Thus, we see that $\Box_N^{-1} B_\Lambda$ is equal to $B_\Lambda \sigma_N B_\Lambda$ up to a negligible operator. Let us recall that $\mathcal{H}_{\Lambda,\FBI}^k$ is the image of $\mathcal{H}_\Lambda^k$ by $T_\Lambda$ (which is also the image of $L_k^2(\Lambda)$ by $B_\Lambda$).

Fix $N > n$. Notice that $\widetilde{A}$ is bounded, uniformly in $h$, as an operator from $L_k^2(\Lambda)$ to $L_{k+N}^2(\Lambda)$ (since $B_\Lambda$ is bounded on $L_k^2(\Lambda)$ and on $L_{k+N}^2(\Lambda)$). We can then write:
\begin{equation}\label{eq:A_produit}
\widetilde{A} = \iota \Box_{k}^{-1} \Box_N^{-1} B_\Lambda \Box_N \Box_k \widetilde{A}.
\end{equation}
On the left hand side, $\widetilde{A}$ is seen as an operator from $L_k^2(\Lambda)$ to $L_{k-1}^2(\Lambda)$. On the right hand side, $\widetilde{A}$ sends $L_k^2(\Lambda)$ into $\mathcal{H}_{\Lambda,\FBI}^{k+N}$, the operator $\Box_k$ sends $\mathcal{H}_{\Lambda,\FBI}^{k+N}$ into $\mathcal{H}_{\Lambda,\FBI}^N$, the operator $\Box_N$ sends $\mathcal{H}_{\Lambda,\FBI}^{N}$ into $\mathcal{H}_{\Lambda,\FBI}^0$, the operator $\Box_N^{-1}B_\Lambda$ sends $\mathcal{H}_{\Lambda,\FBI}^{0}$ into $\mathcal{H}_{\Lambda,\FBI}^0$, the operator $\Box_k^{-1}$ sends $\mathcal{H}_{\Lambda,\FBI}^{0}$ into $\mathcal{H}_{\Lambda,\FBI}^{k}$ and $\iota$ is the inclusion of $\mathcal{H}_{\Lambda,\FBI}^{k}$ into $L_{k-1}^2(\Lambda)$. With these mapping properties, the operators $\widetilde{A},\Box_k, \Box_N,\Box_k^{-1}$ and $\iota$ on the right hand side of \eqref{eq:A_produit} are bounded uniformly in $h$. From \cite[Lemma 2.25]{BJ20}, we see that $\Box_N^{-1}B_\Lambda$ is trace class on $L_0^2(\Lambda)$ (since $B_\Lambda \sigma_N B_\Lambda$ is). Moreover, its trace is given by the integral of its kernel on the diagonal, which is $\mathcal{O}(h^{-n})$. Indeed, $\Box_N^{-1}B_\Lambda$ is a ``complex FIO associated to $\Delta_\Lambda$ of order $-N$'' in the sense of \cite[Definition 2.5]{BJ20} as a consequence of \cite[Lemmas 2.16 and 2.23]{BJ20}. Since $\mathcal{H}_{\Lambda,\FBI}^0$ is a closed subset of $L_0^2(\Lambda)$, we see that $\Box_N^{-1}B_\Lambda$ is also a trace class operator from $\mathcal{H}_{\Lambda,\FBI}^0$ to itself, with the same trace. Moreover, $\Box_N^{-1}B_\Lambda$ is a positive self adjoint operator on $\mathcal{H}_{\Lambda,\FBI}^0$ with $h$ small enough (because $\brac{\va{\alpha}}^N$ is positive), so that its trace class norm coincides with its trace. This ends the proof of the lemma.
\end{proof}

We are now ready to prove Lemma \ref{lemma:bound_extended_resonances}.

\begin{proof}[Proof of Lemma \ref{lemma:bound_extended_resonances}]
For $\omega \in V$, let us introduce the spectral determinant
\begin{equation*}
f_h(\omega) = \det\p{I - (\widetilde{P}_h(\omega)+A)^{-1} A}.
\end{equation*}
Since $\widetilde{P}_h(\omega) - \widetilde{P}_h(0)$ is a holomorphic family of bounded operators from $\mathcal{H}_\Lambda^k$ to $\mathcal{H}_\Lambda^{k-1}$, we see that $\widetilde{P}_h(\omega) + A$ is a holomorphic family of operators from $D_k$ to $\mathcal{H}_\Lambda^{k-1}$. From Proposition \ref{proposition:inverse}, the operators $(\widetilde{P}_h(\omega) + A)^{-1} : \mathcal{H}_\Lambda^{k-1} \to D_k$ are bounded uniformly in $\omega \in V$, and thus it is a holomorphic family of operators in $V$. Consequently, the spectral determinant $f_h(\omega)$ is holomorphic in $V$.

The logarithmic derivative of $f_h$ is given by
\begin{equation*}
\begin{split}
\frac{f_h'(\omega)}{f_h(\omega)} & = \tr\Big(\p{I - \p{\widetilde{P}_h(\omega) + A}^{-1} A}^{-1} \p{\widetilde{P}_h(\omega) + A}^{-1} \\ & \qquad \qquad \qquad \qquad \qquad \qquad \qquad \qquad\times \partial_\omega \widetilde{P}_h(\omega)\p{\widetilde{P}_h(\omega) + A}^{-1} A \Big) \\
   & = \tr\Big(\p{\widetilde{P}_h(\omega) + A}^{-1} A \p{I - \p{\widetilde{P}_h(\omega)  + A}^{-1} A}^{-1} \\ & \qquad \qquad \qquad \qquad \qquad \qquad \qquad \qquad\times \p{\widetilde{P}_h(\omega) + A}^{-1}\partial_\omega \widetilde{P}_h(\omega)\Big).
\end{split}
\end{equation*}
Let us then write
\begin{equation}\label{eq:manipulation_resolvent}
\begin{split}
& \p{\widetilde{P}_h(\omega) + A}^{-1} A \p{I - \p{\widetilde{P}_h(\omega) + A}^{-1} A}^{-1} \\ & \qquad \qquad \qquad \qquad \qquad \qquad \qquad \qquad \qquad \times \p{\widetilde{P}_h(\omega) + A}^{-1}\partial_\omega \widetilde{P}_h(\omega) \\
    & = \p{\p{I - \p{\widetilde{P}_h(\omega) + A}^{-1} A}^{-1} - I}\p{\widetilde{P}_h(\omega) + A}^{-1}\partial_\omega \widetilde{P}_h(\omega) \\
    & = \widetilde{P}_h(\omega)^{-1} \partial_\omega \widetilde{P}_h(\omega) - \p{\widetilde{P}_h(\omega) + A}^{-1}\partial_\omega \widetilde{P}_h(\omega)
\end{split}
\end{equation}
Hence, if $\omega_0$ is in $V$, the residue of the family of operators \eqref{eq:manipulation_resolvent} at $\omega_0$ is the same as the residue of the family of operators $\omega \mapsto \widetilde{P}_h(\omega)^{-1} \partial_\omega \widetilde{P}_h(\omega)$. Consequently, the order of annulation of $f_h$ at $\omega_0$ coincides with the null multiplicity of $\omega \mapsto \widetilde{P}_h(\omega)$ at $\omega_0$.

Since $V$ is open, there is $\eta > 0$ such that the closed disk of center $i \nu$ and radius $\nu + \delta + 2 \eta$ is contained in $V$. Since the poles of $\widetilde{P}_h(\omega)^{-1}$ are isolated, we may choose $0 \leq  \eta' \leq \eta$ such that there is no poles of $\widetilde{P}_h(\omega)^{-1}$ on the circle of center $i \nu$ and radius $\nu + \delta+ \eta + \eta'$. For $r \geq 0$, let $n_h(r)$ denote the number of zeros of $f_h$ in the disk of center $i \nu$ and radius $r$. Notice that
\begin{equation}\label{eq:estimee_simple}
n_h(\nu + \delta) \leq \frac{\nu + \delta + \eta}{\eta} \int_{\nu + \delta}^{\nu + \delta+ \eta} \frac{n_h(r)}{r}\mathrm{d}r \leq \frac{\nu + \delta + \eta}{\eta} \int_{0}^{\nu + \delta + \eta+ \eta'} \frac{n_h(r)}{r}\mathrm{d}r.
\end{equation}
From Jensen's formula, we know that
\begin{equation}\label{eq:Jensen}
\int_{0}^{\nu + \delta + \eta+ \eta'} \frac{n_h(r)}{r}\mathrm{d}r \leq - \log \va{f_h(i \nu)} + \sup_{\va{\omega - i \nu} =  \nu + \delta + \eta + \eta'} \log \va{f_h(\omega)}.
\end{equation}

From Proposition \ref{proposition:inverse} and Lemma \ref{lemma:trace_class}, we know that the trace class norm of the operator $(\widetilde{P}_h(\omega) + A)^{-1} A$ is $\mathcal{O}(h^{-n})$ uniformly in $h$ and in $\omega$ on the circle of center $i \nu$ and radius $\nu + \delta + \eta+ \eta'$. Then, from \cite[Theorem IV.5.2]{book_trace_determinant}, we find that
\begin{equation}\label{eq:estimate_circle}
\sup_{\va{\omega - i \nu} = \nu + \delta + \eta+ \eta'} \log \va{f_h(\omega)} \leq C h^{-n},
\end{equation}
for some $C > 0$ and $h$ small enough. In order to estimate $\va{f_h(i \nu)}$ from below, let us write
\begin{equation*}
\begin{split}
& \p{I - (\widetilde{P}_h(i \nu) + A)^{-1}A}^{-1} = I + \widetilde{P}_h(i \nu)^{-1}A.
\end{split}
\end{equation*}
From Proposition \ref{proposition:inverse} and Lemma \ref{lemma:trace_class}, we see that the trace class operator norm of $(I - (\widetilde{P}_h(i \nu) + A)^{-1})^{-1} - I$ is $\mathcal{O}(h^{-n})$. Since
\begin{equation*}
f_h(i \nu)^{-1} = \det\p{\p{I - (\widetilde{P}_h(i \nu) + A)^{-1}}^{-1}},
\end{equation*}
we find using \cite[Theorem IV.5.2]{book_trace_determinant} again that 
\begin{equation}\label{eq:estimate_point}
- \log \va{f_h(i \nu)} \leq C h^{-n}
\end{equation}
for some $C > 0$ and $h$ small enough. From \eqref{eq:estimee_simple},\eqref{eq:Jensen}, \eqref{eq:estimate_circle} and \eqref{eq:estimate_point}, we find that $n_h(\nu +\delta)$ is $\mathcal{O}(h^{-n})$. The result follows since the disk of center $0$ and radius $\delta$ is contained in the disk of center $i \nu$ and radius $\nu + \delta$.
\end{proof}

\subsection{Summary}\label{subsec:summary}

Let us put together the definitions from \S \ref{subsec:choice_parameters} and \S \ref{subsec:spaces} and the results from \S \ref{subsection:invertibility_and_Fredholm} and \S \ref{subsec:counting_resonances} to check that Proposition \ref{proposition:general_statement} holds.

\begin{proof}[Proof of Proposition \ref{proposition:general_statement}]
We just need to collect facts that we already proved. We recall that the modification $P_h(\omega)$ of $\mathcal{P}_h(\omega)$ is given by \eqref{eq:modification_operators}. Recalling \eqref{eq:deffk}, we let $\mathcal{H}_1 = \set{ u \in \mathcal{D}'(M): e^{\frac{\psi}{h}} u \in D_k} = \set{ u \in \mathcal{F}_k : P_h(0) u \in \mathcal{F}_{k-1}}$ and $\mathcal{H}_2 = \mathcal{F}_{k-1}$ (for any value of $k \in \mathbb{R}$). 

The inclusions $C^\infty(X) \subseteq \mathcal{H}_j \subseteq \mathcal{D}'(X)$ for $j =1,2$ are given by Proposition \ref{proposition:inclusions_naturelles}. The fact that the elements of $\mathcal{H}_j$ are continuous near $\partial Y$ follows from Proposition \ref{proposition:continuity}.

All the properties that we need for $P_h(\omega) : \mathcal{H}_1 \to \mathcal{H}_2$ follow from the same properties for $\widetilde{P}_h(\omega): D_k \to \mathcal{H}_\Lambda^{k-1}$. The holomorphic dependence on $\omega$ follows from the remark after Proposition \ref{proposition:multiplication_formula}. The invertibility for $\omega = i \nu$ with a $\nu > 0$ is given by Proposition \ref{proposition:inverse}. Point \ref{item:fredholm_family} follows from Proposition \ref{proposition:meromorphic_inverse}. 

Finally, the counting bound \ref{item:counting_resonances} is given by Lemma \ref{lemma:bound_extended_resonances}.
\end{proof}

\bibliographystyle{alpha}
\bibliography{biblio.bib}

\end{document}